%% file: main.tex
\documentclass[english, 11pt]{article}
\usepackage{import}

\subimport{preamble/}{preamble_en_fn_custom.tex}

\usepackage[a4paper]{geometry}
\usepackage[
  backend=biber,
  style=alphabetic, 
  sortlocale=EN_GB,
  useprefix=false,
  url=false,
  doi=true,
  eprint=false,
  maxbibnames=99,
]{biblatex}

\begin{document}

\thispagestyle{empty} 
\begin{center}
\textbf{\Large Master thesis} \\
\vspace{1cm}
\Huge Structure-preserving model reduction of Hamiltonian systems by learning a~symplectic autoencoder \\
\vspace{2.1cm}
\begin{Large}	
\textbf{Author:} \\
\medskip
Florian Konrad Josef Niggl \\
\vspace{2.1cm}
Supervisor: Prof. Dr. Tatjana Stykel \\
Second Examiner: Prof. Dr. Daniel Peterseim \\
\vspace{2.1cm}
Faculty of Mathematics, Natural Sciences, and Materials Engineering \\
\vspace{1cm}
Universität Augsburg \\
\end{Large}
\vspace{2.1cm}
\includegraphics[width=0.34\textwidth]{{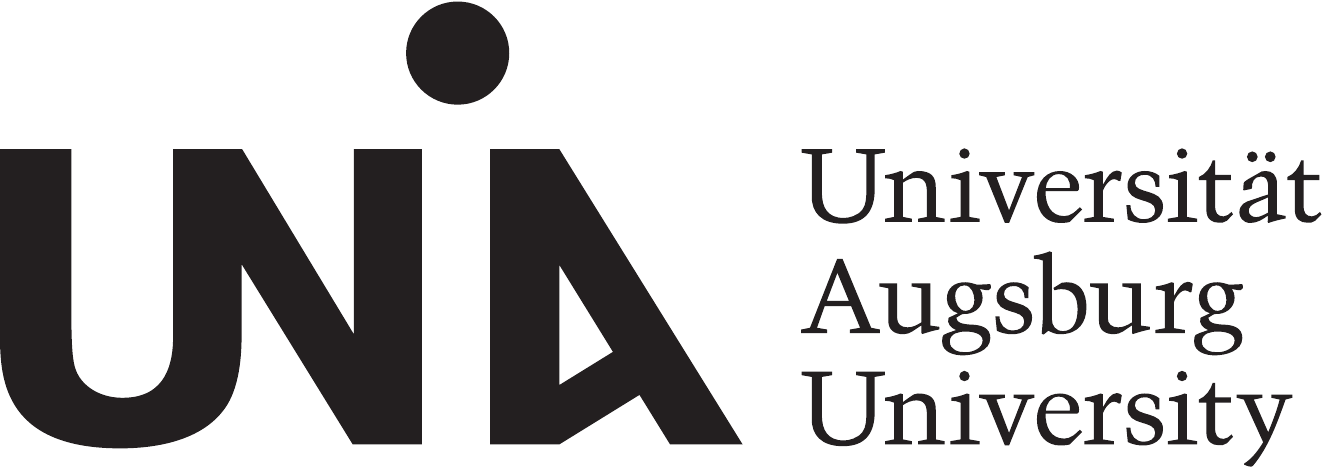}}  
\end{center}

\newpage
\thispagestyle{empty} 

\begin{abstract}
Evolutionary partial differential equations play a~crucial role in many areas of science and engineering. Spatial discretization of these equations leads to a~system of ordinary differential equations which can then be solved by numerical time integration. Such a~system is often of very high dimension, making the simulation very time consuming. One way to reduce the computational cost is to approximate the large system by a~low-dimensional model using a~model reduction approach and to solve the reduced-order model instead of the original one.

This master thesis deals with structure-preserving model reduction of Hamiltonian systems by using machine learning techniques. We discuss a~nonlinear model reduction approach based on the construction of an~encoder-decoder pair that minimizes the approximation error and satisfies symplectic constraints in order to guarantee the preservation of the symplectic structure inherent in Hamiltonian systems. More specifically, we study an~autoencoder network that learns a~symplectic encoder-decoder pair. Symplecticity poses some additional difficulties, as we need to ensure this structure in each network layer. Since these symplectic constraints are described by the (symplectic) Stiefel manifold, we use manifold optimization techniques to ensure the symplecticity of the encoder and decoder. A~particular challenge is to adapt the ADAM optimizer to the manifold structure. We present a~modified ADAM optimizer that works directly on the Stiefel manifold and compare it to the existing approach based on homogeneous spaces. In addition, we propose several modifications to the network and training setup that significantly improve the performance and accuracy of the autoencoder. Finally, we numerically validate the modified optimizer and different learning configurations on two Hamiltonian systems, the 1D wave equation and the sine-Gordon equation, and demonstrate the improved accuracy and computational efficiency of the presented learning algorithms.
\end{abstract}

\vspace{2cm}
\textbf{Accompanying Julia Code:} \\
\indent \http{https://github.com/florianniggl/Master-thesis-implementation}

\newpage
\pagenumbering{roman}   
\tableofcontents   
\newpage
\pagenumbering{arabic} 
\newpage

\newpage
\section{Introduction}

In nature, we find the concept of a~wave in many different places. Just imagine, for example, a~wave travelling in water, 
the propagation of acoustic waves or a~slackline that is stretched between two trees and set swinging.
When we try to describe this physical phenomenon of a~wave mathematically (and make some assumptions for simplification), 
we get the \textit{partial differential equation}
\begin{align} \label{waveequationexplained}
u_{tt} = \mu^2 u_{\xi \xi}.
\end{align}
Given the slackline example, this can be interpreted as follows: 
The function $u(t,\xi)$ describes the elevation of the slackline (relative to the resting state) 
at a~certain position $\xi \in[0,L]$ on the slackline at time $t \in [0,T]$. The factor $\mu > 0$ describes the wave speed. 
In this context, equation~\eqref{waveequationexplained} just means that the acceleration in time equals the acceleration along the position on the 
slackline multiplied by the squared wave speed.
This correlation originates from the Newtons well-known law $F=ma$, where $F$ denotes the force (that acts in the swinging system), 
$m$ the mass (of the slackline at a~point) and $a$ the acceleration (in time).
What is left to complete the physical mechanism of a~swinging slackline, are the initial and boundary states. 
In this case, we have the boundary conditions 
$$u(t,0) = u(t,L) = 0,$$
since the ends of the slackline are tied to the tree. 
Although these boundary conditions can vary for all the examples above, 
their \enquote{movement} is described by equation~\eqref{waveequationexplained}. This partial differential equation is called 
\textit{1D linear wave equation}. It is an~example of a~so-called \textit{Hamiltonian system}.
Partial differential equations like the 1D linear wave equation can in general not be solved analytically. 
They are usually reduced to a~\textit{ordinary differential equation} (ODE) by discretizing the spatial domain. 
The resulting ODE can then be solved by numerical integration. 
The problem that we run into here is that the discretization of the spatial domain leads to very high dimensional ODEs.
This makes it numerically expensive.
The question is: How can we efficiently solve such a~discretized Hamiltonian system with sufficient accuracy?
One way of doing this is a~\textit{structure-preserving model reduction}. 
This approach uses so-called \textit{encoder} and \textit{decoder} functions to map the Hamiltonian system into a~system of smaller dimension.
The crucial point is that in this reduced dimension, we get a~Hamiltonian system again.
We then solve the reduced system and recover an~approximate solution of the original system by mapping the reduced solution back to 
the full dimension using the decoder function.

The goal of this thesis is to discuss a~nonlinear approach to determine an~encoder-decoder pair such that we minimize the 
error of the reconstructed approximate solutions.
Precisely, we will use an~\textit{autoencoder network} to learn the encoder-decoder pair.
The conservation of structure requires that the encoder and decoder functions are so-called \textit{symplectic} functions. 
This poses some additional challenges, as we need to ensure this structure in each network layer. 
Since these symplecticity constraints are described by certain manifolds, 
we will apply manifold optimization techniques to ensure symplecticity of the encoder and decoder. 

At the beginning, in Section~\ref{manifolds}, we give a~brief overview of the most common manifold concepts and the geometric 
objects needed in later chapters. We will also focus on how to generalize optimization algorithms to manifolds.
Section~\ref{sec:modelreduction} is the main part of the thesis. 
Here, we describe the model reduction of Hamiltonian systems and introduce an~autoencoder network to find the encoder-decoder pair. 
The network we present was first introduced in \cite{main:brantner2023}. Our goal is to examine the different network layers and 
discuss some change proposals.
Here, we will draw special attention to the manifold update step.
More specifically, we will introduce a~new manifold optimization step that generalizes the Adam optimizer to the Stiefel manifold.
At the end of this section, we show different learning setups for the network with and without our change proposals. We also explain how
we implemented these setups in the \textit{Julia} programming language.
Finally, in Section~\ref{sec:numericalexperiments}, we present the numerical results for the implemented learning setups for two different 
examples of Hamiltonian systems.

\newpage
\section{Manifolds and optimization} \label{manifolds}

This chapter is the theoretical basis for the following chapters. 
In Section~\ref{sec:manifolds_subsec:manifoldtheory}, we give a~compact summary of the manifold theory needed in this thesis. 
In Section~\ref{sec:manifolds_subsec:optimizationtheory}, we discuss the problems and differences of manifold optimization in comparison 
to Euclidean optimization. 
Lie groups, a~special kind of manifolds with some additional structure, are considered in Section~\ref{sec:manifolds_subsec:liegrouptheory}.
In Section~\ref{sec:manifolds_subsec:examplemanifolds}, we introduce the orthogonal group, the compact Stiefel manifold and the symplectic Stiefel 
manifold. We present core properties of these manifolds which we will use in later chapters.

\subsection{Manifold theory} \label{sec:manifolds_subsec:manifoldtheory}

This section aims to give a~brief summary of the most important concepts of manifold theory. 
It is intended more as a~reminder than a~complete and clean introduction. 
A~detailed introduction can be found in \cite{optonmatmans:absil2008}, from which most of the results in this section are taken.

To start with, an~\textit{$N$-dimensional manifold} is a~geometrical object consisting of a~set $\M$ and a~maximal atlas $A^{+}$ of charts 
from $\M$ to $\R^N$ so that the topology induced by $A^{+}$ on $\M$ is second-countable and Hausdorff. One can think of a~manifold as a~space 
that locally looks like the Euclidean space $\R^{N}$ in every point. But, whereas the Euclidean space is "straight", the manifold can be curved. 

There are many ways to discover new manifolds from existing ones. One possibility is to obtain a~manifold $(\Nmann, B^{+})$, which is a~subset of a~given manifold 
$(\M, A^{+})$, as a~so-called \textit{embedded submanifold of $\M$}. For this, the inclusion map $\iota: \Nmann \to \M$ must be an~immersion and the topology induced by 
$B^{+}$ needs to equal the subspace topology induced by the ambient space.

A~manifold as defined above carries only a~differentiable structure, but we have no information about angles or distances on the manifold. 
To get this information, we need to add more structure. For this, we endow the tangent spaces of $\M$ at a~point $X \in \M$, 
denoted by $T_X\M$, with a~smoothly varying inner product 
$$g_X(Z_1,Z_2) := \langle Z_1,Z_2 \rangle_X \quad \text{for all } Z_1,Z_2 \in T_X\M.$$
We call the tuple $(\M, g)$ a~\textit{Riemannian manifold}. 

We can define a~generalization of the gradient of a~smooth function $f: \M \to \R$. For the Euclidean space $\M=\R^N$, we have 
$\nabla f(X) := \left(\tfrac{\partial f(X)}{\partial X_1}, ... ,\tfrac{\partial f(X)}{\partial X_N}\right)^T$. We cannot use this, because we have no equivalent for
$\tfrac{\partial f(X)}{\partial X_i}$ on general manifolds. But we do know that $\nabla f(X)$ is the unique vector satisfying $\langle \nabla f(X), Z \rangle_X = DF(X)[Z]$ for all tangent vectors $Z \in T_X\M$.
We use this property to define the \textit{Riemannian gradient} on a~Riemannian manifold $(\M, g)$ at a~point $X \in \M$ as the unique element 
$\grad f(X) \in T_X\M$ satisfying
$$\langle \grad f(X), Z \rangle_X = DF(X)[Z] \quad \text{ for all } Z \in T_X\M.$$ 

One problem, we have on (Riemannian) manifolds, is that there is no canonical connection between tangent spaces of different points of $\M$.
The so-called \textit{affine connection} and in particular the unique \textit{Riemannian connection} (with respect to a~Riemannian metric) 
establish such a~connection. They generalize the concept of directional derivates to vector fields.

Using the Riemannian connection, we can now specify second-order derivates of curves on a~Riemannian manifold. This is realized by \textit{acceleration vector fields}. 
We call a~curve on $\M$ with zero acceleration a~\textit{geodesic}. Geodesics play a~crucial role in many manifold applications.
They are uniquely defined by a~starting point and an~initial velocity vector. Using geodesics, we can define the \textit{exponential mapping} (for a~point $X \in \M$)
\begin{align} \label{sec:manifold_env:align_name:manifoldtheoryexponentialmap}
  \Exp_X: T_X\M &\to \M \\
  Z &\mapsto \gamma(1,X,Z), \nonumber
\end{align}
where $\gamma(1,X,Z)$ denotes the geodesic with starting point $X$ and an~initial velocity vector $Z$ evaluated at time $t=1$. 
Note that the exponential mapping is not necessarily defined for all elements of the tangent space.

The exponential mapping maps elements of the tangent space back to the manifold. It fulfills two core properties:
\begin{itemize}
  \item[1.] $\Exp_X$ is differentiable, and it holds $\Exp_X(0_X) = X$ for all $X \in \M$.
  \item[2.] $D\Exp_X(0_X) \equiv \id_{T_X\M}$ (using the canonical identification $T_{0_X}T_X\M \simeq T_X\M$).
\end{itemize}
These are two important properties. Locally around a~point $X \in \M$, it makes no difference (for the first-order derivative) 
if we go along a~curve $\beta$ on $\M$ with initial direction $Z$ or if we go in direction $Z$ on the tangent space and then map 
back to $\M$ via $\Exp_X$:
$$ \frac{d}{dt} \beta(t)\bigg|_{t=0} = \frac{d}{dt} \Exp_X(tZ)\bigg|_{t=0}.$$
A~mapping $R_X: T_X\M \to \M$ with these properties is generally called a~\textit{retraction}. Retractions play an~important role in 
manifold optimization.

The exponential map has another interesting property: It \textit{parallel-transports} its own velocity vectors, meaning that it 
keeps the direction (with respect to the Riemannian connection) of the tangent vector while transporting it from one tangent space to another.

Computation of the parallel transport often requires to solve a~nonlinear ordinary differential equation, which makes it an~expensive task in practice.
A~generalization of the concept of parallel transport that is computationally cheaper is the \textit{vector transport}.
Vector transport, in general, does not make use of second-order information, and thus the computational costs can be reduced significantly.
This cost reduction comes with a~loss of accuracy because the vectors are no longer transported 
with respect to the underlying Riemannian connection. 
As this is not a~standard concept of manifold theory, we will outline some important properties of vector transport in more detail.
A~formal definition of this concept is as follows:
\begin{definition}
  A~vector transport on a~manifold $\M$ with a~tangent bundle $T\M$ is a~smooth mapping
  \begin{align*}
    \T: T\M \oplus T\M &\to T\M \\
    (Z_X, Y_X) &\mapsto \T_{Z_X}(Y_X) \in T\M
  \end{align*}
  satisfying the following properties for all $X \in \M$:
  \begin{itemize}
    \item[(i)] (Associated retraction) There exists a~retraction $R_X$, called the retraction associated with $\T$, such that
          $$\pi \left( \T_{Z_X}(Y_X) \right) = R_X(Z_X) \quad \text{for all } X \in \M,$$
          where $\pi\left( \T_{Z_X}(Y_X) \right)$ denotes the foot of the tangent vector $\T_{Z_X}(Y_X)$.
    \item[(ii)] (Consistency) $\T_{0_X}(Y_X) = Y_X$ for all $Y_X \in T_X\M$.
    \item[(iii)] (Linearity) $\T_{Z_X}(aY_X + bW_X) = a\T_{Z_X}(Y_X) + b\T_{Z_X}(W_X)$ for all $Y_X,W_X \in T_X\M$ and $a,b \in \R$.
  \end{itemize}
\end{definition}

There are different ways of retrieving a~vector transport. Here, we introduce two computationally efficient ways to obtain 
a~vector transport along a~given retraction. One method uses a~differentiated retraction.

\begin{lemma}\cite[Chapter 8.1.2]{optonmatmans:absil2008} \label{sec:manifold_env:lemma_name:manifoldtheorydiffvectransport}
  Let $\M$ be a~manifold endowed with a~retraction $R$. Then a~vector transport along $R$ on $\M$ is defined by 
  \begin{align*}
    \T_{Z_X}(Y_X) 
    &:= DR_X(Z_X)[Y_X] \\
    &= \frac{d}{dt}R_X(Z_X + tY_X)\bigg|_{t=0}.
  \end{align*}
\end{lemma}
This method works for all smooth manifolds.
Another method only works for submanifolds of Euclidean spaces.

\begin{lemma}\cite[Chapter 8.1.3]{optonmatmans:absil2008} \label{sec:manifold_env:lemma_name:manifoldtheorysubvectransport}
  If $\M$ is an~embedded submanifold of a~Euclidean space $\mathcal{E}$ and $\M$ is endowed with a~retraction $R$, then, relying on the natural 
  inclusion $T_X\M \subseteq \mathcal{E}$ for all $X \in \M$, the vector transport is determined by
  \begin{align*}
    \T_{Z_X}(Y_X)
    &:= \P_{R_X(Z_X)}(Y_X),
  \end{align*}
  where $\P_X$ denotes the orthogonal projector onto $T_X\M$.
\end{lemma}
Note that we are not restricted to any specific Riemannian metric as vector transports do not rely on any affine connection. 
It is not at all clear whether the retrieved vector transport approximately maintains the vector's direction with respect 
to the underlying affine connection as parallel transport does. How well each vector transport works depends on the application for which it is used.

\subsection{Optimization on manifolds} \label{sec:manifolds_subsec:optimizationtheory}

An~unconstrained Euclidean optimization problem has the form
\begin{align}\label{sec:manifold_env:align_name:optimizationeuclideanoptprob}
  &\min f(x) \quad \st \quad x \in \R^{N},
\end{align}
where $f:\R^N \to \R$ is an~at least continuous function.
Often, the search space is restricted to a~subset of $\R^N$. In this case, we talk about a~constrained optimization problem.
Sometimes we can reformulate the constraints so that the search space is a~manifold.
This way, we aim to solve an~unconstrained optimization problem on a~manifold instead:
\begin{align}\label{sec:manifold_env:align_name:optimizationmanifoldoptprob}
  &\min f(X) \quad \st \quad X \in \M.
\end{align}
This causes some new issues as we need to modify the Euclidean optimization algorithms for \eqref{sec:manifold_env:align_name:optimizationeuclideanoptprob}
to make them work on manifolds. 
We will now use the \textit{gradient descent method} in $\R^N$ to discuss problems and adaptations for optimization algorithms on Riemannian manifolds. 
The concepts used in this chapter come from \cite{optonmatmans:absil2008}.
Let $f\in C^1$ and $x^{(k)} \in R^N$. An~optimization step of the gradient descent method $\R^N$ takes the following steps:
\begin{itemize}
  \item[1.] Compute a~descent direction $v^{(k)} = - \nabla f(x^{(k)})$.
  \item[2.] Compute a~step size $\eta^{(k)} > 0$ (by using any preferred rule like an Armijo rule).
  \item[3.] Update the iterate: $x^{(k+1)} = x^{(k)} + \eta^{(k)} v^{(k)}$.  
\end{itemize}
The gradient descent method is based on using the negative gradient as the direction of descent. Indeed, Taylor's theorem shows
$$f(x - \eta \nabla f(x) ) < f(x) $$
for small enough $\eta > 0$. Step one and two of the gradient descent method above can easily be adapted to manifolds. Instead of the Euclidean gradient,
we use the Riemannian gradient in step one. We can also find a~matching update rule for $\eta$ (details omitted). Problems arise in 
step three: We cannot simply add the descent direction to the current iterate because manifolds do not necessarily have the structure 
of a~vector space. As the Riemannian gradient $\grad f(X)$ is an~element of the tangent space, we will need to find a~way to map the update vector
back to the manifold. Retractions, discussed in Section~\ref{sec:manifolds_subsec:manifoldtheory}, are a~good choice here.
Given a~retraction $R_{X^{(k)}}: T_{X^{(k)}}\M \to \M$, step three turns into
\begin{align}\label{sec:manifold_env:align_name:optimizationmanifoldmanifoldupdatestep}
  X^{(k+1)} = R_{X^{(k)}}\left(\eta^{(k)} V^{(k)}\right) \in \M,
\end{align}
where $V^{(k)} = -\grad f(X^{(k)})$.
It remains to show that the retraction maintains the descent property, i.e., $f(R_X(-\eta \grad f(X))) < f(X)$ for $\eta$ small enough. 
Again, we can use Taylor's theorem: 
\begin{align*}
  f(R_X(-\eta \grad f(X))) 
  &= f(R_X(0_X)) + D(f\circ R_X)(0_X)[-\eta \grad f(X)] + \O(\norm{\eta \grad f(X)}^2_X) \\
  &= f(X) + Df(X)[-\eta \grad f(X)] + \O(\norm{\eta \grad f(X)}^2_X) \\
  &= f(X) - \eta \norm{ \grad f(X)}^2_X + \O(\eta^2 \norm{\grad f(X)}^2_X).
\end{align*}
For $\eta > 0$ small enough this gives the desired result.
The difficulty is to find a~retraction for the manifold that can be computed efficiently. 
A~canonical choice would be the exponential mapping~\eqref{sec:manifold_env:align_name:manifoldtheoryexponentialmap}. 
But since this is computationally expensive, it is often not a~good choice in practice. 

The Riemannian gradient descent method requires only first-order derivatives, which makes it fairly easy to adapt to manifolds via retractions. 
Optimization algorithms that need second-order information are more difficult.
More detailed information about this can be found in \cite[Section 6]{optonmatmans:absil2008}.

\subsection{Matrix Lie groups and homogeneous spaces}\label{sec:manifolds_subsec:liegrouptheory}

A~special kind of manifolds are \textit{(matrix) Lie groups}. These are manifolds with an~additional group structure. In this section, we summarize 
the main properties of matrix Lie groups and how they interact with other manifolds as a~so-called \textit{homogeneous space}.
The definitions and results used in this chapter are taken from \cite{matrixlie:hall2000}, \cite{top:steimle2021} and \cite{adam:brantner2023}.

\subsubsection{Matrix Lie groups}
Matrix Lie groups are defined as closed subgroups of the group of all invertible  real $N\times N$ matrices $\GL(N, \R)$.
\begin{definition} \label{sec:manifold_env:definition_name:liegroup}
  A~closed subgroup $\G$ of $\GL(N, \R)$ is called a~\textit{matrix Lie group}.
\end{definition}
\noindent Although this is not at all clear from the definition, (matrix) Lie groups are indeed (smooth) manifolds itself (cf. \cite[Theorem 2.15]{matrixlie:hall2000}) and 
thus offer a~useful link between group theory and manifold theory.
One of the most important concepts is the \textit{Lie algebra} of a~matrix Lie group.
\begin{definition} \label{sec:manifold_env:definition_name:liealgebra}
  Let $\G \subseteq \GL(N,\R)$ be a~matrix Lie group. Then the \textit{Lie algebra} of $\G$, denoted by $\g$, is the set of all matrices
  $Z\in \R^{N \times N}$ such that the matrix exponential $e^{tZ}$ is in $\G$ for all real numbers $t$, i.e.,
  \begin{align*}  
    \g = \left\{ Z \in \R^{N \times N} \; \big| \; e^{tZ} \in \G \text{ for all } t \in \R \right\}.
  \end{align*}
\end{definition}
\noindent The Lie algebra is defined by curves of the matrix exponential. It can be shown that $\g$ is a~$\R$-vector space and thus offers a~wide range 
of operations that cannot be performed on the Lie group directly.
The \textit{exponential map}
\begin{align} \label{sec:manifold_env:align_name:liegroupexponentialmap} 
    \exp: \g &\to \G \\ 
    Z &\mapsto e^{tZ} \nonumber
\end{align}
takes the Lie algebra back to the Lie group and is locally homeomorphic around $0 \in \g$ and $I_N \in \G$ (cf. \cite[Theorem 3.23]{matrixlie:hall2000}). 
From this we can deduce that the tangent space at each point $A \in \G$ is simply a~transformation of the Lie algebra:
\begin{align} \label{sec:manifold_env:align_name:liegrouptangentspacetransformedliealgreba} 
  T_A\G = \g A:= \left\{ZA \; \big| \; Z \in \g \right\}.
\end{align} 
Note that for the identity matrix $I_N \in \G$, we have $T_{I_N}\G = \g$.
In the following, we will use $\g A$ both for the tangent space $T_A\G$ and for the map $\g A: \g \to T_A\G, Z \mapsto ZA$.

\subsubsection{Lie groups and homogeneous spaces}
As already mentioned, matrix Lie groups combine an~algebraic group structure with manifold properties. We want to use this to define 
a link between a~matrix Lie group and a~matrix manifold for a~better understanding of the internal structure of the manifold.
\begin{definition} \label{sec:manifold_env:definition_name:homogeneousspace}
  Let $\G \subseteq \GL(N,\R)$ be a~matrix Lie group and let $\M \subseteq \GL(N,\R)$ be a~(Riemannian) matrix manifold such that the
  standard matrix multiplication 
  \begin{align*}
    \G \times \M &\to \M \\
    (A,X) &\mapsto AX,
  \end{align*}
  (which defines a~smooth group action) acts transitive on $\M$. Then the tuple $(\G,\M)$ is called a~\textit{homogeneous space}. 
\end{definition}
\noindent Note that this is a~very specific definition of a~homogeneous space that fits our purposes. In a~more general setting, homogeneous spaces are 
often defined as a~tuple of a~topological group and a~topological space together with a~transitive group action.
Acting transitive means that $\G\cdot X = \M$ for all $X \in \M$. 
As we only require the group and any element of $\M$ to recover the whole manifold, a~homogeneous space can be understood as some 
kind of inner symmetry of $\M$ encoded by the Lie group $\G$.

\subsubsection{Retractions on homogeneous spaces} 
We now want to use the homogeneous space connection between the matrix Lie group $\G$ and the matrix manifold $\M$ to construct a~retraction $R^{\M}_X$ on $\M$ at $X \in \M$ 
from a~retraction $R^{\G}_{I_N}$ on $\G$ at the identity $I_N$.

Let $\G \subseteq \GL(N,\R)$ be a~matrix Lie group with a~metric $\langle\cdot, \cdot\rangle_{\G}$ that turns it into a~Riemannian manifold.
Furthermore, let $\M \subseteq \GL(N,\R)$ be a~Riemannian matrix manifold such that the
tuple $(\G,\M)$ defines a~homogeneous space. For an~element $X \in \M$, we define 
\begin{itemize}
  \item[1.] $\G X: \G \to \M, A\mapsto AX$,
  \item[2.] $\g X: \g \to T_X\M, Z \mapsto ZX$.
\end{itemize}
In the following, we will use $\G X$ (and $\g X$) both as the map and the map's image.
Note that $\g X$ is a~surjective linear map, and it holds $D(\G X)(I_N) \equiv \g X$.

Using the metric $\langle \cdot, \cdot \rangle_\G$, we can now make the following definitions.
\begin{definition}
  Let $(\G,\M)$ be a~homogeneous space as described above. Let $X,E \in \M$.
  \begin{itemize}
    \item[1.] The orthogonal complement $\g^{hor,X} := \ker(\g X)^{\perp}$ of the kernel of $\g X$ with respect to $\langle\cdot, \cdot\rangle_{\G}$
    is called \textit{horizontal component} of $\g$ at $X$.
    \item[2.] The isomorphism $\Omega_X := \left(\g X \big|_{\ker(\g X)^{\perp}}\right)^{-1}: T_X\M \to \g^{hor,X}$ 
    is called \textit{lift} at~$X$.
    \item[3.] A~function $\lambda_E: \M \to \G$ such that
    $$ \lambda_E(X)E = X, \quad X \in \M,$$
    is called a~\textit{section} at $E$. We call $E$ a~\textit{distinct element}.
  \end{itemize}
\end{definition}
\noindent A~section can be thought of as a~left-inverse to $\G E$ as $\G E \circ \lambda_E = \id_{\M}$. 
Note that it is not at all clear that such a~section exists, whether it is unique, or how to construct it. 
This heavily depends on the choice of the distinct element $E$. 
\begin{lemma} \label{sec:manifold_env:lemma_name:ghorXtoghorE}
  \cite[Proposition 1]{adam:brantner2023}
  Let $X \in \M$ and let $E \in \M$ be a~distinct element with a~section $\lambda_E$. Then the map
  \begin{align*} 
    \iota: \g^{hor,X} &\to \g^{hor,E} \\ 
    Z &\mapsto \lambda_E(X)^{-1}Z\lambda_E(X) \nonumber
  \end{align*}
  is a~well-defined isomorphism.  
\end{lemma}
\noindent Using Lemma~\ref{sec:manifold_env:lemma_name:ghorXtoghorE}, we can now construct a~retraction on $\M$ at an~arbitrary element $X\in \M$.
\begin{theorem} \label{sec:manifold_env:theorem_name:derivedretraction}
  Let $X \in \M$ and let $E\in \M$ be a~distinct element with a~section $\lambda_E$. 
  Let $R_{I_N}^{\G}: \g \to \G$ be an~arbitrary retraction on $\G$ at the identity $I_N$. Then the map
  \begin{align*}
    R_X^{\M}: T_X\M &\to \M \\ 
    Z &\mapsto \lambda_E(X) R_{I_N}^{\G} \left( \lambda_E(X)^{-1} \Omega_X(Z)\lambda_E(X) \right)E \nonumber
  \end{align*} 
  defines a~retraction on $\M$ at $X$.
\end{theorem}

\begin{proof} \ 
  \begin{itemize}
    \item[1.] First, we calculate the image of $R_X^{\M}$ at $0_X$:
    \begin{align*}
      R_X^{\M}(0_X) &= \lambda_E(X) R_{I_N}^{\G}(0_X)E = \lambda_E(X)I_NE = X. 
    \end{align*}
    \item[2.] Next, we calculate $DR_X^{\M}(0_X)$:
    \begin{align*}
      DR_X^{\M}(0_X)[Z] &= \lambda_E(X) \left( DR_{I_N}^{\G}(0_X)\left[ \lambda_E(X)^{-1} \Omega_X(Z)\lambda_E(X) \right]\right)E   \\
      &= \lambda_E(X) \left( \lambda_E(X)^{-1} \Omega_X(Z)\lambda_E(X) \right)E \\
      &= \Omega_X(Z)X \\
      &= Z
    \end{align*}
    for all $Z \in T_X\M$.
  \end{itemize}
\end{proof}
\noindent The chain of mappings to receive the retraction $R_X^{\M}$ from Theorem~\ref{sec:manifold_env:theorem_name:derivedretraction} can be
visualized as
\begin{align}\label{sec:manifold_env:align_name:chainofmapsofderivedretractions}
  T_X\M 
\xrightarrow{\Omega_X} \g^{hor,X} 
\xrightarrow{\lambda_E(X)^{-1}(\cdot)\lambda_E(X)} \g^{hor,E} 
\xrightarrow{R_{I_N}^{\G}} \G 
\xrightarrow{\G E} \M 
\xrightarrow{\lambda_E(X)\M} \M.
\end{align}
This way of deriving a~retraction on a~manifold $\M$ may look unnecessarily complex, but it gives us 
a~vector space that is accessed in the retraction map at every $X \in \M$. 
As displayed in the chain map \eqref{sec:manifold_env:align_name:chainofmapsofderivedretractions}, this vector space is $\g^{hor,E}$, 
the so-called \textit{global tangent space representation}.
As we will see later, this helps us to apply an~optimization step using moments.

\subsection{Special matrix manifolds}\label{sec:manifolds_subsec:examplemanifolds}

This section presents some important matrix Riemannian manifolds. We will have a~close look at some of their main properties 
and their connection to each other. First, we examine the \textit{orthogonal group} both as a~Riemannian manifold and 
as a~matrix Lie group. Second, we introduce the \textit{compact Stiefel manifold}. We define two different Riemannian metrics on this manifold, 
and we show that the compact Stiefel manifold together with the orthogonal group admits the structure of a~homogeneous space.
Finally, we take a~short look at the \textit{symplectic Stiefel manifold} and how it is related to the compact Stiefel manifold.

Most of the results about the orthogonal group and the compact Stiefel manifold are taken from \cite{algoswithorthcons:edelmann1998}, \cite{feasiblemethodortho:wen2013} and \cite{adam:brantner2023}. 
As our main source for the symplectic Stiefel manifold, we used \cite{symplstiefel:stykel2021}.
We will need all three of these manifolds in later chapters, where we will use most of the observed structures 
in a~manifold optimization use case. 

\subsubsection{Orthogonal group} \label{sec:manifolds_subsubsec:orthogonalgroup}

The set of all orthogonal matrices 
\begin{align} \label{sec:manifolds_env:align_name:setoforthogonalmatrices}
  \Orth(N) &= \{ X \in \R^{N \times N} \; \big| \; X^TX = I_N \}
\end{align}
has the structure of a~closed, embedded submanifold of $\R^{N \times N}$. To see this, we can apply the submersion 
theorem to the smooth map
\begin{align*}
  F: \R^{N \times N} &\to \S_{\sym}(N) \\ 
  X &\mapsto X^TX
\end{align*}
from the Euclidean manifold $\R^{N \times N}$ to the Euclidean manifold $\S_{\sym}(N)$ of all symmetric matrices in $\R^{N \times N}$. 
As the identity matrix $I_N \in \R^{N \times N}$ is a~regular value of $F$, the orthogonal group $\Orth(N) = F^{-1}(I_N)$ is a~submanifold of dimension $N^2 - \tfrac{N(N+1)}{2} = \tfrac{N(N-1)}{2}$.
This set is closed, as it is the preimage of a~one-point set under a~continuous function. Since the set is bounded, compactness follows.
For a~point $X \in \Orth(N)$, we obtain the tangent space 
\begin{align} \label{sec:manifolds_env:align_orthogonalgrouptangentspace}
  T_X\Orth(N) &= \ker(DF(X)) = \{Z \; \big| \; X^TZ + Z^TX = 0 \} = \{XW \; \big| \; W \in \S_{\skew}(N) \}.
\end{align}

\subsubsubsection{Euclidean metric and the Riemannian gradient}
As a~submanifold, $\Orth(N)$ naturally inherits the Euclidean metric of $\R^{N \times N}$
\begin{align} \label{sec:manifolds_env:align_name:orthogonalgroupeuclideanmetric}
  g_e(Z_1, Z_2) := \tr(Z_1^T Z_2) \qquad \text{ for all } Z_1, Z_2 \in T_X\Orth(N),
\end{align}
turning it into a~Riemannian submanifold. Note that this metric induces the Frobenius norm.

The orthogonal complement of a~tangent space $T_X\Orth(N)$ with respect to the Euclidean metric is given by 
$(T_X\Orth(N))^{\perp} = \{XS \; \big| \; S \in \S_{\sym}(N)\}$.
For any $Y \in \R^{N \times N}$, the Euclidean metric induces the orthogonal projection 
\begin{align} \label{sec:manifold_env:align_name:orthogonalgroupeuclideanprojection}
   \P_X(Y) = X\skew(X^TY)  \in T_X\Orth(N) 
\end{align}
onto the tangent space at $X$, and
\begin{align} \label{sec:manifold_env:align_name:orthogonalgroupeuclideanorthogonalprojection}
   \P_X^{\perp}(Y) = X\sym(X^TY) \in (T_X\Orth(N))^{\perp}
\end{align}
onto its orthogonal complement. Here, $\skew(X^TY) := \frac{1}{2}\left(X^TY - Y^TX\right)$ denotes the skew-symmetric part 
of $X^TY$ and $\sym(X^TY) := \frac{1}{2}\left(X^TY + Y^TX\right)$ denotes the symmetric part. 
Due to a~well-known result for Riemannian submanifolds, 
this immediately yields the Riemannian gradient of a~function $f$ with respect to $g_e$,
\begin{align} \label{sec:manifold_env:align_name:orthogonalgroupeuclideanriemanniangradient}
 \grad_e f(X) &= \P_X \left(\nabla \overline{f}(X)\right) = \frac{1}{2} X\left(X^T\nabla \overline{f}(X) - \nabla \overline{f}(X)^TX\right),
\end{align}
where $\nabla \overline{f}(X)$ denotes the Euclidean gradient of a~smooth extension $\overline{f}$ of $f$ around $X$ in $\R^{N \times N}$.
Computing this gradient requires $\O(N^3)$ operations.

\subsubsubsection{The orthogonal group as a~matrix Lie group} \label{sec:manifolds_subsubsubsec:orthogonalgroupliegroup}
As the name implies, the orthogonal group is not only a~Riemannian manifold 
but also an~algebraic group with the matrix multiplication as the group operation and the identity matrix as the neutral element.
With respect to the Euclidean metric on $\R^{N \times N}$, $\Orth(N)$ turns into a~closed subgroup of the group of all invertible real $N\times N$ matrices $\GL(N,\R)$.
According to Definition~\ref{sec:manifold_env:definition_name:liegroup}, $\Orth(N)$ is therefore a~matrix Lie group.
This link between group structure and manifold structure offers a~number of useful properties. Due to \eqref{sec:manifold_env:align_name:liegrouptangentspacetransformedliealgreba},
the tangent space at each point $X \in \Orth(N)$ is simply a~transformation of the Lie algebra $\g$:
\begin{align} \label{sec:manifold_env:align_name:orthogonalgrouptangentspacetransformedliealgreba}
  T_X\Orth(N) = \g X := \left\{ZX  \; \big| \; Z \in \g \right\}.
\end{align}
The Lie algebra is defined as the set of all matrices for which the curve of the matrix exponential lies in $\Orth(N)$ (see Definition~\ref{sec:manifold_env:definition_name:liealgebra}).
Therefore, the tangent space of the neutral element $I_N \in \Orth(N)$ is simply $T_{I_N}\Orth(N) = \g$.

\subsubsubsection{Cayley retraction} \label{sec:manifolds_subsubsubsec:orthogonalgroupcayley}
The \textit{Cayley transform} of a~Lie group is defined as 
\begin{align} \label{sec:manifold_env:align_name:orthogonalgroupcayleytransform} 
  \cay: \g &\to \Orth(N) \\ 
  Z &\mapsto (I_N-Z)^{-1}(I_N+Z). \nonumber
\end{align} 
According to \cite[Lemma 8.8]{geomnumint:hairer2006}, its directional derivatives at a~point $Z \in \g$ are given as
\begin{align} \label{sec:manifold_env:align_name:orthogonalgroupDcay} 
  D\cay(Z)[Y] = 2(I_N-Z)^{-1}Y (I_N-Z)^{-1}.
\end{align}
Now it is easy to check that 
\begin{align} \label{sec:manifold_env:align_name:orthogonalgroupcayleyretraction} 
  R^{\cay}_{I_N}: T_{I_N} \Orth(N) &\to \Orth(N) \\
  Z &\mapsto \cay\left(\frac{1}{2}Z\right) \nonumber
\end{align}
defines a~retraction on $\Orth(N)$ at the identity $I_N$. 
From a~numerical point of view, it is an~important question whether this can be calculated efficiently. The main problem here is that 
\eqref{sec:manifold_env:align_name:orthogonalgroupcayleyretraction} contains a~high dimensional matrix inverse. 
In general, this problem cannot be solved efficiently. But, given a~$Z \in T_{I_N}\Orth(N) = \g = \S_{\skew(N)}$ with $\rank(Z) := n \ll N$, we can
get around this issue by applying the Sherman–Morrison–Woodbury formula (SMW). 
First, we observe that $Z$ can be decomposed as $Z = UV$ with $U \in \R^{N \times n}$ and $V \in \R^{n \times N}$.
Now we can apply the SMW formula to the matrix-inverse and get
\begin{align} \label{sec:manifold_env:align_name:orthogonalgroupSMWformulaforinverse}
  \left(I_N-\frac{1}{2}Z\right)^{-1} &= \left(I_N-\frac{1}{2}UV\right)^{-1} = I_N + \frac{1}{2} U \left(I_{n} -\frac{1}{2}VU\right)^{-1}V.
\end{align}
This only requires a~$n \times n$ matrix inverse causing costs of $\O(n^3)$. This is especially advantageous for applications where $n \ll N$.
Putting these steps together, we compute the Cayley retraction as follows:
\begin{align*} 
R^{\cay}_{I_N}(Z)
  =\cay\left(\frac{1}{2}Z\right) 
  &= \left(I_N-\frac{1}{2}UV\right)^{-1}\left(I_N+\frac{1}{2}UV\right) \\
  &= \left(I_N + \frac{1}{2} U \left(I_{n} -\frac{1}{2}VU\right)^{-1}V\right)\left(I_N+\frac{1}{2}UV\right). \nonumber
\end{align*}
This calculation can be done in $\O(N^2n)$. Whereas this is still quite expensive, the calculation gets cheaper 
when a~matrix $X \in \R^{N \times n}$ with $n \ll N$ is multiplied from the right:
\begin{align}\label{sec:manifold_env:align_name:orthogonalgroupSMWformulacayleywithrightmultiplication}
  \cay\left(\frac{1}{2}Z\right) X
  &= \left(I_N + \frac{1}{2} U \left(I_{n} -\frac{1}{2}VU\right)^{-1}V\right)\left(I_N+\frac{1}{2}UV\right)X \\
  &= \left(X + \frac{1}{2}U (VX) \right) + \frac{1}{2}U \left( \left( I_{n} - \frac{1}{2}VU \right)^{-1} \left( VX + \frac{1}{2}VU(VX)  \right) \right). \nonumber
\end{align}
The computational costs of \eqref{sec:manifold_env:align_name:orthogonalgroupSMWformulacayleywithrightmultiplication} are now $\O(Nn^2)$. We will make use of this 
computational trick several times in later chapters.

\subsubsubsection{Overview}
We summarize all the important results from this section in Table~\ref{sec:manifold_env:table_name:orthogonalgrouptable}. 
Note that we did not always display the most efficient representations, and therefore they may not match the computational cost shown.
\begin{table}[ht!]
\centering
\begin{tabular}{||c c c||} 
 \hline
 Name & Formula & Costs \\ [0.5ex] 
 \hline\hline
 $T_X\Orth(N)$      & $\{XW | W \in \S_{\skew}(N) \} = \left\{ZX | Z \in \g \right\}$                     & ---                 \\
 $g_e(Z_1,Z_2)$     & $\tr(Z_1^T Z_2)$                                                                      & $\O(N^2)$                 \\
 $\P_X(Y)$            & $X\skew(X^TY)$                                                                      & $\O(N^3)$           \\
 $\P^{\perp}_X(Y)$    & $X\sym(X^TY)$                                                                       & $\O(N^3)$           \\
 $\grad_e f(X)$     & $\frac{1}{2} X \left(X^T\nabla \overline{f}(X) - \nabla \overline{f}(X)^TX\right)$   & $\O(N^3)$           \\
 $R^{\cay}_{I_N}(Z) $ & $\left(I_N-\frac{1}{2}Z\right)^{-1}\left(I_N+\frac{1}{2}Z\right)$                 & $\O(N^3)\text{ or }\O(N^2n)$  \\
 \hline
\end{tabular}
\caption{Geometric concepts for the orthogonal group $\Orth(N)$.}
\label{sec:manifold_env:table_name:orthogonalgrouptable}
\end{table}

\subsubsection{Compact Stiefel manifold}\label{sec:manifolds_subsubsec:compactstiefel}

In Section~\ref{sec:manifolds_subsubsec:orthogonalgroup} we have analyzed the set of all orthogonal matrices in $\R^{N \times N}$ 
and were able to detect the structure of a~Riemannian submanifold. The concept of orthogonal matrices can easily be generalized 
whilst still keeping the structure of a~manifold. Instead of orthogonal matrices, we now want to observe the set of all matrices 
consisting of $1\leq n\leq N$ orthonormal columns. We define this set as
\begin{align} \label{sec:manifolds_env:align_name:setofcompactstiefelmatrices}
  \St(n,N) &:= \{ X \in \R^{N \times n} \; \big| \; X^TX = I_n \}.
\end{align}
Similar to the case of the orthogonal group, applying the submersion theorem to the map 
\begin{align*}
  F: \St(n,N) &\to \S_{\sym}(n) \\
X &\mapsto X^TX 
\end{align*}
gives us a~compact, embedded submanifold of $\R^{N \times n}$ of dimension $Nn - \tfrac{n(n+1)}{2}$, 
the so-called \textit{compact Stiefel manifold}. Note that the orthogonal group can be seen as a~special case of the compact Stiefel 
manifold where $N=n$.
As done for $\Orth(N)$ in \eqref{sec:manifolds_env:align_orthogonalgrouptangentspace}, we obtain the tangent space to $\St(n,N)$ at a~point $X \in \St(n,N)$ as
\begin{align} \label{sec:manifolds_env:align_compactstiefeltangentspace}
  T_X\St(n,N) &= \ker(DF(X)) \\ 
  &= \{Z \; \big| \; X^TZ + Z^TX = 0 \} \nonumber \\ 
  &= \{XW + X_{\perp}K \; \big| \; W \in \S_{\skew}(n), K \in \R^{(N-n) \times n} \nonumber \},
\end{align}
where $\S_{\skew}(n)$ denotes the set of all skew-symmetric matrices in $\R^{n \times n}$ and $X_{\perp}$ is any 
matrix that spans the orthogonal complement of $X$. Unlike the tangent space of $\Orth(N)$, here each tangent vector consists 
of a~skew-symmetric part and an~additional part lying in the orthogonal complement of $X$.  

\subsubsubsection{Euclidean metric and the Riemannian gradient}
Analogous to \eqref{sec:manifolds_env:align_name:orthogonalgroupeuclideanmetric}, $\St(n,N)$ naturally inherits 
the Euclidean metric 
\begin{align} \label{sec:manifold_env:align_name:compactstiefeleuclideanmetric}
  g_e(Z_1, Z_2) := \tr(Z_1^T Z_2) \text{ for all } Z_1, Z_2 \in T_X\St(n,N)
\end{align}
from the ambient space $\R^{N \times n}$, turning it into a~Riemannian submanifold. 
The orthogonal complement of the tangent space at a~point $X$ is $(T_X\St(n,N))^{\perp} = \{XS \; \big| \; S \in \S_{\sym}(n)\}$.
For any $Y \in \R^{N \times n}$, we get the orthogonal projection 
\begin{align} \label{sec:manifold_env:align_name:compactstiefeleuclideanprojection}
   \P_X(Y) = (I_N - XX^T)Y + X \skew(X^TY) \in T_X\St(n,N) 
\end{align}
onto the tangent space at $X$, and
\begin{align} \label{sec:manifold_env:align_name:compactstiefeleuclideanorthogonalprojection}
   \P_X^{\perp}(Y) = X \sym(X^TY) \in (T_X\St(n,N))^{\perp}
\end{align}
onto its orthogonal complement. Like in \eqref{sec:manifold_env:align_name:orthogonalgroupeuclideanriemanniangradient}, we easily obtain the 
Riemannian gradient of a~function $f$ with respect to $g_e$ by applying the orthogonal projection. This is
\begin{align} \label{sec:manifold_env:align_name:compactstiefeleuclideanriemanniangradient}
 \grad_e f(X) 
 &= \P_X \left(\nabla \overline{f}(X)\right) \\
 &= (I_N - XX^T)\nabla \overline{f}(X) + X \frac{1}{2}\left(X^T\nabla \overline{f}(X) - \nabla \overline{f}(X)^TX\right) \nonumber \\
 &= \nabla \overline{f}(X) -  \frac{1}{2} X\left(X^T\nabla \overline{f}(X) + \nabla \overline{f}(X)^TX\right), \nonumber 
\end{align}
where $\nabla \overline{f}(X)$ denotes the Euclidean gradient of a~smooth extension $\overline{f}$ of $f$ around $X$ in $\R^{N \times n}$.
The matrix operations to compute the Riemannian gradient require $\O(Nn^2)$ flops.

\subsubsubsection{Canonical metric and the Riemannian gradient}
The compact Stiefel manifold can alternatively be endowed with another metric, the so-called \textit{canonical metric}:
\begin{align} \label{sec:manifold_env:align_name:compactstiefelcanonicalmetric}
  g_c(Z_1, Z_2) &:= \tr\left(Z_1^T \left(I_N - \frac{1}{2}XX^T \right) Z_2\right) \text{ for all } Z_1, Z_2 \in T_X\St(n,N).
\end{align}
It can be shown that, given the $X_{\perp}$ dependent representations $Z_i = XW_i + X_{\perp}K_i, i=1,2$, the metric can be 
calculated as 
\begin{align*}
  g_c(Z_1,Z_2) &= \frac{1}{2} \tr(W_1^TW_2) + \tr(K_1^TK_2).
\end{align*}
Note that $\St(n,N)$ together with the canonical metric is still a~submanifold, but not a~Riemannian submanifold. 
Thus, calculating the gradient requires a~different method than simply applying the orthogonal projection. 
We calculate for $X \in \St(n,N)$ and $Z \in T_X\St(n,N)$:
\begin{align*}
  0 
  &= \langle\nabla \overline{f}(X),Z \rangle_e - \langle\grad_c f(X),Z \rangle_c \\
  &= \tr(Z^T\nabla \overline{f}(X)) - \tr\left(Z^T \left(I_N - \frac{1}{2}XX^T \right) \grad_c f(X)\right) \\
  &= \tr\left(Z^T\left(\nabla \overline{f}(X) - \left(I_N - \frac{1}{2}XX^T \right) \grad_c f(X)\right) \right).
\end{align*}
As this holds for all $Z \in T_X\St(n,N)$, there exists a~matrix $S \in \S_{sym}(n)$ such that
\begin{align*}
  \nabla \overline{f}(X) - \left(I_N - \frac{1}{2}XX^T \right) \grad_c f(X) &= XS.
\end{align*}
Thus, the Riemannian gradient with respect to the canonical metric is the unique element in $T_X\St(n,N)$ that fulfills 
\begin{align*}
  X^T \nabla \overline{f}(X) - S &= \frac{1}{2}X^T \grad_c f(X).
\end{align*}
This leads to 
\begin{align} \label{sec:manifold_env:align_name:compactstiefelcanonicalriemanniangradient}
  \grad_c f(X) = \nabla \overline{f}(X)- X\nabla \overline{f}(X)^T X.
\end{align}
The flops to compute the Riemannian gradient are $\O(Nn^2)$.

\subsubsubsection{Cayley retraction}
In Section~\ref{sec:manifolds_subsubsubsec:orthogonalgroupcayley} we used the Cayley transform to derive the Cayley 
retraction on $\Orth(N)$. We will use this transformation to retrieve a~retraction on the compact Stiefel manifold.
First, recall that $Z \mapsto \cay\left(\frac{1}{2}Z\right)$ 
defines the Cayley retraction on $\Orth(N)$ at $I_N$. 
Now, as shown in \cite[(4)-(7)]{feasiblemethodortho:wen2013}, for each $X~\in~\St(n,N)$ the function 
\begin{align} \label{sec:manifold_env:align_name:compactstiefelcayleyretraction}
  R^{\cay}_X: T_X \St(n,N) &\to \St(n,N) \\
  Z &\mapsto \cay\left(\frac{1}{2} A_{X,Z} \right) X  \nonumber
\end{align}
with 
\begin{align}\label{sec:manifold_env:align_name:compactstiefelAXZ}
 A_{X,Z} := \left( I_N - \frac{1}{2}XX^T \right) ZX^T - XZ^T \left( I_N - \frac{1}{2}XX^T \right)
\end{align}
defines a~retraction on $\St(n,N)$, called \textit{Cayley retraction on $\St(n,N)$}.
Similar to Section~\ref{sec:manifolds_subsubsubsec:orthogonalgroupcayley}, we will use the SMW formula
for a~numerically efficient computation.
Indeed, if we set
\begin{align} \label{sec:manifold_env:align_name:definitionUV}
  U:&= 
    \begin{bmatrix}
    (I_N - \frac{1}{2}XX^T)Z + \frac{1}{2}XZ^TX & -X 
    \end{bmatrix}  
    = 
    \begin{bmatrix}
    Z - \frac{1}{2}X\left(X^TZ - Z^TX\right) & -X 
    \end{bmatrix} \in \R^{N\times 2n}, \\
  V:&=
    \begin{bmatrix}
    X^T  \\
    Z^T  \\
    \end{bmatrix} \in \R^{2n \times N},
\end{align}
we obtain $A_{X,Z} = UV$. We now can apply the SMW formula. According to \eqref{sec:manifold_env:align_name:orthogonalgroupSMWformulacayleywithrightmultiplication}, we get
\begin{align}\label{sec:manifold_env:align_name:compactstiefelSMWformulacayleywithrightmultiplication}
  R_X^{\cay}(Z)
  &=\cay\left( \frac{1}{2} A_{X,Z} \right) X 
   \\ &= 
   \left(X + \frac{1}{2}U (VX) \right) + \frac{1}{2}U \left( \left( I_{2n} - \frac{1}{2}VU \right)^{-1} \left( VX + \frac{1}{2}VU(VX)  \right) \right). \nonumber
\end{align}
The Cayley retraction can be computed in $\O(Nn^2)$.

\subsubsubsection{Vector transport along Cayley} \label{sec:manifold_subsubsubsec:vectortransport}
As $\St(n,N)$ is an~embedded submanifold of the Euclidean space $\R^{N \times n}$, we can apply Lemma~\ref{sec:manifold_env:lemma_name:manifoldtheorysubvectransport} 
to obtain a~submanifold vector transport along the Cayley retraction on $\St(n,N)$. Let ${X \in \St(n,N)}$ and $Z,Y \in T_X\St(n,N)$.
A~vector transport of $Y$ along $R^{\cay}_X$ in direction $Z$ is given by
\begin{align} \label{sec:manifold_env:align_name:compactstiefelvectranssub}
    \T^{\cay,\sub}_Z(Y) 
    &= \P_{R^{\cay}_X(Z)}(Y) \\
    &= \left(I_N - R^{\cay}_X(Z)R^{\cay}_X(Z)^T\right)Y + R^{\cay}_X(Z) \skew\left(R^{\cay}_X(Z)^TY\right) \nonumber \\
    &= Y - \frac{1}{2}R^{\cay}_X(Z)\left(R^{\cay}_X(Z)^TY + Y^TR^{\cay}_X(Z)\right). \nonumber 
\end{align}
We have $\O(Nn^2)$ flops for the computation of $\T^{\cay,\sub}_Z(Y)$.

Applying Lemma~\ref{sec:manifold_env:lemma_name:manifoldtheorydiffvectransport} and \eqref{sec:manifold_env:align_name:orthogonalgroupDcay}, we can derive another vector 
transport along $R^{\cay}_X$ by differentiating the Cayley retraction. Transporting $Y \in T_X \St(n,N)$ in direction $Z$ 
(along $R^{\cay}_X$) is given by
\begin{align} \label{sec:manifold_env:align_name:compactstiefelvectransdiff}
  \T^{\cay,\diff}_Z(Y) 
  &= DR^{\cay}_X(Z)[Y] \\ 
  &= \frac{d}{dt}R^{\cay}_X(Z + tY)\bigg|_{t=0} \nonumber \\
  &= \frac{d}{dt}\cay\left( \frac{1}{2}A_{X,Z + tY}\right) X \bigg|_{t=0} \nonumber \\
  &= \frac{d}{dt}\cay\left( \frac{1}{2}A_{X,Z + tY}\right) \bigg|_{t=0} X  \nonumber \\
  &= D\cay\left(\frac{1}{2}A_{X,Z}\right) \left[\frac{d}{dt}\left(\frac{1}{2}A_{X,Z + tY}\right)\bigg|_{t=0} \right] X \nonumber \\
  &= D\cay\left(\frac{1}{2}A_{X,Z}\right) \left[\frac{1}{2}A_{X,\frac{d}{dt}Z + tY\big|_{t=0}} \right] X \nonumber \\
  &= 2\left(I_N - \frac{1}{2}A_{X,Z} \right)^{-1}\frac{1}{2}A_{X,Y} \left( I_N - \frac{1}{2}A_{X,Z} \right)^{-1} X \nonumber \\
  &= \left(I_N - \frac{1}{2}A_{X,Z} \right)^{-1}A_{X,Y} \left( I_N - \frac{1}{2}A_{X,Z} \right)^{-1} X. \nonumber
\end{align}
We use the same decomposition as in \eqref{sec:manifold_env:align_name:definitionUV} with $A_{X,Z} = UV$ and $A_{X,Y} = U_YV_Y$.
Once again applying the SMW formula yields
{\small\begin{align}\label{sec:manifold_env:align_name:compactstiefelvectransdiffSMW}
  \T^{\cay,\diff}_Z(Y) 
  &= \left(I_N + \frac{1}{2}U \left( I_{2n} - \frac{1}{2}VU \right)^{-1}V \right) U_YV_Y \left(I_N + \frac{1}{2}U \left( I_{2n} - \frac{1}{2}VU \right)^{-1}V \right)X  \\
  &=  \left(U_Y + \frac{1}{2}U \left(\left( I_{2n} - \frac{1}{2}VU \right)^{-1}(VU_Y)\right)\right)  \left(V_Y\left(X + \frac{1}{2}U \left(\left( I_{2n} - \frac{1}{2}VU \right)^{-1}(VX)\right) \right)\right). \nonumber
\end{align}}
The computational costs are again $\O(Nn^2)$.

\subsubsubsection{Homogeneous space as a~link to the orthogonal group} \label{sec:manifold_subsubsubsec:compactstiefelhomogeneousspace}
In Section~\ref{sec:manifolds_subsubsubsec:orthogonalgroupliegroup}, we have seen that the orthogonal group is a~Lie group. 
This allows us to define a~transitive smooth group action
\begin{align} \label{sec:manifold_env:align_name:compactstiefelgroupaction}
    \Orth(N) \times \St(n,N) &\to \St(n,N) \\
    (A,X) &\mapsto AX \nonumber
\end{align}
on $St(n,N)$ by simple left-multiplication of matrices in $\Orth(N)$. Recall that transitive means 
$${\Orth(N)\cdot X = \St(n,N)} \quad \text{ for all } X \in \St(n,N).$$ 
This turns the tuple $(\Orth(N),\St(n,N))$ into a~homogeneous space, see Definition~\ref{sec:manifold_env:definition_name:homogeneousspace}.
Given a~retraction $R^{\Orth(N)}_{I_N}$ on $\Orth(N)$ at the identity $I_N$, we can now construct a~retraction $R^{\St(n,N)}_X$ on $\St(n,N)$ at an~arbitrary point $X$.   
Due to Theorem~\ref{sec:manifold_env:theorem_name:derivedretraction}, we have
\begin{align} \label{sec:manifold_env:align:compactstiefelderivedretraction}
    R_X^{\St(n,N)}: T_X\St(n,N) &\to \St(n,N) \\ 
    Z &\mapsto \lambda_E(X) R_{I_N}^{\Orth(N)} \left( \lambda_E(X)^{-1} \Omega_X(Z)\lambda_E(X) \right)E. \nonumber
\end{align}
The difficulty here is to compute the lift $\Omega_X$ and to make a~reasonable choice of a~distinct element $E \in \St(n,N)$ that makes it easy to find a~section $\lambda_E$.
According to \cite{adam:brantner2023}, the lift at an~arbitrary $X \in \St(n,N)$ is given by
\begin{align} \label{sec:manifold_env:align:compactstiefellift}
\Omega_{X}(Z) = \left( I_N - \frac{1}{2}X X^T \right)ZX^T - X Z^T \left( I_N - \frac{1}{2}X X^T \right).
\end{align}
As we need to calculate the matrix-product of an~$N \times n$ and an~$n \times N$ matrix, this costs $\O(N^2n)$ flops.
Analogous to \cite{adam:brantner2023}, we choose
\begin{align} \label{sec:manifold_env:align:compactstiefeldistinctelement}
  E = \begin{bmatrix}
        I_n \\
        0 \\
      \end{bmatrix} 
      \in \R^{N \times n}
\end{align}
as our distinct element. This choice makes it quite easy to find a~section. We can simply apply an~orthogonal extension to the elements $X \in \St(n,N)$:
\begin{align} \label{sec:manifold_env:align:compactstiefelsection}
  \lambda_E(X) := [X,\bar{\lambda}] \in \Orth(N),
\end{align}
where $\bar{\lambda}$ denotes the orthogonal extension of $X$ to an~orthogonal matrix. 
To ensure that the retraction \eqref{sec:manifold_env:align:compactstiefelderivedretraction} is a~smooth map, we need to make a~smooth 
extension-choice. A~common way to do this is to calculate a~$QR$-decomposition as shown in Algorithm~\ref{sec:manifold_env:algorithm:compactstiefelalgosection}.
To our knowledge, if the matrix $A$ is not sparse, this results in costs of $\O(N(N-n)^2)$. 
\begin{algorithm}
\begin{algorithmic}[1]
  \State $A \gets \rand (N, N-n)$ {\small\Comment{Sample $A$ from a~given distribution.}}
  \State $A \gets A - XX^TA$ {\small\Comment{Remove part of $A$ that is spanned by the columns of $X$.}}
  \State $Q,R \gets \text{qr}(A)$ {\small\Comment{Apply a~$QR$ decomposition.}}
  \State $\lambda_E(X) \gets [X,Q[1:N,1:(N-n)]]$ {\small\Comment{Output $X$ and the first $(N-n)$ columns of $Q$.}}
\end{algorithmic}
    \caption{Computation of the section $X \mapsto \lambda_E(X) \in \Orth(N)$ with $QR$-decomposition similar to \cite{adam:brantner2023}.}
    \label{sec:manifold_env:algorithm:compactstiefelalgosection}
\end{algorithm}
The product 
\begin{align} \label{sec:manifold_env:align_name:compactstiefellift+sectionproduct}
 \lambda_E(X)^{-1} \Omega_X(Z)\lambda_E(X) 
 &=
         \begin{bmatrix}
          X^TZ & -Z^T\bar{\lambda} \\
          \bar{\lambda}^TZ & 0 \\
         \end{bmatrix}
\end{align}
does not require the inverse $\lambda_E(X)^{-1}$. But, as it includes the multiplication of an~$(N-n)\times N$ matrix with an~$N \times n$ matrix,
it causes costs of $\O(N(N-n)n)$. 
The matrix in \eqref{sec:manifold_env:align_name:compactstiefellift+sectionproduct} has rank $n$.  
If we choose $R_{I_N}^{\Orth(N)} := R^{\cay}_{I_N}$ as the Cayley retraction \eqref{sec:manifold_env:align_name:orthogonalgroupcayleyretraction}, we can compute 
\begin{align} \label{sec:manifold_env:align_name:compactstiefelhomogeneouscayleyE}
  R_{I_N}^{\cay} (Z)E
\end{align}
in $\O(Nn^2)$ as shown in \eqref{sec:manifold_env:align_name:orthogonalgroupSMWformulacayleywithrightmultiplication}.

Both the lift and the section remain the numerically expensive parts of this derived retraction. 
As long as there is no way to make these computations cheaper (i.e. reducing the computational costs to $\O(Nn^2)$), both the calculation of $\Omega_X$ and $\lambda_E$ 
are critical in terms of numerical efficiency.
As we will see later, this will cause performance problems for applications with a~large dimension~$N$ and a~small dimension $n \ll N$.

\subsubsubsection{Overview}
We summarize all the important results concerning the compact Stiefel manifold in Table~\ref{sec:manifold_env:table_name:compactstiefeltable}. 
Note that we did not always display the most efficient representations, and therefore they may not match the computational cost shown.
\begin{table}[ht!]
\centering
\begin{tabular}{||c c c||} 
 \hline
 Name & Formula & Costs \\ [0.5ex] 
 \hline\hline
 $T_X\St(n,N)$          & $\{XW + X_{\perp}K \; \big| \; W \in \S_{\skew}(n), K \in \R^{(N-n) \times n}\}$                                                   & ---                             \\
 $g_e(Z_1,Z_2)$         & $\tr(Z_1^T Z_2)$                                                                                                         & $\O(Nn)$                            \\
 $\P_X(Y)$                & $(I_N - XX^T)Y + X \skew(X^TY)$                                                                                          & $\O(Nn^2)$                      \\
 $\P^{\perp}_X(Y)$        & $X\sym(X^TY)$                                                                                                            & $\O(Nn^2)$                      \\
 $\grad_e f(X)$         & $\nabla \overline{f}(X) - \frac{1}{2} X\left(X^T\nabla \overline{f}(X) + \nabla \overline{f}(X)^TX\right)$   & $\O(Nn^2)$                      \\
 $g_c(Z_1,Z_2)$         & $\tr\left(Z_1^T \left(I_N - \frac{1}{2}XX^T \right) Z_2\right)$                                                          & $\O(Nn^2)$                             \\
 $\grad_c f(X)$         & $\nabla \overline{f}(X)- X\nabla \overline{f}(X)^T X$                                                                    & $\O(Nn^2)$                      \\
 $R^{\cay}_{X}(Z)$      & $\left(I_N-\frac{1}{2}A_{X,Z}\right)^{-1}\left(I_N+\frac{1}{2}A_{X,Z}\right)X$                                           & $\O(Nn^2)$                      \\
 $\T^{\cay,\sub}_Z(Y)$  & $Y - \frac{1}{2}R^{\cay}_X(Z)\left(R^{\cay}_X(Z)^TY + Y^TR^{\cay}_X(Z)\right)$                                           & $\O(Nn^2)$                      \\
 $\T^{\cay,\diff}_Z(Y)$ & $\left(I_N - \frac{1}{2}A_{X,Z} \right)^{-1}A_{X,Y} \left( I_N - \frac{1}{2}A_{X,Z} \right)^{-1} X$                      & $\O(Nn^2)$                      \\
 $\Omega_{X}(Z)$        & $\left( I_N - \frac{1}{2}X X^T \right)ZX^T - X Z^T \left( I_N - \frac{1}{2}X X^T \right)$                                & $\O(N^2n)$                      \\
 $\lambda_E(X)$         & $[X,\bar{\lambda}]$                                                                                                      & $\O(N(N-n)^2)$                  \\
 $R_X^{\St(n,N)}$       & $\lambda_E(X) R_{I_N}^{\Orth(N)} \left( \lambda_E(X)^{-1} \Omega_X(Z)\lambda_E(X) \right)E$                              & $\O(N(N-n)^2 + N^2n)$       \\
 \hline
\end{tabular}
\caption{Geometric concepts for the compact Stiefel manifold with $A_{X,Z}, A_{X,Y}$ defined in \eqref{sec:manifold_env:align_name:compactstiefelAXZ}.}
\label{sec:manifold_env:table_name:compactstiefeltable}
\end{table}

\subsubsection{Symplectic Stiefel manifold}\label{sec:manifolds_subsubsec:symplecticstiefel}

Another manifold that is closely related to the compact Stiefel manifold is the set of all \textit{symplectic} matrices
\begin{align} \label{sec:manifolds_env:align_name:setofsymplecticstiefelmatrices}
  \Sp(2n,2N) &:= \{X \in \R^{2N \times 2n} \; \big| \; X^T J_{2N}X = J_{2n} \},
\end{align}
where $J_{2N} = 
\left[
\begin{smallmatrix}
  0 & I_N \\
  -I_N & 0 \\
\end{smallmatrix}
\right]$
is the so-called \textit{Poisson matrix}.
Similar to Section~\ref{sec:manifolds_subsubsec:orthogonalgroup}, we use the submersion theorem on the map 
\begin{align*}
  F: \R^{2N \times 2n} &\to \S_{\skew(2n)} \\ 
  X &\mapsto X^TJ_{2N}X
\end{align*}
to show that $\Sp(2n,2N) = F^{-1}(J_{2n})$ is a~closed submanifold of $\R^{2N \times 2n}$. It is called the \textit{symplectic Stiefel manifold}. 
Unlike the compact Stiefel manifold, the symplectic Stiefel manifold is not compact and thus, not bounded. 
This leads to an~increase of complexity in manifold optimization as we have an~unbounded search space.
We can turn $\Sp(2n,2N)$ into a~Riemannian submanifold using the Euclidean metric. Similar to the compact Stiefel manifold, we can also define an~alternative 
metric called the \textit{canonical-like metric}.
We will not go into this in detail here. A~good overview of the most important results can be found in \cite{symplstiefel:stykel2021}. 

\newpage
\subsubsubsection{A compact subset isomorphic to the compact Stiefel manifold}\label{sec:manifold_subsubsubsec:immersionofcompactstiefelinsymplectic}
Finally, we discuss the relation between the compact Stiefel manifold $\St(n,N)$ and the symplectic Stiefel manifold $\Sp(2n,2N)$.
For this, we define the map
\begin{align} \label{sec:manifold_env:align_name:symplecticstiefelimmersionofcompactstiefel}
 \iota: \St(n,N) &\to \Sp(2n,2N) \\
 X &\mapsto 
 \begin{bmatrix}
  X & 0 \\
  0 & X \\
 \end{bmatrix}. \nonumber
\end{align}
Simple calculation shows that 
$\left[
\begin{smallmatrix}
  X & 0 \\
  0 & X \\
\end{smallmatrix}
\right]$
is indeed symplectic and thus $\iota$ is well-defined. It follows from the definition that $\iota$ is injective. Thus, we have some kind of inclusion of 
$\St(n,N)$ in $\Sp(2n,2N)$. As $\iota$ is a~linear map (viewed as a~map of the embedding spaces), it is a~smooth map and even more an~\textit{immersion} 
between manifolds. 
This means, we can identify the compact Stiefel manifold as some immersed (and compact) subset of the symplectic Stiefel manifold. 
This is especially useful for the optimization on $\Sp(2n,2N)$. As we have an~unbounded manifold, we have the risk of drifting too far from 
the optimal solution. Instead of searching the whole symplectic space, we can apply the optimization search on the compact immersion $\St(n,N)$.
Note that we also have 
$\left[\begin{smallmatrix}
 X & 0 \\
 0 & X \\ 
\end{smallmatrix}\right] \in \St(2n,2N)$
and thus $\im(\iota) \subseteq \St(2n, 2N) \cap \Sp(2n,2N)$.
We will make use of these correlations in later sections.

\newpage
\section{Model reduction of Hamiltonian systems by a~symplectic autoencoder} \label{sec:modelreduction}

In physics, we often have to solve \textit{partial differential equations} (PDEs). 
A~common approach to solve these systems is to reduce the problem to an~\textit{ordinary differential equation} (ODE) by using a~spatial discretization. 
One problem that commonly arises here is that the resulting ODE often is of very high dimension. 
This makes it numerically expensive to solve. 
There are various approaches to deal with this issue. 

In this section, we examine so-called \textit{Hamiltonian systems}, and we present a~method that allows us to solve these systems efficiently 
even for high dimensions.
First, in Section~\ref{sec:application_subsec:problemsetup}, we consider a~solution method for Hamiltonian systems that 
reduces the system dimension using encoder and decoder functions while preserving the structure of the original system.
In Section~\ref{sec:application_subsec:symplecticautoencodernetwork}, we 
present a~\textit{symplectic autoencoder network} to determine such a~pair of encoder and decoder functions.
To begin, we describe the exact network setup as it was introduced by \cite{main:brantner2023}. Then we discuss a~self-created 
modification of the update step.
Finally, in Section~\ref{sec:application_subsec:implementation}, we show different learning setups of the network that we 
have implemented in the Julia programming language.

\subsection{Structure-preserving model reduction of Hamiltonian systems} \label{sec:application_subsec:problemsetup}

The goal of this section is to find a~numerically efficient method to solve high dimensional \textit{Hamiltonian systems}.
A~Hamiltonian system is an~autonomous, i.e. time-independent, ordinary differential equation of a~specific form:
\begin{definition} \label{sec:application_env:definition_name:hamiltoniansystem}
  Let $U\subseteq \R^{2N}$ be an~open set and $f:U \to \R^{2N}$ a~Lipschitz-continuous function. 
  The autonomous system $\dot{x} = f(x)$ is called \textit{Hamiltonian system} 
  if there exists a~$C^1$-function $\H:U \to \R$ such that for all $(x_1,x_2) \in U$, the following holds
  \begin{itemize}
    \item[(i)] $\nabla_{x_1}\H(x_1,x_2) = -f_2(x_1,x_2)$,
    \item[(ii)] $\nabla_{x_2}\H(x_1,x_2) = f_1(x_1,x_2)$.
  \end{itemize}
  We call $\H$ the \textit{Hamiltonian function}.
\end{definition}
\noindent Using the Poisson matrix
$ J_{2N}:=\left[
\begin{smallmatrix}
  0 & I_N \\
  -I_N & 0 \\
\end{smallmatrix}
\right]$,
we can write the Hamiltonian system as 
$${\dot{x} = J_{2N} \nabla \H(x)}.$$ 
A~Hamiltonian system preserves energy, meaning that its \textit{flow} $\phi_t(x_0) := x(t,x_0)$ remains unchanged under the Hamiltonian $\H$: 
\begin{align} \label{sec:application_env:align_name:hamiltonianenergypreservation}
  \frac{d}{dt} \H(\phi_t(x_0)) 
  &= \nabla \H(\phi_t(x_0))^T \frac{d}{dt}\phi_t(x_0) \\
  &= \nabla \H(\phi_t(x_0))^T J_{2N} \nabla \H(\phi_t(x_0)) \nonumber\\
  &= 0. \nonumber
\end{align}
As already mentioned, the dimension $2N$ of a~Hamiltonian system can be very large, leading to high computational costs. We now derive 
a \textit{model reduction} to obtain an~ODE of lower dimension that preserves most of the structure of the full model. 
Then we solve the reduced system and use this solution to reconstruct an~approximation to the full model.
We proceed similarly to \cite{weaklysymp:buchfink2021}. 

\subsubsection{From FOM to ROM} \label{sec:appliation_subsbubsubsec:fromFOMtoROM}
First, we define the \textit{full order model} (FOM) 
\begin{align}\label{sec:application_env:align_name:FOM}
  \dot{x}_{\mu} &= J_{2N} \nabla_{x}\H (x_{\mu}; \mu) \\ 
  x_\mu(0) &= x_{\mu}^0 \nonumber
\end{align}
as a~Hamiltonian system with a~variable system parameter $\mu \in \R^p$ and an~initial condition $x_{\mu}^0$. 
This is the system of high dimension $2N \in \N$ that we want to solve efficiently.
Next, we want to derive a~reduced ODE of dimension $2n$ with $n \ll N$. 
To perform the reduction (and reconstruction) we need
\begin{itemize}
  \item an~encoder $e: \R^{2N} \to \R^{2n}$ from the full dimension to the reduced dimension and
  \item a~decoder $d: \R^{2n} \to \R^{2N}$ from the reduced dimension to the full dimension.
\end{itemize}
Our goal is to get a~reduced system
\begin{itemize}
  \item that is a~Hamiltonian system itself, i.e., that has the form $\dot{x}_{r,\mu} = J_{2n}\nabla_{x_r}\H_{r} (x_{r,\mu}; \mu) $, 
  \item whose reconstructed solution approximates that of the FOM, \\ i.e., $x_{\mu}(t) \approx \tilde{x}_{\mu}(t) := x_{\reff,\mu} + d(x_{r,\mu}(t))$ for a~\textit{reference state} $x_{\reff,\mu} \in \R^{2N}$, and
  \item whose initial value is reconstructed exactly, i.e., $x_{\mu}^0 = x_{\reff,\mu} + d(x_{r,\mu}^0)$.
\end{itemize} 
According to our conditions, we define the \textit{reduced order model} (ROM) as
\begin{align}\label{sec:application_env:align_name:ROM}
  \dot{x}_{r,\mu} &= J_{2n}\nabla_{x_r}\H_{r} (x_{r,\mu}; \mu) \\ 
  x_{r,\mu}(0) &= x_{r,\mu}^0 \nonumber
\end{align}
with 
\begin{itemize}
  \item[(i)] $\H_{r} (x_{r,\mu}; \mu) := \H(x_{\reff,\mu} + d(x_{r,\mu}); \mu)$,
  \item[(ii)] $x_{r,\mu}^0 := y$, where $y \in \R^{2n}$ can be chosen arbitrary, and
  \item[(iii)] $x_{\reff,\mu} := x_{\mu}^0 - d(x_{r,\mu}^0)$. 
\end{itemize}
This choice of the ROM ensures an~exact reconstruction of the initial value. Note that a~good choice of the initial value $x_{r,\mu}^0$
depends on the chosen encoder-decoder pair. Since both the FOM and the ROM preserve energy, we can conclude 
from \eqref{sec:application_env:align_name:hamiltonianenergypreservation} that the error in the Hamiltonian 
\begin{align*}
  \H(x_{\mu}(t); \mu) - \H(\tilde{x}_{\mu}(t); \mu) &\equiv \H(x_{\mu}(0); \mu) - \H(\tilde{x}_{\mu}(0); \mu) = \H(x_{\mu}^0; \mu) - \H(x_{\mu}^0; \mu)
\end{align*}
vanishes for all $t$. Algorithm~\ref{sec:application_env:algorithm_name:FOMtoROM} displays a~high level overview of all steps 
required to determine an~approximate solution of the FOM by a~symplectic model order reduction.
\begin{algorithm}
  \textbf{Input:} FOM $\quad \dot{x}_{\mu} = J_{2N} \nabla_{x}\H (x_{\mu}; \mu), \quad x_\mu(0) = x_{\mu}^0.$ \\
  \textbf{Output:} Approximate solution $\tilde{x}_{\mu}(t)$ of the FOM. \\
  \textbf{Steps:}
\begin{algorithmic}[1]
  \State Determine an encoder-decoder pair $e:\R^{2N} \to \R^{2n}$ and $d:\R^{2n} \to \R^{2N}$.
  \State Choose an~initial value for the ROM $x_{r,\mu}^0 := y \in \R^{2n}$.
  \State Compute the reference state $x_{\reff,\mu} = x_{\mu}^0 - d(x_{r,\mu}^0) = x_{\mu}^0 - d(y)$.
  \State Compute the Jacobian $\nabla_{x_r} \H_{r} (x_{r,\mu}; \mu) = (D_{x_{r,\mu}}d(x_{r,\mu}))^T \nabla_x \H(x_{\reff,\mu} + d(x_{r,\mu}); \mu)$ of the reduced Hamiltonian.
  \State Solve the ROM $\quad \dot{x}_{r,\mu} = J_{2n}\nabla_{x_r}\H_{r} (x_{r,\mu}; \mu), \quad x_{r,\mu}(0) = x_{r,\mu}^0 = y$.
  \State Reconstruct an~approximate solution of the FOM $x_{\mu}(t) \approx \tilde{x}_{\mu}(t) = x_{\reff,\mu} + d(x_{r,\mu}(t))$.
\end{algorithmic}
    \caption{High level steps to solve a~Hamiltonian system by symplectic model reduction.}
    \label{sec:application_env:algorithm_name:FOMtoROM}
\end{algorithm}

\subsubsection{Choice of an~encoder-decoder pair}
So far we get an~exact reconstruction of the initial value. 
However, we still do not know how good the reconstruction of the whole solution is.
This heavily depends on the choice of the encoder and decoder functions.
The \textit{residual} 
\begin{align*}
  r(t) 
  :&= \dot{\tilde{x}}_{\mu}(t) - J_{2N} \nabla_x \H(\tilde{x}_{\mu}(t); \mu) \\
  &= D_{x_{r,\mu}}d(x_{r,\mu}(t)) \dot{x}_{r,\mu}(t) - J_{2N}\nabla_x \H(x_{\reff,\mu} + d(x_{r,\mu}); \mu) 
\end{align*}
is a~measure of the quality of our approximation. We cannot make it vanish, as the dimension reduction necessarily causes a~loss of information.
But if we assume that the decoder $d$ is a~symplectic function, the symplectic projection of the residual vanishes:
\begin{align*}
  (D_{x_{r,\mu}}d(x_{r,\mu}(t)))^{+}r(t) 
  &= \dot{x}_{r,\mu}(t) - J_{2n}(D_{x_{r,\mu}}d(x_{r,\mu}(t)))^TJ_{2N}^TJ_{2N}\nabla_x \H(x_{\reff,\mu} + d(x_{r,\mu}); \mu) \\
  &= \dot{x}_{r,\mu} - J_{2n} \nabla_{x_r}\H_{r} (x_{r,\mu}; \mu) \nonumber \\ 
  &= 0. \nonumber
\end{align*}
Here, we used that the symplectic inverse $(D_{x_{r,\mu}}d(x_{r,\mu}(t)))^{+}$ equals $J_{2n}(D_{x_{r,\mu}}d(x_{r,\mu}(t)))^TJ_{2N}^T$.
This indicates that symplectic encoder-decoder pairs are a~good choice. 

To get a~good approximation of the FOM, applying the composition of the encoder and decoder should roughly deliver the identity
\begin{align*}
  d(e(x_{\mu}(t))) &\approx x_{\mu}(t)
\end{align*}
for all solutions $x_{\mu}(t)$. Thus, we want to minimize the error
\begin{align}\label{sec:application_env:align_name:errorfunction}
  \L(e,d) &= \frac{1}{2N|\X|} \sum_{x \in \X} \norm{x - d(e(x))}^2,
\end{align}
where $\X$ is a~finite set of solution snapshots and $\norm{\cdot}$ is any norm on $\R^{2N}$.
We scale it by the inverse of the system dimension $2N$ and the number of snapshots $|\X|$.
In summary, the general solution to the model reduction problem can be determined by solving the optimization problem
\begin{align}\label{sec:application_env:align_name:encoderdecoderminprob}
  \min \quad &\L(e,d) \\
  \st \quad &e,d \text{ symplectic, i.e. } \nonumber \\ 
  & D_{x_{r,\mu}}d(x_{r,\mu})^T J_{2N} D_{x_{r,\mu}}d(x_{r,\mu}) = J_{2n} \nonumber \\
  & D_{x_{\mu}}e(x_{\mu}) J_{2N} D_{x_{r,\mu}}e(x_{\mu})^T = J_{2n}. \nonumber
\end{align}

\subsubsection{Approaches to solve the optimization problem} \label{sec:application_subsubsubsec:approachestosolvetheoptprob}
Given a~finite set of snapshots $\X$, \eqref{sec:application_env:align_name:encoderdecoderminprob} is a~well-defined optimization problem.
In practice, however, it is often not possible to solve the problem exactly. 
Usually, the set of all admissible functions is further restricted.

The simplest approach for computing an~encoder-decoder pair is 
to find the best linear solution of \eqref{sec:application_env:align_name:encoderdecoderminprob} for a~given snapshot matrix $M \in \R^{2N \times k}$. 
We can formulate this problem as
\begin{align}\label{sec:application_env:align_name:PSDProblem}
  \underset{A \in \Sp(2n,2N)}{\min} \quad &\L(A^{+},A) =  \norm{M - AA^{+}M}_F.
\end{align} 
Due to the unboundness and non-convexity of $\Sp(2n,2N)$, this is still considered to be difficult to solve, see \cite{symplstiefel:stykel2021} for details. We can simplify things 
by further restricting our search space to a~bounded set. As discussed in Section~\ref{sec:manifold_subsubsubsec:immersionofcompactstiefelinsymplectic}, we reduce it to 
$A = \left[\begin{smallmatrix}
 X & 0 \\
 0 & X \\ 
\end{smallmatrix}\right] \in
\St(2n,2N) \cap \Sp(2n,2N)$ 
with ${X \in \St(n,N)}$ an~element of the compact Stiefel manifold.
There are different methods to calculate a~suboptimal solution for this.
We apply the \textit{proper symplectic decomposition} (PSD) method. The solution to this is displayed in Algorithm~\ref{sec:application_env:algorithm_name:PSDsolution}.
Using the optimal solution $A$ of the PSD method, we can define $e:= A^{+}$ and $d:=A$ as the encoder-decoder pair for the
model order reduction in Algorithm~\ref{sec:application_env:algorithm_name:FOMtoROM}.
A~good choice for the initial value of the ROM would be $x_{r,\mu}^0 := e(x_{\mu}^0)$.
\begin{algorithm}
\begin{algorithmic}[1]
  \State $M = 
  \left[\begin{smallmatrix}
    M_1 \\
    M_2 \\
  \end{smallmatrix}\right]
  \in \R^{2N \times k} 
  \mapsto
  M_{\psd} := [M_1,M_2] \in \R^{N \times 2k}$. {\small\Comment{Rearrange $M$.}}
  \State $M_{\psd} = U\Sigma V^T$ {\small\Comment{Compute SVD of $M_{\psd}$.}}
  \State $X:= U[N,1:n]$ {\small\Comment{Extract the first $n$ columns of $U$.}}
  \State $A := 
  \left[\begin{smallmatrix}
    X & 0 \\
    0 & X \\
  \end{smallmatrix}\right]$ {\small\Comment{Return the cotangent lift constructed from $M$.}}
\end{algorithmic}
    \caption{PSD method for a~snapshot matrix $M \in \R^{2N \times k}$.}
    \label{sec:application_env:algorithm_name:PSDsolution}
\end{algorithm}

One problem with \eqref{sec:application_env:align_name:PSDProblem} is that too much information may be lost due to the restriction to linear encoder-decoder maps, 
which prevents a~good approximation. 
Therefore, in recent years, different methods have been developed to find a~nonlinear (suboptimal) solution for 
\eqref{sec:application_env:align_name:encoderdecoderminprob}. 
Here, we present a~method that uses an~artificial neural network (ANN) to learn both the encoder and the decoder while 
enforcing symplecticity constraints in every update step.
More precisely, we learn a~symplectic \textit{deep convolutional autoencoder} (DCA) network. A~DCA $\A :=(e(\cdot;\theta), d(\cdot;\theta))$ consists 
of an~encoder $e(\cdot;\theta)$ that converts the input values and a~decoder $d(\cdot;\theta)$ that reconstructs the input values. 
The functions are parametrized by \textit{weights} $\theta$ that are updated in the learning process. 
The \textit{target function} is designed to learn the identity map $d \circ e \equiv \id$ on some \textit{training set} $\X$ while being restricted to conditions that only allow 
an~approximate reconstruction. Since the target values are given by the inputs, we speak of \textit{unsupervised learning}. 
In accordance with \eqref{sec:application_env:align_name:errorfunction}, we choose the \textit{loss function} for a~given 
target function $f$ and a~batch of training data $\X$ as 
\begin{align}\label{sec:application_env:align_name:nnlossfunction}
  \loss(\X, f) := \frac{1}{2N|\X|}\sum_{x \in \X}\norm{x - f(x)}^2_2.
\end{align}
Given a~matrix \textit{batch} $\X_b \in \R^{2N \times k}$ of $k$ snapshots and 
a target function output matrix $\Y_b := f(\X_b) \in \R^{2N \times k}$, the loss function turns into 
\begin{align}\label{sec:application_env:align_name:nnlossfunctionformatrixbatch}
  \loss(\X_b, \Y_b) 
  &= \frac{1}{2N|\X_b|}\sum_{x \in \X_b}\norm{x - f(x)}^2_2 \\
  &= \frac{1}{2Nk} \norm{\X_b - \Y_b}_F^2. \nonumber
\end{align}
Thus, the loss function is computed as the Frobenius norm distance of the input and the output normalized by the system dimension.
In our case, the \textit{error function} for the minimization problem \eqref{sec:application_env:align_name:encoderdecoderminprob} is
given as the composition of the target function $f = d \circ e$ and the loss function. For a~snapshot batch $\X_b$ this is 
\begin{align}\label{sec:application_env:align_name:nnerrorfunction}
  \L(\theta,\X_b) := \frac{1}{2N|\X_b|}\norm{\X_b - d(e(\X_b;\theta);\theta)}_F^2.
\end{align}
According to \cite{weaklysymp:buchfink2021} the training data should be normalized. Thus, we choose the training set as a~set of normalized snapshots 
\begin{align}\label{sec:application_env:align_name:trainingset}
   \X := \X(\P_{\text{train}}) := \left\{ x^k_\mu - x^0_\mu \; \big| \; 0 \leq k \leq K, \mu \in \P_{\text{train}} \right\},
\end{align}
where $K \in \N$ denotes the number of discrete time steps chosen and $\P_{\text{train}}$ a~selection of variable system parameter values~$\mu$.
For a~given time interval $\I=[t_0,t_1]$, $x^k_\mu := x(t_0 + k\tfrac{t_1-t_0}{K};\mu)$ denotes the exact solution 
of the FOM at time $t_k=t_0 + k\tfrac{t_1-t_0}{K}$ for a~chosen system parameter $\mu \in \P_{\text{train}}$.
Since we use normalized data and thus $0 \in \X$, we train the target function to get $(d\circ e)(0) \approx 0$. This makes 
$x_{r,\mu}^0:= e(0)$ a~good choice for the initial value of the ROM when solving the FOM with Algorithm~\ref{sec:application_env:algorithm_name:FOMtoROM}.
To ensure the symplecticity of the encoder-decoder, the network layers and update steps must be chosen wisely. In the next section, we will discuss such a~network setup in detail.

\subsection{A symplectic autoencoder network}\label{sec:application_subsec:symplecticautoencodernetwork}

In this section, we discuss the construction of a~symplectic autoencoder network to obtain an~encoder-decoder pair 
for a~suboptimal solution of the minimization problem \eqref{sec:application_env:align_name:encoderdecoderminprob}.
The network we present was first introduced in \cite{main:brantner2023}. 
We describe the exact architecture and the individual layers of this network. We focus, in particular, on the required manifold optimization steps. 
Furthermore, we discuss some network modifications we carried out to improve the performance of the network. 
In particular, we introduce a~modified Adam optimizer to work on the compact Stiefel manifold instead of the homogeneous space as proposed in \cite{main:brantner2023}.

\subsubsection{General network design}

The general goal of our autoencoder network is to optimize the network parameters $\theta=(\theta_1,...,\theta_l)$ (meaning the tuple of all layer specific parameters $\theta_i$) 
so that in the end, we have a~pair of symplectic encoder and decoder whose composition $d \circ e$ approximates the identity map for the elements of the 
training set $\X$. This requires an~update mechanism. 
One update step of an~autoencoder network consists of four main stages:
\begin{itemize}
  \item[1.] Choose a~snapshot batch $\X_b \subseteq \X$.
  \item[2.] Calculate the output $\L(\theta,\X_b)$ of the error function (\textit{forward propagation}).
  \item[3.] Calculate the gradient $\nabla_{\theta}\L(\theta,\X_b)$ of the error function with respect to the network 
  parameters $\theta$ (\textit{backward propagation}). Here, it is important to store each of the \enquote{sub-gradients} $\nabla_{\theta_i}\L(\theta,\X_b)$ 
  of the layer-specific parameters $\theta_i$ of $\theta$. 
  \item[4.] Use the layer-specific \enquote{sub-gradients} $\nabla_{\theta_i}\L(\theta,\X_b)$ to perform the parameter update 
  for each layer~$i$.
\end{itemize}
The main idea of this update mechanism is the same as for the gradient descent method described in Section~\ref{sec:manifolds_subsec:optimizationtheory}: 
Use the negative gradient as the steepest descent direction to perform an~update step on the error function $\L$. 

The first three steps are very similar for most common neural networks. 
There are optimized standard mechanisms for an~efficient calculation, particularly for the forward and backward propagation steps.
Step four is more interesting. The update mechanism heavily depends on the architecture of the network layers. In our case, we must 
design the layers and the update step in a~way that ensures symplecticity for our target function. 
Next, we describe such an~architecture and how the update step for each layer is performed.

\subsubsection{Architecture}\label{sec:application_subsubsec:architecture}
The network introduced in \cite{main:brantner2023} consists of a~combination of two main layer types: \textit{GradientLayer} and \textit{PSDLayer}. 
The single network layers are chained to build the whole network. 
First, four full-dimensional GradientLayers are applied, followed by a~reduction layer realized by one PSDLayer. Together they build the encoder $e(\cdot;\theta)$. 
The decoder $d(\cdot;\theta)$ consists of two low-dimensional GradientLayers, followed by the reconstruction layer again realized 
by one PSDLayer and finally one more GradientLayer. 
As we will see, both the GradientLayer and the PSDLayer consist of symplectic functions. 
Since the composition of symplectic functions is again symplectic, we get a~network that learns a~strictly symplectic function. The difficulty 
lies in designing the parameter update in a~way that preserves the symplecticity of the layer functions. 
The network setup is displayed in Figure~\ref{sec:application_env:tikzpicture_name:network}. Next, we discuss the setup of each layer type.
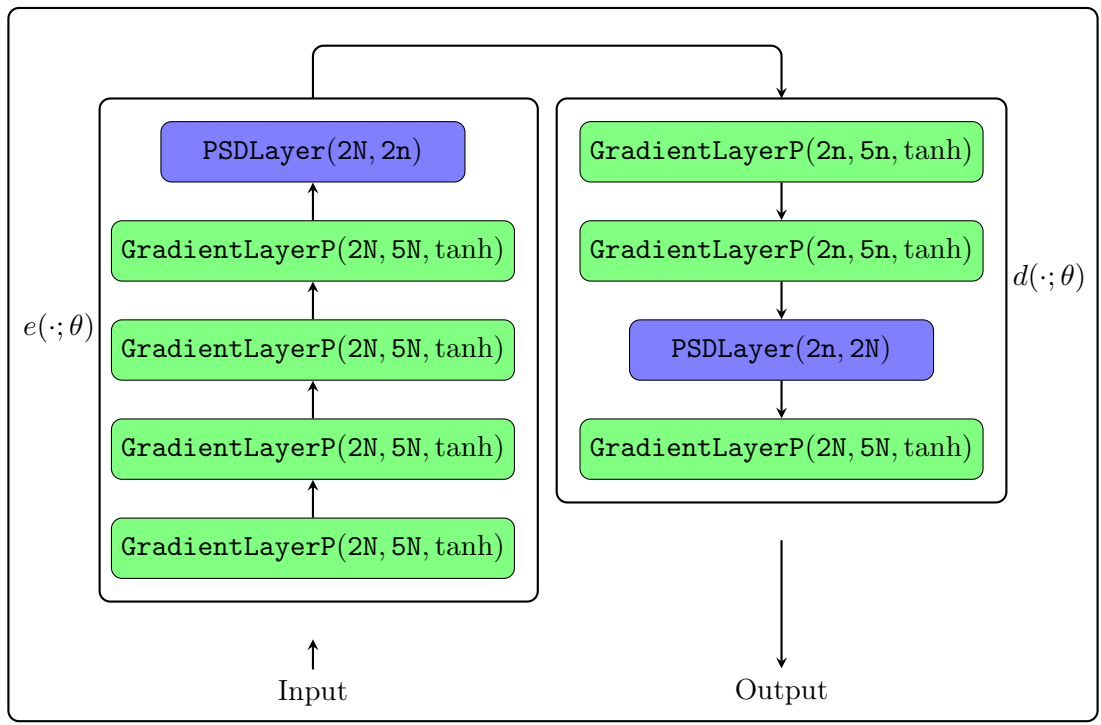
\begin{figure}[ht]
    \centering
    \begin{tikzpicture}[scale=1]
        \tikzstyle{button} = [rectangle, rounded corners, minimum width=4cm, minimum height=0.8cm, text centered, draw=black, fill=green!50]
        \tikzstyle{bluebutton} = [rectangle, rounded corners, minimum width=4cm, minimum height=0.8cm, text centered, draw=black, fill=blue!50]
        \tikzstyle{arrow} = [thick, ->, >=stealth]
        \node[bluebutton] (btn1) {$\tt PSDLayer(2N,2n)$};
        \node[button, below=0.5cm of btn1] (btn2) {$\tt GradientLayerP(2N,5N,\tanh)$};
        \node[button, below=0.5cm of btn2] (btn3) {$\tt GradientLayerP(2N,5N,\tanh)$};
        \node[button, below=0.5cm of btn3] (btn4) {$\tt GradientLayerP(2N,5N,\tanh)$};
        \node[button, below=0.5cm of btn4] (btn5) {$\tt GradientLayerP(2N,5N,\tanh)$};
        \node[button, right=1.5cm of btn1] (btn6) {$\tt GradientLayerP(2n,5n,\tanh)$};
        \node[button, below=0.5cm of btn6] (btn7) {$\tt GradientLayerP(2n,5n,\tanh)$};
        \node[bluebutton, below=0.5cm of btn7] (btn8) {$\tt PSDLayer(2n,2N)$};
        \node[button, below=0.5cm of btn8] (btn9) {$\tt GradientLayerP(2N,5N,\tanh)$};
        \foreach \i in {5,4,3,2} {
            \draw[arrow] (btn\i.north) -- (btn\the\numexpr\i-1.south);
        }
        \foreach \i in {6,7,8} {
            \draw[arrow] (btn\i.south) -- (btn\the\numexpr\i+1.north);
        }
        \draw[rounded corners, thick] ([shift={(-0.8,0.3)}]btn1.north west) rectangle ([shift={(0.3,-0.3)}]btn5.south east) node[midway, left=2.8cm, shift={(0, 0.3)}] {$e(\cdot;\theta)$};
        \draw[rounded corners, thick] ([shift={(-0.3,0.3)}]btn6.north west) rectangle ([shift={(0.3,-0.3)}]btn9.south east) node[midway, right=2.9cm, shift={(0, 0.3)}] {$d(\cdot;\theta)$};
        \draw[rounded corners, thick] ([shift={(-2,1.5)}]btn1.north west) rectangle ([shift={(1.5,-3.2)}]btn9.south east);
        \node[below=1.5cm of btn5, shift={(0, 0.3)}] (input) {Input};
        \node[below=2.8cm of btn9, shift={(0, 0.3)}] (output) {Output};
        \draw[arrow] (input) -- ([shift={(0,-0.8)}]btn5.south);
        \draw[arrow] ([shift={(0,-0.8)}]btn9.south) -- (output);
        \draw[arrow, rounded corners] ([shift={(0,0.3)}]btn1.north) -- ++(0,0.7) -| ([xshift=0cm, yshift=0.7cm]btn6.north) -- ([shift={(0,0.3)}]btn6.north);
    \end{tikzpicture}
    \caption{Schematic depiction of the network architecture taken from \cite[p.17]{main:brantner2023}.}
    \label{sec:application_env:tikzpicture_name:network}
\end{figure}

\newpage
\subsubsubsection{GradientLayer} \label{sec:manifold_subsubsubsec:gradientlayer}
GradientLayers are dimension-preserving symplectic maps first introduced in \cite{sympnets:pengzhan2020} and can take one of the 
following two forms:
\begin{align}
{\tt GradientLayerP(2N,L,\sigma)}: \R^{2N} &\to \R^{2N} \label{sec:application_env:align_name:functiongradientlayerP} \\ 
\begin{bmatrix}
    q \\
    p \\
\end{bmatrix}
&\mapsto
\begin{bmatrix}
    q \\
    K^T \diag(a) \sigma(Kq + b) + p \nonumber
\end{bmatrix}, \\
{\tt GradientLayerQ(2N,L,\sigma)}: \R^{2N} &\to \R^{2N} \label{sec:application_env:align_name:functiongradientlayerQ} \\ 
\begin{bmatrix}
    q \\
    p \\
\end{bmatrix}
&\mapsto
\begin{bmatrix}
    K^T \diag(a) \sigma(Kp + b) + q \\
    p \\
\end{bmatrix}, \nonumber
\end{align}
where $a,b \in \R^L$ and $K \in \R^{L \times N}$ can be chosen arbitrarily. The parameter $L \gg N$ is referred to as
the \textit{upscaling parameter}. These layer types are somewhat \enquote{half fully connected}: For the \stexttt{GradientLayerP} all inputs $x=[q^T,p^T]^T$ 
determine the output of the lower vector half $p$, while the upper vector half $q$ is only affected by one input parameter (outputting the identity).
For \stexttt{GradientLayerQ}, it is the other way round.
The symplecticity of \eqref{sec:application_env:align_name:functiongradientlayerP} and \eqref{sec:application_env:align_name:functiongradientlayerQ} 
can be shown by simply calculating the Jacobian.  
Furthermore, $\sigma$ denotes the component-wise applied \textit{activation function} operating on $\R$ and is responsible for the nonlinearity of the layer.
Most common activations are
\begin{itemize}
  \item $\relu(x) = \max\{0,x\}$, 
  \item $\explu(x) = 
        \begin{cases}
          x & x \geq 0 \\
          e^{x} - 1 & x < 0 \\ 
        \end{cases}$,
  \item $\tanh(x) = \frac{1-e^{-x}}{1+e^{-x}}$. 
\end{itemize}
As displayed in Figure~\ref{sec:application_env:tikzpicture_name:network}, we use the hyperbolic tangent $\tanh(x)$ as our activation function for all 
GradientLayers.

Since the symplecticity of \eqref{sec:application_env:align_name:functiongradientlayerP} and \eqref{sec:application_env:align_name:functiongradientlayerQ} does not impose any restriction on the parameter set $(K,a,b)$,
these layers can be trained in the classical machine learning way. The update step in iteration $(t+1)$ follows the scheme:
\begin{itemize}
 \item[1.] Compute the Euclidean gradient of the loss function \eqref{sec:application_env:align_name:nnerrorfunction}:  
 $Y_t = \nabla \L_{(K,a,b)_t}(\theta,\X_t)$, where $\X_t$ denotes the current batch.
 \item[2.] Compute an~update vector using the Euclidean gradient $Y_t$ and an~optimizer. For the \textit{gradient descent} optimizer, this basically is 
 the negative gradient scaled by some factor $\eta > 0$: \linebreak 
 $V_t = -\eta Y_t$. 
 \item[3.] Update the current parameters using the update vector: 
 $(K,a,b)_{(t+1)} \leftarrow (K,a,b)_t + V_t$.
\end{itemize} 
\begin{algorithm}
\textbf{Require:} Step size $\eta=0.001$; decay rates $\beta_1 = 0.9, \beta_2=0.99$; numerical stabilizer $\delta=10^{-8}$; cached moments $(B_1^{\ca},B_2^{\ca})$;
descent vector $Y_t$ (mostly $\nabla \L_{(K,a,b)_t}(\theta,\X_t)$). \\
\textbf{Return:} New update vector $V_t$. \\
\textbf{Step:}
\begin{algorithmic}[1]
  \State $B_1^{\ca} \leftarrow \frac{\beta_1 - \beta_1^t}{1 -  \beta_1^t}B_1^{\ca} + \frac{1 - \beta_1}{1 - \beta_1^t} Y_t$ {\small\Comment{Update $B_1^{\ca}$ using $\beta_1$ and $Y_t$.}}
  \State $B_2^{\ca} \leftarrow \frac{\beta_2 - \beta_2^t}{1 -  \beta_2^t}B_2^{\ca} + \frac{1 - \beta_2}{1 - \beta_2^t} Y_t \odot Y_t$ {\small\Comment{Update $B_2^{\ca}$ using $\beta_2$ and Hadamard product $Y_t \odot Y_t$.}}
  \State $V_t \leftarrow -\eta B_1^{\ca} \bigg/ \sqrt{B_2^{\ca} + \delta}$ {\small\Comment{Return update vector $V_t$ computed from $B_1^{\ca}$, $B_2^{\ca}$ and $\delta$.}}
\end{algorithmic}
    \caption{Euclidean Adam optimizer step in iteration $t \mapsto t+1$ similar to \cite[Algorithm~1]{main:brantner2023}.}
    \label{sec:application_env:algorithm_name:classicAdam}
\end{algorithm}
\begin{algorithm}
\begin{algorithmic}[1]
  \State $Y_t = \nabla \L_{(K,a,b)_t}(\theta,\X_t)$ {\small\Comment{Compute the Euclidean gradient.}}
  \State $V_t = \adam(Y_t)$ {\small\Comment{Compute the update vector $V_t$ using Adam.}}
  \State $(K,a,b)_{(t+1)} \leftarrow (K,a,b)_t + V_t$ {\small\Comment{Update the current weights using $V_t$.}}
\end{algorithmic}
    \caption{Complete parameter update step $t \mapsto t+1$ for a~GradientLayer using Adam optimizer.}
    \label{sec:application_env:algorithm_name:GradientLayerUpdatestep}
\end{algorithm}
This is exactly the gradient descent method for Euclidean optimization discussed in Section~\ref{sec:manifolds_subsec:optimizationtheory}.
But we are not bound to the gradient descent optimizer here. The optimizer to determine the update vector $V_t$ can be chosen arbitrarily.
The most common optimizers are \textit{stochastic gradient descent} and \textit{Adam}.
We use the Adam optimizer in all GradientLayers. One Euclidean Adam iteration, referred to as $\adam(\cdot)$, is displayed in Algorithm~\ref{sec:application_env:algorithm_name:classicAdam}. 
Analogous to \cite{main:brantner2023}, we use the values given in Algorithm~\ref{sec:application_env:algorithm_name:classicAdam} for the 
parameters $\beta_1, \beta_2$ and $\delta$. 
We set $(B_1^{\ca},B_2^{\ca}) := (0,0)$ for the initial moments.
These values can be changed, of course. The choice of parameter values is often based on empirical 
experience.
The complete update step for one GradientLayer is shown in Algorithm~\ref{sec:application_env:algorithm_name:GradientLayerUpdatestep}.

Note that we apply GradientLayers both in the reduced dimension $2n$ and the full dimension $2N$. 
The GradientLayers are meant to perform some \textit{symplectic preprocessing} before and after performing the actual dimension
reduction and reconstruction. This brings nonlinearity into the learned function.

\newpage
\subsubsubsection{PSDLayer}
PSDLayers are linear symplectic maps either upscaling or reducing the dimension. They can take the forms
\begin{align}
{ \tt PSDLayer(2n,2N)}: \R^{2n} &\to \R^{2N} \label{sec:application_env:align_name:functionpsdlayernN} \\ 
 \begin{bmatrix}
    q \\
    p \\
 \end{bmatrix}
 &\mapsto
 A~
 \begin{bmatrix}
    q \\
    p \\
 \end{bmatrix}
 :=
 \begin{bmatrix}
    X & 0 \\
    0 & X \\
 \end{bmatrix}
 \begin{bmatrix}
    q \\
    p \\
 \end{bmatrix} \text{ and} \nonumber \\ 
 {\tt PSDLayer(2N,2n)}: \R^{2N} &\to \R^{2n} \label{sec:application_env:align_name:functionpsdlayerNn} \\ 
 \begin{bmatrix}
    q \\
    p \\
 \end{bmatrix}
 &\mapsto
 A^{+} 
 \begin{bmatrix}
    q \\
    p \\
 \end{bmatrix}
 :=
 \begin{bmatrix}
    X^T & 0 \\
    0 & X^T \\
 \end{bmatrix}
 \begin{bmatrix}
    q \\
    p \\
 \end{bmatrix}, \nonumber
\end{align} 
where $X \in \St(n,N)$ is an~arbitrary element of the compact Stiefel manifold and $A$ is an~element of ${\St(2n,2N) \cap \Sp(2n, 2N)}$ like discussed
in Section~\ref{sec:manifold_subsubsubsec:immersionofcompactstiefelinsymplectic}. Despite the fact that the optimal solution being found on the symplectic Stiefel manifold $\Sp(2n,2N)$, 
we reduce the search space for this layer type to a~subspace isomorphic to the compact Stiefel manifold $\St(n,N)$. 
This has two main advantages: We cannot drift too far from the optimal solution because of the compactness. 
Second, we can optimize on the compact Stiefel manifold, which is less complex than optimizing on the symplectic Stiefel manifold.
As the notation already implies, $A^{+}=
\left[\begin{smallmatrix}
  X^T & 0 \\
  0 & X^T \\
\end{smallmatrix}\right]$
denotes the symplectic inverse of $A$.
Here, similar to GradientLayers, we have some \enquote{half fully connected} layer. 
Each output vector $q$ is influenced by the input vector $q$ but not by $p$.
For the output vector $p$, it is the other way round. 

Unlike for GradientLayers, we now have symplecticity constraints for our parameter matrix~$A$. 
Thus, we cannot simply apply the classical machine learning update steps as we must ensure
that the parameter matrix still has the form shown in \eqref{sec:application_env:align_name:functionpsdlayernN} respectively \eqref{sec:application_env:align_name:functionpsdlayerNn}.
The idea is to apply a~manifold update step on the sub-matrix $X \in \St(n,N)$ instead of the full matrix and again try to apply the Adam optimizer here.
Some problems arise here:
\begin{itemize}
  \item[(i)] We cannot use the Euclidean gradient because it is not an~element of the tangent space.
  \item[(ii)] Since the tangent spaces of $\St(n,N)$ 
  do not have the exact same structure as the Euclidean space $\R^N$, some Adam operations such as the Hadamard product $\odot$ and the dot-wise root $\sqrt{(\cdot)}$ cannot be 
  applied. In addition, we do not have a~canonical connection between different tangent spaces, which makes it difficult 
  to update moments.
  \item[(iii)] We cannot apply the parameter update step analogously to Algorithm~\ref{sec:application_env:algorithm_name:GradientLayerUpdatestep} 
  because $\St(n,N)$ is not a~Euclidean manifold.
\end{itemize}
We handle the first issue by the adaption of optimization algorithms to manifolds. 
Thereby, we use the Euclidean gradient to compute the Riemannian gradient $\grad\L_{X}(\theta,\X_b)$ 
(with respect to a~chosen metric on $\St(n,N)$) which is an~element of the tangent space $T_X\St(n,N)$. 
Note that $X \in \St(n,N)$ is the PSDLayer specific weight matrix of all network weights $\theta$.
The authors of \cite{main:brantner2023} came up with a~combined solution for the second and third problem: 
They used the fact that the Lie group $\Orth(N)$ together with $\St(n,N)$ admits the structure of a~homogeneous space 
(we discussed this connection in Section~\ref{sec:manifold_subsubsubsec:compactstiefelhomogeneousspace}).  
As done in \eqref{sec:manifold_env:align:compactstiefelderivedretraction}, this allows us to construct the retraction 
\begin{align*}
  R_X^{\St(n,N)}(Z) = \lambda_E(X) R_{I_N}^{\Orth(N)} \left( \lambda_E(X)^{-1} \Omega_X(Z)\lambda_E(X) \right)E
\end{align*}
at an~arbitrary point $X \in \St(n,N)$ from a~retraction 
$R_{I_N}^{\Orth(N)}$ on $\Orth(N)$ at the identity $I_N$. 
This retraction has the nice property that we come across the same vector space, the global tangent space representation 
$\g^{hor,E} \subseteq \g = T_{I_N}\Orth(N) = \S_{\skew}(N)$, no matter at what point $X \in \St(n,N)$ we apply the retraction
(see \eqref{sec:manifold_env:align_name:chainofmapsofderivedretractions} for details).
The idea is to now split up this retraction and use the intermediate step $\g^{hor,E}$ to apply the classical Adam optimizer from 
Algorithm~\ref{sec:application_env:algorithm_name:classicAdam}.
This yields two major advantages: We do not need to transport the moments as we always work on the same vector space $\g^{hor,E}$, and 
we can apply Adam without the need to modify it.
First, we lift the Riemannian gradient $Z:=\grad\L_{X}(\theta,\X_b)$ to the global tangent space representation via the map
\begin{align} \label{sec:application_env:align_name:lift}
  \lift(Z):=  \lambda_E(X)^{-1} \Omega_X(Z)\lambda_E(X).
\end{align}
Then we apply the classic Adam optimizer in Algorithm~\ref{sec:application_env:algorithm_name:classicAdam} to $B:= \lift(Z)$.
Here, although we do not operate on the space $\R^N$, we can indeed make all necessary computations. This is due to the fact that the Hadamard product of 
skew-symmetric matrices is a~symmetric, positive semidefinite matrix.
Now, we map the Adam output $V := \adam(B) \in \g^{hor,E}$ back to $T_X\St(n,N)$ via the map
\begin{align} \label{sec:application_env:align_name:retract}
  \retract(V) := \lambda_E(X) R_{I_N}^{\Orth(N)}(V)E.
\end{align}
In the end, we must construct the new iterate from the updated parameter matrix $X \in \St(n,N)$. The complete parameter update step is shown 
in Algorithm~\ref{sec:application_env:algorithm_name:PSDLayerUpdatestep}.  
\begin{algorithm}
\begin{algorithmic}[1]
  \State $Y_t = \nabla \L_{X_t}(\theta,\X_t)$ {\small\Comment{Compute the Euclidean gradient.}}
  \State $Z_t = \grad\L_{X_t}(\theta,\X_t)$ {\small\Comment{Compute the Riemannian gradient from $Y_t$.}}
  \State $S_t = \lambda_E(X_t)$ {\small\Comment{Compute the section $\lambda_E$ at $X_t$.}}
  \State $B_t = \lift(Z_t) = S_t^{-1} \Omega_{X_t}(Z_t)S_t$ {\small\Comment{Lift $Z_t$ to $\g^{hor,E}$.}}
  \State $V_t = \adam(B_t)$ {\small\Comment{Compute the update vector $V_t$ using Adam.}}
  \State $X_{t+1} = \retract(V_t) = S_t R_{I_N}^{\Orth(N)}(V_t)E$ {\small\Comment{Retract $V_t$ back to $\St(n,N)$.}}
  \State $A_{t+1} = \left[\begin{smallmatrix}
    X_{t+1} & 0 \\
    0 & X_{t+1} \\
  \end{smallmatrix}\right]$ {\small\Comment{Construct the new iterate $A_{t+1}$.}}
\end{algorithmic}
    \caption{Complete update step $t \mapsto t+1$ for a~PSDLayer using Adam on homogeneous spaces.}
    \label{sec:application_env:algorithm_name:PSDLayerUpdatestep}
\end{algorithm}

In conclusion, the PSDLayers are used to perform the actual dimension reduction (and reconstruction) whilst still enforcing the symplecticity
condition. The setup requires manifold optimization techniques on the compact Stiefel manifold and the orthogonal group to enforce 
symplecticity. It takes a~detour over homogeneous spaces to be able to apply the Adam optimizer always on the same 
vector space $g^{hor,E}$. This made it possible to handle moments without the need of transporting vectors between different 
tangent spaces. 
Next, we present an~optimizer approach that modifies Adam to work on different tangent spaces of $\St(n,N)$ directly 
by transporting the moments using a~vector transport.

\subsubsection{An alternative optimizer approach}\label{sec:application_subsubsec:alternativeoptimizer}

Our goal is now to apply an~Adam optimizer step on the \textit{PSDLayer} presented in \cite{main:brantner2023} without the need of 
a~detour over a~homogeneous space. To achieve this, a~first naive idea is to apply the classic Adam optimizer on the tangent space 
$T_X\St(n,N)$ for each iterate $X \in St(n,N)$
and then use vector transport to move the caches to the new iterate: 
\begin{itemize}
  \item[1.] Compute the Riemannian gradient $Z_t:=\grad\L_{X_t}(\theta,\X_t)$ with respect to the preferred metric using the Euclidean gradient $\nabla \L_{X_t}(\theta,\X_t)$.
  \item[2.] Apply the Adam optimizer on the tangent space $T_{X_t}\St(n,N)$: 
            \begin{itemize}
             \item[2.1.] $B_1^{\text{cache}} \leftarrow \frac{\beta_1 - \beta_1^t}{1 -  \beta_1^t}B_1^{\text{cache}} + \frac{1 - \beta_1}{1 - \beta_1^t} Z_t$,
             \item[2.2.] $B_2^{\text{cache}} \leftarrow \frac{\beta_2 - \beta_2^t}{1 -  \beta_2^t}B_2^{\text{cache}} + \frac{1 - \beta_2}{1 - \beta_2^t} Z_t \odot Z_t$,
             \item[2.3.] $V_t \leftarrow -\eta B_1^{\text{cache}} \bigg/ \sqrt{B_2^{\text{cache}} + \delta}$. 
            \end{itemize}
  \item[3.] Retract the update vector $V_t$ along a retraction $R_{X_t}$ to obtain the new iterate: \\$X_{t+1} \leftarrow R_{X_t}\left(V_t\right)$.
  \item[4.] Transport the cache $(B_1^{\text{cache}},B_2^{\text{cache}})$ to the new iterate tangent space $T_{X_{t+1}}\St(n,N)$ 
            in direction $V_t$ along $R_{X_t}$ using the preferred vector transport $\T_{V_t}$:
            \begin{itemize}
              \item[4.1.] $B_1^{\text{cache}} \leftarrow \T_{V_t}(B_1^{\text{cache}})$,
              \item[4.2.] $B_2^{\text{cache}} \leftarrow \T_{V_t}(B_2^{\text{cache}})$.
            \end{itemize}
\end{itemize}
Two problems arise here, that do not come up on the homogeneous space setup:
\begin{itemize}
  \item Unlike the global tangent space $\g^{hor,E} \subseteq \S_{\skew}(N)$, our tangent spaces here are not simply skew-symmetric matrices. 
  Thus, when we apply step $2.3.$, we do not remain in the current tangent space. 
  \item $B_2^{\text{cache}}$ (unlike $B_1^{\text{cache}}$) is not part of the tangent space, as it can be any $N \times n$ matrix with positive entries. Thus, step 4.2. cannot be 
        applied, and we cannot simply move $B_2^{\text{cache}}$ to the current iterate. Note that for the homogeneous space setup we do not have this problem as 
        we are operating on the global tangent space $\g^{hor,E}$ in each step.
\end{itemize}
We have different options to overcome these problems:
\begin{itemize}
  \item[(i)] Dismiss the second moment information $B_2^{\text{cache}}$ completely.
  \item[(ii)] Find a~workaround to keep at least some second moment information.
\end{itemize}
The first option is clear. Here, we completely lose the quadratic gradient information which empirically has turned out not to work sufficiently good. 
Another approach would be to use the structure of the tangent space and to modify Adam so that we can overcome the problems by at least keeping some 
second moment information. We propose the following setup:
\begin{itemize}
  \item[1.] Compute the Riemannian gradient $Z_t:=\grad\L_{X_t}(\theta,\X_t)$ with respect to the preferred metric on $\St(n,N)$ using the Euclidean gradient $\nabla \L_{X_t}(\theta,\X_t)$.
  \item[2.] Apply a~modified Adam optimizer on the tangent space $T_{X_t}\St(n,N)$ by keeping only one cache but using the current second moment information: 
            \begin{itemize}
             \item[2.1.] $B_2 \leftarrow \sqrt{\frac{\beta_2 - \beta_2^t}{1 -  \beta_2^t}(B_1^{\text{cache}} \odot B_1^{\text{cache}}) + \frac{1 - \beta_2}{1 - \beta_2^t} Z_t \odot Z_t + \delta}$,
             \item[2.2.] $B_1^{\text{cache}} \leftarrow \frac{\beta_1 - \beta_1^t}{1 -  \beta_1^t}B_1^{\text{cache}} + \frac{1 - \beta_1}{1 - \beta_1^t} Z_t$,
             \item[2.3.] $B_1^{\text{cache}} \in T_{X_t}\St(n,N)$ can be decomposed as $B_1^{\text{cache}} = X_tW + X_t^{\perp}K$ with $W \in \S_{\skew}(n)$ and 
             $K \in \R^{(N-n)\times n}$. 
                         Now, to be able to apply the parameter update, we do this separately for $X_t W$ and the orthogonal part $X_t^{\perp}K$:
                         \begin{itemize}
                          \item $X_tW \leftarrow X_t \left(W/\sqrt{B_2^T B_2}\right)$,
                          \item $X_t^{\perp}K \leftarrow (I_N - X_t X_t^T ) ((X_t^{\perp}K)/ B_2)$,
                          \item $V_t \leftarrow -\eta \left(X_t W + X_t^{\perp}K\right)$.
                         \end{itemize}
            \end{itemize}
  \item[3.] Retract the update vector $V_t$ to obtain the new iterate: $X_{t+1} \leftarrow R_{X_t}\left(V_t\right)$.
  \item[4.] Transport the cache $B_1^{\text{cache}}$ to the new iterate tangent space $T_{X_{t+1}}\St(n,N)$ 
            in direction $V_t$ along $R_{X_t}$ using the preferred vector transport $\T_{V_t}$:
              $$B_1^{\text{cache}} \leftarrow \T_{V_t}(B_1^{\text{cache}}).$$
\end{itemize}
Instead of the second cache $B_2^{\text{cache}}$, we now use the quadratic information of the previous gradients contained in the (not yet updated) $B_1^{\text{cache}}$ cache.
This way, instead of using the weighted sum of quadratic gradients stored in $B_2^{\text{cache}}$, 
we use the quadratic information of the weighted sum of gradients stored in $B_1^{\text{cache}}$ plus the quadratic sum of the current gradient. 
Alternatively, one can first update $B_1^{\text{cache}}$ and then use only the Hadamard product of $B_1^{\text{cache}}$. This is a~little cheaper, but then 
the quadratic sum of the current gradient is not used directly.
In step 2.3., we used that $W$ is skew-symmetric, which thus allows us to apply the Adam update here directly. To do so, we needed to 
update the orthogonal part $X_t^{\perp}K$ separately. As the Adam update step destroys the orthogonality, we enforced this by the left-multiplication of
$(I_N - X_t X_t^T)$. 
In the end, we just keep one moment $B_1^{\text{cache}}$ that is transported via the preferred vector transport. 
This way, we minimize the information loss of quadratic gradient information by still being able to respect the structure of the underlying 
tangent space. The modified Adam optimizer step, referred to as $\stiefeladam(\cdot)$, is displayed in Algorithm~\ref{sec:application_env:algorithm_name:ModifiedAdam}.
The whole update step using the modified Adam optimizer is shown in Algorithm~\ref{sec:application_env:algorithm_name:PSDLayerModifiedAdamStep}.
\begin{algorithm}
\textbf{Require:} Step size $\eta=0.001$; decay rates $\beta_1 = 0.9, \beta_2=0.99$; numerical stabilizer $\delta=10^{-8}$; cached moment $B_1^{\ca}$;
vector $Z_t \in T_{X_t}\St(n,N)$ (mostly $\grad \L_{X_t}(\theta,\X_t)$). \\
\textbf{Return:} New update vector $V_t \in T_{X_t}\St(n,N)$. \\
\textbf{Step:}
\begin{algorithmic}[1]
  \State $B_2 = \sqrt{\frac{\beta_2 - \beta_2^t}{1 -  \beta_2^t}(B_1^{\text{cache}} \odot B_1^{\text{cache}}) + \frac{1 - \beta_2}{1 - \beta_2^t} Z_t \odot Z_t + \delta}$ {\small\Comment{Compute \enquote{pseudo cache} $B_2$ (and add $\delta$).}}
  \State $B_1^{\text{cache}} \leftarrow \frac{\beta_1 - \beta_1^t}{1 -  \beta_1^t}B_1^{\text{cache}} + \frac{1 - \beta_1}{1 - \beta_1^t} Z_t$ {\small\Comment{Update $B_1^{\text{cache}}.$}}
  \State $W = X_t^TB_1^{\text{cache}}$ {\small\Comment{Compute the matrix $W$ of the skew-symmetric part $X_tW$ of $B_1^{\text{cache}}$.}}
  \State $X^{\perp}K = B_1^{\text{cache}} - X_tW$ {\small\Comment{Compute the complement part $X_t^{\perp}K$ of $B_1^{\text{cache}}$.}}
  \State $XW = X_t \left(W/\sqrt{B_2^T B_2}\right)$ {\small\Comment{Compute the new skew-symmetric part.}}
  \State $X^{\perp}K \leftarrow (I_N - X_t X_t^T ) ((X^{\perp}K)/ B_2)$ {\small\Comment{Update the complement part.}}
  \State $V_t \leftarrow -\eta \left(XW + X^{\perp}K\right)$ {\small\Comment{Return the update vector $V_t$ computed from $XW$ and $X^{\perp}K$.}}
\end{algorithmic}
    \caption{Modified Adam optimizer step in iteration $t \mapsto t+1$ to be applied on $T_{X_t}\St(n,N)$.}
    \label{sec:application_env:algorithm_name:ModifiedAdam}
\end{algorithm}
\begin{algorithm}
\begin{algorithmic}[1]
  \State $Y_t = \nabla \L_{X_t}(\theta,\X_t)$ {\small\Comment{Compute the Euclidean gradient.}}
  \State $Z_t = \grad\L_{X_t}(\theta,\X_t)$ {\small\Comment{Compute the Riemannian gradient from $Y_t$.}}
  \State $V_t = \stiefeladam(Z_t)$ {\small\Comment{Compute the update vector $V_t$ using modified Adam.}}
  \State $X_{t+1} = R_{X_t}(V_t)$ {\small\Comment{Retract $V_t$ back to $\St(n,N)$ using a~retraction $R_{X_t}$.}}
  \State $A_{t+1} = \left[\begin{smallmatrix}
    X_{t+1} & 0 \\
    0 & X_{t+1} \\
  \end{smallmatrix}\right]$ {\small\Comment{Construct the new iterate $A_{t+1}$.}}
  \State $B_1^{\ca} \leftarrow \T_{V_t}(B_1^{\text{cache}})$  {\small\Comment{Transport the modified Adam cache to $T_{X_{t+1}}\St(n,N)$.}}
\end{algorithmic}
    \caption{Complete update step $t \mapsto t+1$ for a~PSDLayer using Adam on $\St(n,N)$ directly.}
    \label{sec:application_env:algorithm_name:PSDLayerModifiedAdamStep}
\end{algorithm}

\subsection{Implementation of different setups to learn the network}\label{sec:application_subsec:implementation}

As part of this thesis, we have implemented different training routines in the Julia programming language to learn the 
symplectic autoencoder network presented in Section~\ref{sec:application_subsec:symplecticautoencodernetwork}. We have tested the setups 
for two different examples of Hamiltonian systems.
To do so, we used the Julia library \textit{GeometricMachineLearning.jl} (GML). 
This library implements the single layers of the network architecture described in Section~\ref{sec:application_subsubsec:architecture}.
It contains pre-defined building blocks to chain the layers to a~network, execute an~optimizer step, 
handle data, conduct the model reduction and much more. 
The library is developed and maintained by the authors of \cite{main:brantner2023}, who were also the first to introduce this network.
It is still under current development. The latest release version as of June 15th 2024 is v$0.3.0$. 
For our implementation, we use the release version v$0.2.0$, which we will refer to throughout the rest of the thesis. 
At some points, we will comment on important changes made in the latest version v$0.3.0$ compared to v$0.2.0$.
There also is a~GitHub repository \cite{githubGML:brantner2024} of that project, where one can find the source code, release information 
and documentations of the project.

We use this library to implement complete learning and testing routines. In this section, we give an~overview of all tested configurations. We omit 
most of the implementation details here, as they can be found in our accompanying Julia implementation project. 
We focus on the core parts of the PSDLayer optimizer step and the functional changes made to the GML library modules.

First, in Section~\ref{sec:application_subsubsec:generalprojectstructure}, we introduce the structure of our implementation project.
Next, we work through the different parts of a~complete learning setup and describe what configurations we implemented. 
We took the learning configuration used in \cite{main:brantner2023} as a~reference setup. The authors published their implementation of 
this setup in \cite[scripts/symplectic\_autoencoder]{githubGML:brantner2024}. 
We also show how we have integrated our own network-specific modifications into GML and what adjustments we had to make to GML.

\subsubsection{General structure} \label{sec:application_subsubsec:generalprojectstructure}
The project entry file is named \stexttt{Implementation}. An~overview of the project structure is displayed in Figure~\ref{sec:application_env:figure_name:folderstructure}.
\begin{figure}[htbp]
    \centering
 \scalebox{0.9}{
    \begin{forest}
    for tree={
        font=\ttfamily,
        grow'=0,
        child anchor=west,
        parent anchor=south,
        anchor=west,
        calign=first,
        edge path={
            \noexpand\path [draw, \forestoption{edge}] (!u.south west) -- +(5pt,0) |- (.child anchor)\forestoption{edge label};
        },
        before typesetting nodes={
            if n=1
                {insert before={[,phantom]}}
                {}
        },
        fit=band,
        before computing xy={l=15pt},
    }
    [Implementation
        [plots\_temp]
        [resources]
        [src
            [reproduction
                [general
                  [fn\_global\_consts.jl]
                  [...]
                ]
                [main
                  [fn\_training\_v020\_reproduction.jl]
                  [...]
                ]
            ]
            [sineGordon
                [examples
                  [...]
                ]
                [general
                  [...]
                ]
                [main
                  [...]
                ]
            ]
            [alternative\_optimizer
                [fn\_adam\_added\_scripts
                  [...]
                ]
                [fn\_adam\_mofified\_scripts
                  [...]
                ]
                [general
                  [...]
                ]
                [main
                  [...]
                ]
            ]
            [shared
                [fn\_GML\_custom
                  [fn\_batch\_v020.jl]
                  [fn\_reduced\_system\_v020.jl]
                ]
                [fn\_customtypes.jl]
                [fn\_loss.jl]
            ]
        ]
        [Project.toml]
        [Manifest.toml]
    ]
    \end{forest}
 }
    \caption{Project structure of the implementation.}
    \label{sec:application_env:figure_name:folderstructure}
\end{figure}
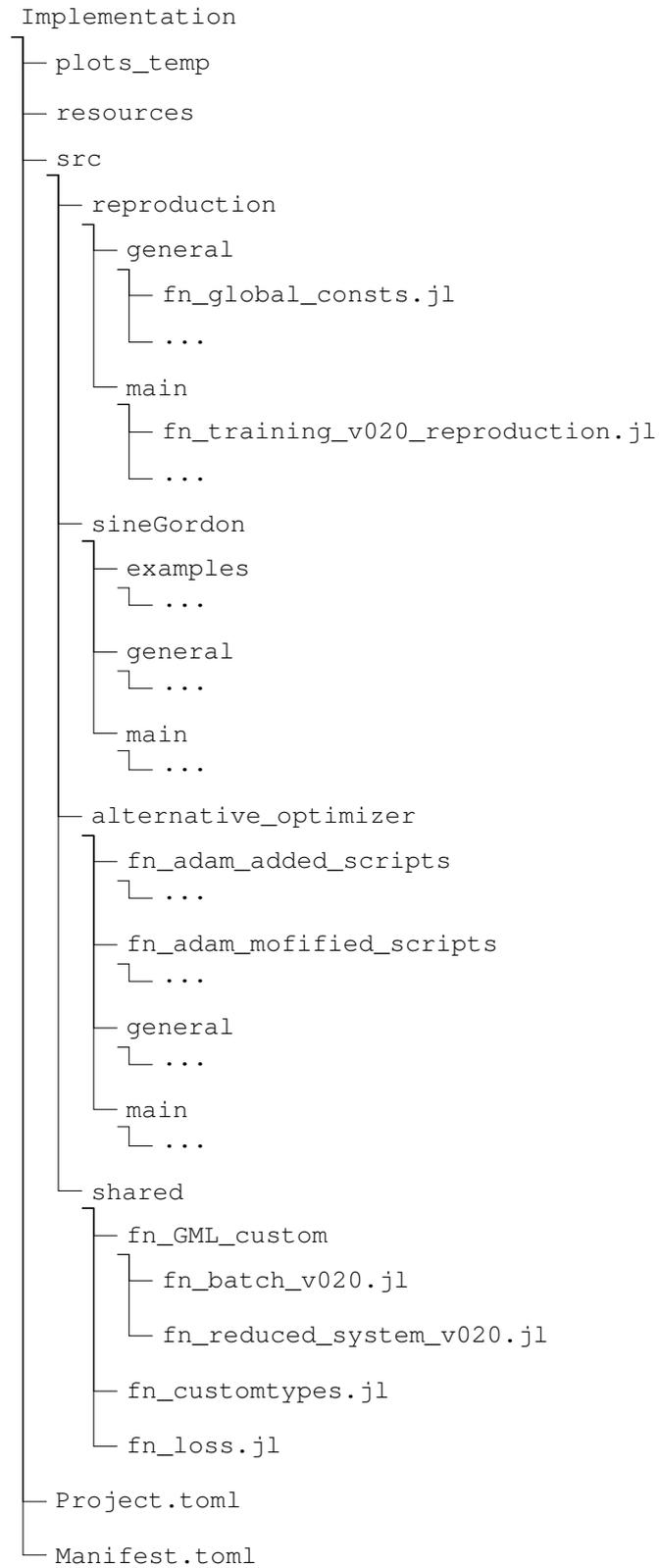
Like for all Julia projects, the package and environment setups are specified in the files \stexttt{Project.toml} and 
\stexttt{Manifest.toml}. 
The actual source code is found in the folder \stexttt{src}. The source code is structured in the four folders 
\begin{itemize}
  \item \stexttt{reproduction},
  \item \stexttt{sineGordon},
  \item \stexttt{alternative\_optimizer} and
  \item \stexttt{shared}.
\end{itemize}
This structure evolved in the course of the sequence in which the individual tasks were created.
First, in \stexttt{reproduction}, the setup from \cite{main:brantner2023} was rebuilt and modified to try out some parameter and network variations. 
This also includes the implementation of an~example with which the setup was tested (see Section~\ref{sec:numericalexperiments_subsec:waveequation} for details). 
Subsequently, an~example of our own was implemented in \stexttt{sineGordon} (see Section~\ref{sec:numericalexperiments_subsec:sinegordon} for details). 
For this, learning setups similar to \stexttt{reproduction} were used. Therefore, there are modules in \stexttt{shared} that are used for both subprojects. 
It also contains modifications of the library modules from GML. 
Finally, the use of the StiefelAdam optimizer introduced in Section~\ref{sec:application_subsubsec:alternativeoptimizer} was implemented in \stexttt{alternative\_optimizer}. 
The learning and testing routines of both previously implemented examples were also implemented with the new optimizer here.
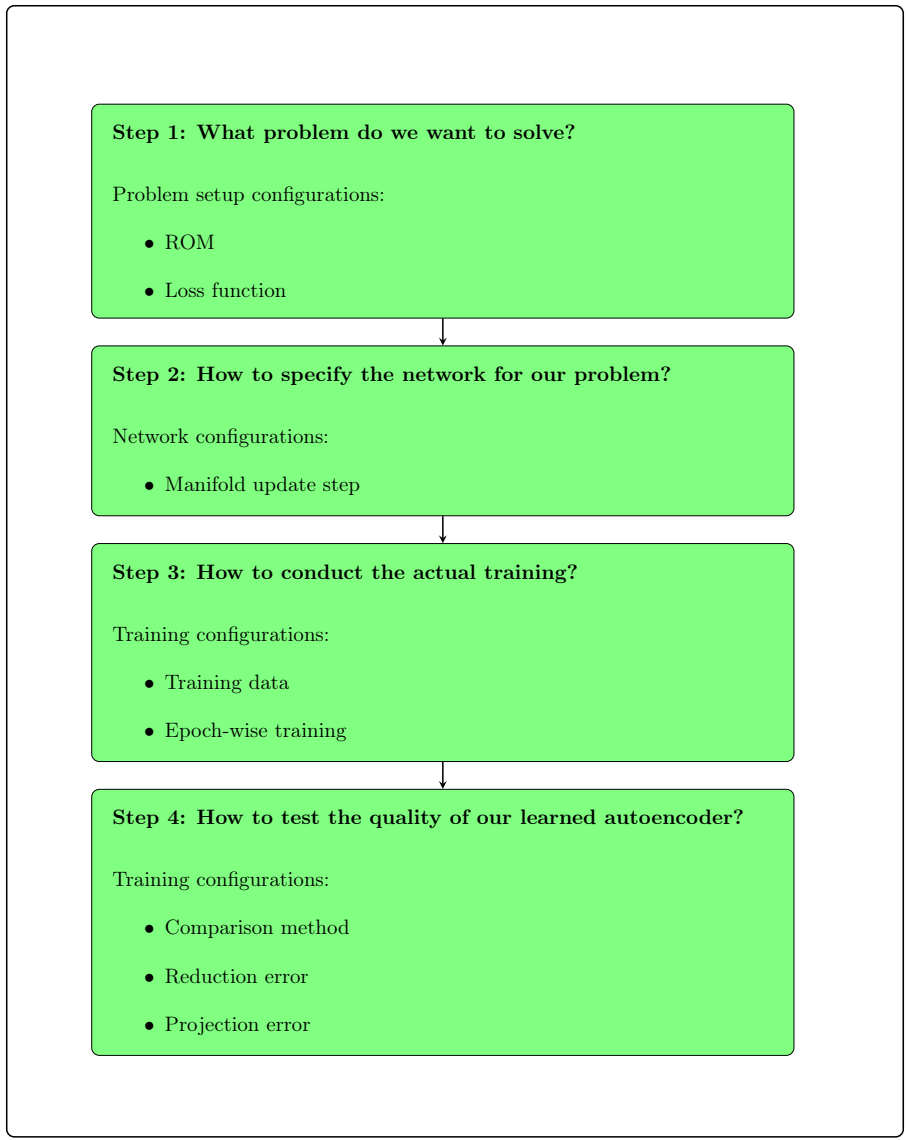
\begin{figure}[ht!] 
    \centering
    \resizebox{0.75\textwidth}{!}{ 
    \begin{tikzpicture}[scale=1]
        \tikzstyle{button} = [
            rectangle, 
            rounded corners, 
            minimum width=8cm, 
            minimum height=0.8cm, 
            text centered, 
            draw=black, 
            fill=green!50,
            inner sep=1em,
            align=center
        ]
        \tikzstyle{arrow} = [thick, ->, >=stealth]
        \node[button, below=0.5cm of btn1] (btn2){
        \begin{minipage}[c]{12cm}
          \textbf{Step 1: What problem do we want to solve?}  \\ \\
          Problem setup configurations:
            \begin{itemize}
                \item ROM
                \item Loss function 
            \end{itemize}
        \end{minipage}
        };
        \node[button, below=0.5cm of btn2] (btn3) {
        \begin{minipage}[c]{12cm}
        \textbf{Step 2: How to specify the network for our problem?}  \\ \\
        Network configurations:
            \begin{itemize}
                \item Manifold update step
            \end{itemize}
        \end{minipage}
        };
        \node[button, below=0.5cm of btn3] (btn4) {
        \begin{minipage}[c]{12cm}
        \textbf{Step 3: How to conduct the actual training?}  \\ \\
        Training configurations:
            \begin{itemize}
                \item Training data 
                \item Epoch-wise training
            \end{itemize}
        \end{minipage}
        };
        \node[button, below=0.5cm of btn4] (btn5) {
        \begin{minipage}[c]{12cm}
        \textbf{Step 4: How to test the quality of our learned autoencoder?}  \\ \\
        Training configurations:
            \begin{itemize}
                \item Comparison method
                \item Reduction error
                \item Projection error 
            \end{itemize}
        \end{minipage}
        };
        \foreach \i in {4,3,2} {
            \draw[arrow] (btn\i.south) -- (btn\the\numexpr\i+1.north);
        }
        \draw[rounded corners, thick] ([shift={(-6,.5)}]btn1.north west) rectangle ([shift={(2,-1.5)}]btn5.south east);
    \end{tikzpicture}
    }
    \caption{All configuration parts for a~full learning routine of the autoencoder network.}
    \label{sec:application_env:tikzpicture_name:learningconfigurationparts}
\end{figure}

\newpage
\noindent A~complete learning configuration consists of 
\begin{itemize}
  \item the problem setup,
  \item the network,
  \item the training and 
  \item the testing.
\end{itemize}
In the following, we describe which parameter constellations and variants we chose for each of the four main parts and where we made functional modifications of GML. 
We use the learning setup from \cite{main:brantner2023} that was implemented in \cite[scripts/symplectic\_autoencoder]{githubGML:brantner2024} as a~reference configuration 
because this is the learning routine that \cite{main:brantner2023} used to test their network.
We will point out all adaptions and changes we made in comparison to this setup. 
\cite[scripts/symplectic\_autoencoder]{githubGML:brantner2024} was published around October 2023. 
Back then, the latest version of GML was v$0.1.0$. Thus, we will also point out major adaptions we had to perform 
to make the setup compatible with the release version v$0.2.0$. 
An~overview of the different configuration parts is displayed in Figure~\ref{sec:application_env:tikzpicture_name:learningconfigurationparts}.
The actual learning, the implementation of the test examples, and details on how to use the project are then covered in the last chapter~\ref{sec:numericalexperiments}.

\subsubsection{Problem setup configurations}

In contrast to the problem setup described in Section~\ref{sec:application_subsec:problemsetup}, \cite{main:brantner2023} used a~slightly different setup:
\begin{itemize}
  \item They performed the reconstruction differently without the use of a~reference state $x_{\reff}$.
  \item They used a~different loss function $\loss(\X_t,\Y_t)$.
\end{itemize}
  
\subsubsubsection{ROM}
According to Algorithm~\ref{sec:application_env:algorithm_name:FOMtoROM}, we must first select a~reduced initial value 
$x_{r,\mu}^0$ and then compute the reference state $x_{\reff,\mu} = x_{\mu}^0 - d(x_{r,\mu}^0)$ to obtain the ROM.
The reconstruction of the approximate solutions is then simply $\tilde{x}_{\mu}(t)=x_{\reff,\mu} + d(x_{r,\mu}(t))$. 
This setup ensures the exact reproduction of the initial value, i.e., $\tilde{x}_{\mu}(0)=x_{\mu}^0$.
\cite{main:brantner2023} omitted the reference state. This means the reconstruction there is computed as 
$$\tilde{x}_{\mu}(t)=d(x_{r,\mu}(t)).$$
Note that this setup does not ensure the exact reconstruction of the initial value. 
To align with our proposed problem setup, we decided to implement an~alternative configuration using the reference state $x_{\reff,\mu}$. 
For this change, we had to make modifications in the GML file \stexttt{reduced\_system.jl}.
  
\subsubsubsection{Loss function}
In \eqref{sec:application_env:align_name:errorfunction}, we defined the error function for the minimization problem to be solved.
Given a~matrix batch $\X_t$ and a~matrix output estimate $\Y_t := (d\circ e)(\X_t)$ and using the Euclidean norm $\norm{\cdot}_2$
on $\R^{2Nk}$, this yielded the autoencoder loss function (see \eqref{sec:application_env:align_name:nnlossfunctionformatrixbatch})
$$\loss(\X_t,\Y_t) = \frac{1}{2N|\X_t|} \norm{\X_t - \Y_t}_F^2.$$
In contrast, GML uses the relative error 
\begin{align}
  \loss(\X_t,\Y_t) &= \tfrac{\norm{\X_t - \Y_t}_F}{\norm{\X_t}_F}
\end{align}
as the loss function.
We decided to try out both of these loss functions in our implementation.
To do so, we had to modify the loss function implemented in the GML file \stexttt{data\_loader.jl}. We created a~new file 
\stexttt{fn\_loss.jl} and implemented both loss functions here. 
  
\subsubsubsection{Summary}
We summarize the default GML configurations for the ROM and the loss function and our modifications in Table~\ref{sec:application_env:table_name:variantsFOMROM}.
\begin{table}[ht!]
\centering
\begin{tabular}{|p{2.3cm}||p{4.7cm}|p{5.5cm}|} 
 \hline
 \multicolumn{3}{|c|}{Problem setup configuration values} \\
 \hline
 \textbf{Element}  & \textbf{Reference configuration} & \textbf{New configuration} \\
 \hline
 ROM                          & ---                                                                 & $x_{\reff,\mu}=x_{\mu}^0 - d(x_{r, \mu}^0)$                           \\
                              & $\tilde{x}_{\mu}(t)=d(x_{r,\mu}(t))$                        & $\tilde{x}_{\mu}(t)=x_{\reff,\mu} + d(x_{r,\mu}(t))$         \\
 [0.2cm]
 Loss function                & $\loss(\X_t,\Y_t) = \tfrac{\norm{\X_t - \Y_t}_F}{\norm{\X_t}_F}$    & $\loss(\X_t,\Y_t) = \frac{1}{2N|\X_t|} \norm{\X_t - \Y_t}_F^2$    \\
 [0.1cm]
\hline
 \multicolumn{3}{|c|}{Functional GML modifications} \\
 \hline
 \textbf{Element}  & \textbf{Reference configuration} & \textbf{New configuration} \\
 \hline
 ROM                      & ---                    & \small{\stexttt{reduced\_system.jl}}                       \\
 [0.2cm]
 Loss function                & ---                    & \small{\stexttt{data\_loader.jl}}                       \\
 \hline
\end{tabular}
\caption{Variants of problem setup configurations.}
\label{sec:application_env:table_name:variantsFOMROM}
\end{table}

\subsubsection{Network configurations} \label{sec:application_subsubsec:networkconfiguration}
We implemented the network setup in the files 
\begin{itemize}
  \item \stexttt{fn\_training\_v020\_reproduction.jl},
  \item \stexttt{fn\_training\_v020\_sineGordon.jl},
  \item \stexttt{fn\_training\_v020\_alternative\_optimizer\_reproduction.jl},
  \item \stexttt{fn\_training\_v020\_alternative\_optimizer\_sineGordon.jl}.
\end{itemize}
We do this for every implemented Hamiltonian system example (reproduction and sineGordon) and for the different 
optimizer approaches separately.
Here, we composed the network by specifying an~optimizer, initializing weights and parameters and 
chaining the GML-built-in layers (\stexttt{PSDLayer}, \stexttt{GradientLayerQ} and \stexttt{GradientLayerP}) 
to the network block specified in Figure~\ref{sec:application_env:tikzpicture_name:network}.
Note that the name component \enquote{v020} of the training files indicates that we are using a~modified 
version of the training file \stexttt{training.jl} from \cite[scripts/symplectic\_autoencoder]{githubGML:brantner2024}.

\subsubsubsection{GML optimizers}
GML offers the use of different optimizers. As described in Section~\ref{sec:application_subsubsec:architecture}, we choose 
Adam. The GML built-in update step for a~\stexttt{GradientLayer} uses the standard Adam update routine displayed in Algorithm~\ref{sec:application_env:algorithm_name:GradientLayerUpdatestep}.
On a~\stexttt{PSDLayer} GML applies the update step using Adam on homogeneous spaces as displayed in Algorithm~\ref{sec:application_env:algorithm_name:PSDLayerUpdatestep}.
We can specify the Adam parameters in GML ourselves. Analogous to \cite{main:brantner2023}, we chose $(\eta, \beta_1, \beta_2, \delta)=(0.001, 0.9, 0.99, 10^{-8})$. The initial default 
of the cache is zero. These choices are applied for both layer specific Adam optimizers.
In GML, we can choose between different activation functions. We chose $\tanh(x)$.

A~common technique is to reduce the learning rate $\eta$ a~little in each iteration. This is often referred to as 
\textit{decay rate}. This technique can speed up the error loss, because the longer we optimize, the closer we get to an~optimum of the function, 
and therefore the smaller the required steps size gets.
GML does not yet offer this in v$0.2.0$. However, a~decay option was added in version v$0.3.0$.  

For the \stexttt{PSDLayer} update step, we additionally need to choose a~retraction. GML offers two choices here:
\begin{itemize}
  \item the Geodesic retraction and
  \item the Cayley retraction.
\end{itemize}
We decided to use the Cayley retraction $R_{I_N}^{\cay}$ on $\Orth(N)$ (see expressions~\eqref{sec:manifold_env:align_name:orthogonalgroupcayleyretraction},~\eqref{sec:manifold_env:align_name:orthogonalgroupSMWformulacayleywithrightmultiplication}~and~\eqref{sec:manifold_env:align_name:compactstiefelhomogeneouscayleyE}). 
For the Riemannian gradient on $\St(n,N)$, GML by default uses the Riemannian gradient with respect to the canonical metric~\eqref{sec:manifold_env:align_name:compactstiefelcanonicalriemanniangradient}.

\subsubsubsection{Alternative Stiefel optimizer}
To integrate our new manifold update step from Algorithm~\ref{sec:application_env:algorithm_name:PSDLayerModifiedAdamStep} 
in GML (version v$0.2.0$), we had to conduct a~number of library adaptions. To do so, we extended and modified the GML-files
\begin{itemize}
  \item \stexttt{optimizer.jl},
  \item \stexttt{retractions.jl},
  \item \stexttt{stiefel\_manifold.jl},
  \item \stexttt{init\_optimizer\_cache.jl},
  \item \stexttt{utils.jl},
  \item \stexttt{batch.jl}
\end{itemize} 
and added new custom files 
\begin{itemize}
  \item \stexttt{fn\_adam\_optimizer\_stiefel.jl},
  \item \stexttt{fn\_adam\_optimizer\_stiefel\_withDecay.jl}.
\end{itemize}
We will not go into too much detail about all GML modifications we made to integrate our new optimizer. 
Instead, we focus on the implementation of the core steps of the StiefelAdam update step.
First, we modified \stexttt{stiefel\_manifold.jl} to implement the use of two different Riemannian gradients:
\begin{itemize}
  \item Canonical (already present in GML),
  \item Euclidean (added, see definition~\eqref{sec:manifold_env:align_name:compactstiefeleuclideanriemanniangradient} and Figure~\ref{sec:application_env:jllisting:euclideanrgrad}).
\end{itemize}
\begin{figure}[ht!] 
    \centering
  \begin{jllisting}
     #fn_ADDED: Euclidean metric rgrad
     function rgrad(::Euclidean, Y::StiefelManifold, e_grad::AbstractMatrix)
         H = Y.A'*e_grad
         e_grad - 0.5*Y.A*(H + H')
     end
  \end{jllisting}
    \caption{Implementation of the Euclidean Riemannian gradient on $\St(n,N)$.}
    \label{sec:application_env:jllisting:euclideanrgrad}
\end{figure}
\begin{figure}[ht!] 
    \centering
  \begin{jllisting}
    function update!(o::Optimizer{<:StiefelAdamOptimizerWithDecay{T}}, C::AdamCache, B::AbstractArray, Y::StiefelManifold) where T
      # determine B2 from Cache
      C.B₂ .= racᵉˡᵉ( scalar_add(((o.method.ρ₂ - o.method.ρ₂t)/(T(1.) - o.method.ρ₂t))*(C.B₁.^2) + ((T(1.) - o.method.ρ₂)/(T(1.) - o.method.ρ₂t))*B.^2, o.method.δ ))
      # update Cache:
      add!(C.B₁, ((o.method.ρ₁ - o.method.ρ₁t)/(T(1.) - o.method.ρ₁t))*C.B₁, ((T(1.) - o.method.ρ₁)/(T(1.) - o.method.ρ₁t))*B)
      # get skew symmetric part of Cache divided by B2:
      n = size(C.B₂,2)
      B2n = C.B₂[1:n,:]
      W = Y.A'*C.B₁
      # get orthogonal part of Cache divided by B2:
      OB2 = /ᵉˡᵉ(C.B₁ - Y.A*W,C.B₂)
      # store update vector in gradient varible B
      B .= -o.method.η * ( Y.A * /ᵉˡᵉ(W,B2n) + OB2 - Y.A*(Y.A'*OB2) )
    end
  \end{jllisting}
    \caption{StiefelAdam optimizer implementation.}
    \label{sec:application_env:jllisting:stiefeladam}
\end{figure}
In \stexttt{fn\_adam\_optimizer\_stiefel.jl} we implemented the StiefelAdam optimizer routine specified in Algorithm~\ref{sec:application_env:algorithm_name:ModifiedAdam}. 
The implementation is displayed in Figure~\ref{sec:application_env:jllisting:stiefeladam}.
Similar to the homogeneous Adam step, we did this without a~decay rate. 
In \stexttt{fn\_adam\_optimizer\_stiefel\_withDecay.jl} we added the same StiefelAdam routine with the only difference that now we apply 
a slight decay 
$$ \eta \leftarrow \eta \cdot 0.9995$$
to the learning rate $\eta$ in each iteration. The decay is applied in the \stexttt{optimization\_step!} function in the file \stexttt{fn\_optimizer\_v020.jl}.
Finally, we use the Cayley retraction and two different vector transports along Cayley to complete the update step. 
The Cayley retraction on $\St(n,N)$ is calculated in~\eqref{sec:manifold_env:align_name:compactstiefelSMWformulacayleywithrightmultiplication}. 
We use the submanifold and the differentiated vector transports on Cayley discussed in Section~\ref{sec:manifold_subsubsubsec:vectortransport}. 
The submanifold transport is computed in \eqref{sec:manifold_env:align_name:compactstiefelvectranssub}, the differentiated transport in \eqref{sec:manifold_env:align_name:compactstiefelvectransdiffSMW}.
Both the vector transport of the cache $B_1^{\ca}$ together with the retraction of the new iterate are conducted in the same function within the file \stexttt{fn\_retractions\_v020.jl}.
The implementation is displayed in Figure~\ref{sec:application_env:jllisting:vectortransports}.
\begin{figure}[ht!] 
    \centering
  \begin{jllisting}
    # to work with cayley on Stiefel directly: implements vector transport SUBMANIFOLD
    function cayley!(::SubmanifoldVectortransport, Z::AbstractArray{T}, Φ::AbstractArray{T}, C::AdamCache) where T
        H = Φ.A'*Z
        U = hcat(Z - 0.5*Φ.A*(H - H'), -Φ.A)
        V = vcat(Φ.A',Z')
        VU = V*U
        VΦ = V*Φ
    
        # apply retraction to update the weights
        Φ.A .= Φ.A + 0.5*U* ( VΦ + (one(VU) - 0.5*VU) \ (VΦ + 0.5*VU*VΦ) )
    
        # vector transport *sub* of Cache C.B₁ along Cayley 
        H .= Φ.A'*C.B₁
        C.B₁ .= C.B₁ - 0.5*Φ.A * ( H + H' )
    end
    
    # to work with cayley on Stiefel directly: implements vector transport DIFFERENTIAL
    function cayley!(::DifferentialVectortransport, Z::AbstractArray{T}, Φ::AbstractArray{T}, C::AdamCache) where T
        H = Φ.A'*Z
        U = hcat(Z - 0.5*Φ.A*(H - H'), -Φ.A)
        V = vcat(Φ.A',Z')
        VU = V*U
        VΦ = V*Φ
    
        # vector transport *diff* of Cache C.B₁ along Cayley 
        H .= Φ.A'*C.B₁
        U_B = hcat(C.B₁ - 0.5*Φ.A*(H - H'), -Φ.A)
        V_B = vcat(Φ.A',C.B₁')
        VU_B_half = 0.5*(V*U_B)
        mat_to_be_inv = one(VU) - 0.5*VU
        C.B₁ .= (U_B + U * (mat_to_be_inv \ VU_B_half)) * (V_B * (Φ.A + U * (mat_to_be_inv \ (0.5*VΦ))))
    
        # apply retraction to update the weights
        Φ.A .= Φ.A + 0.5*U* ( VΦ + (one(VU) - 0.5*VU) \ (VΦ + 0.5*VU*VΦ) )
    end
  \end{jllisting}
    \caption{Implementation of the vector transport and the Cayley retraction on $\St(n,N)$.}
    \label{sec:application_env:jllisting:vectortransports}
\end{figure}

Note that along with the StiefelAdam optimizer, we implemented a~bunch of other optimizers. These are
\begin{itemize}
  \item StiefelStochasticGradientDescentWithMomentum,
  \item StiefelStochasticGradientDescentWithMomentumWithDecay and
  \item StiefelCayleyAdamOptimizerFromOtherPaper.
\end{itemize}
Since the StiefelAdam (with decay) worked best in all constellations we tested, we did not include any further investigation and 
numerical experiments on the other optimizers in this thesis.

\subsubsubsection{Comparison of update step costs}
Since the autoencoder is supposed to reduce a~high dimensional Hamiltonian system, it is important to know the costs of each update step 
depending on the full dimension $N$ and the reduced dimension $n$, where $n \ll N$. 

For the homogeneous space update setup, we summarize the costs in Table~\ref{sec:application_env:table_name:costshomogeneousadamupdatestep}. 
First, we must compute the canonical Riemannian gradient. According to~\eqref{sec:manifold_env:align_name:compactstiefelcanonicalriemanniangradient}, 
this can be done in $\O(Nn^2)$. Next, we must compute the \stexttt{section} $\lambda_E(X)$. Analogous to \cite{main:brantner2023}, we use
a~QR-decomposition (see Algorithm~\ref{sec:manifold_env:algorithm:compactstiefelalgosection}) of a~$N \times (N-n)$ matrix. 
This causes costs of ${\O(N(N-n)^2)}$, which makes this the most expensive task for $n \ll N$. 
In the next step we \stexttt{lift} the Riemannian gradient to the global tangent space $\g^{hor,E}$, see \eqref{sec:manifold_env:align_name:compactstiefellift+sectionproduct}. 
This comes at costs of $\O(N(N-n)n)$.
Now, we apply the classic Adam optimizer on an~element of $\g^{hor,E}$.
As shown in \cite{adam:brantner2023}, the elements of $\g^{hor,E}$ have at most $Nn$ non-zero entries.  
Since the classical Adam optimizer consists only of point-wise operations, we get costs $\O(Nn)$ in the number of non-zero entries 
of the elements of $\g^{hor,E}$. 
Finally, we apply the \stexttt{retract} operation using the Cayley retraction \eqref{sec:manifold_env:align_name:orthogonalgroupcayleyretraction}.
As shown in \eqref{sec:manifold_env:align_name:orthogonalgroupSMWformulacayleywithrightmultiplication} and \eqref{sec:manifold_env:align_name:compactstiefelhomogeneouscayleyE}, 
we can compute the whole \stexttt{retract} step in $\O(Nn^2)$.

The costs of one StiefelAdam update step are shown in Table~\ref{sec:application_env:table_name:costsstiefeladamupdatestep}. As before, 
we first calculate the Riemannian gradient. Here, we have the choice between the canonical and the Euclidean Riemannian gradient 
given in \eqref{sec:manifold_env:align_name:compactstiefeleuclideanriemanniangradient} and \eqref{sec:manifold_env:align_name:compactstiefelcanonicalriemanniangradient}. 
Both come at costs of $\O(Nn^2)$. Second, we apply our proposed \enquote{pseudo-Adam} optimizer step described in Algorithm~\ref{sec:application_env:algorithm_name:ModifiedAdam} 
on the compact Stiefel manifold tangent spaces. Here, we have to extract the skew-symmetric and the orthogonal components of the tangent vectors and apply point-wise operations. 
This causes costs of $\O(Nn^2)$. Finally, we \stexttt{retract} via the Cayley retraction \eqref{sec:manifold_env:align_name:compactstiefelSMWformulacayleywithrightmultiplication} 
and apply either the submanifold vector transport along Cayley (see~\eqref{sec:manifold_env:align_name:compactstiefelvectranssub}) or the 
differentiated vector transport along Cayley (see~\eqref{sec:manifold_env:align_name:compactstiefelvectransdiffSMW}). 
All three can be computed in $\O(Nn^2)$. 

In total, we get costs of $\O(N^2(N-n)+Nn^2)$ for the homogeneous space update step versus costs of $\O(Nn^2)$ for the direct approach on the 
compact Stiefel manifold. 
A~speed test in Section~\ref{sec:numericalexperiments_subsec:speedtest} for each of the two different approaches shows that the lower costs of 
the StiefelAdam update step results in a~significant speed-up for high dimensions $N$ and a~reduced dimension $n \ll N$.

As mentioned in Section~\ref{sec:manifold_subsubsubsec:gradientlayer}, we apply the classic Adam optimizer on GradientLayers as well. 
Interestingly, since we chose $L=5N$ for each GradientLayer in the upscaling dimension, we end up with costs of $O(N^2)$ for one Adam iteration here. 
This is quite expensive and can cause some critical performance issues for high dimensions $N$. 

\begin{table}[ht!]
\centering
\begin{tabular}{c|c|c|c} 
 \textbf{Step}  & \textbf{Configuration} & \textbf{Formula} & \textbf{Computational costs} \\
 \hline
 \stexttt{rgrad}    & Canonical    &   \eqref{sec:manifold_env:align_name:compactstiefelcanonicalriemanniangradient}  & $\O(Nn^2)$ \\ 
 \hline
 \stexttt{section}  & ---          &   \eqref{sec:manifold_env:algorithm:compactstiefelalgosection}                   & $\O(N(N-n)^2)$ \\
 \hline
 \stexttt{lift}     & ---          &   \eqref{sec:manifold_env:align_name:compactstiefellift+sectionproduct}          & $\O(N(N-n)n)$ \\
 \hline
 \stexttt{Adam}     & ---          &   \eqref{sec:application_env:algorithm_name:classicAdam}                         & $\O(Nn)$ \\
 \hline
 \stexttt{retract}  & Cayley       &   \eqref{sec:manifold_env:align_name:orthogonalgroupcayleyretraction}, \eqref{sec:manifold_env:align_name:orthogonalgroupSMWformulacayleywithrightmultiplication}, \eqref{sec:manifold_env:align_name:compactstiefelhomogeneouscayleyE}   & $\O(Nn^2)$ \\
\end{tabular}
\caption{Computational costs of the single components of an~update step of Adam on a~homogeneous space in Algorithm~\ref{sec:application_env:algorithm_name:PSDLayerUpdatestep}.}
\label{sec:application_env:table_name:costshomogeneousadamupdatestep}
\end{table}

\begin{table}[ht!]
\centering
\begin{tabular}{c|c|c|c} 
 \textbf{Step}  & \textbf{Configuration} & \textbf{Formula} & \textbf{Computational costs} \\
 \hline
 \stexttt{rgrad}           & Canonical                                             &  \eqref{sec:manifold_env:align_name:compactstiefelcanonicalriemanniangradient}              & $\O(Nn^2)$ \\ 
                          & Euclidean                                             &  \eqref{sec:manifold_env:align_name:compactstiefeleuclideanriemanniangradient}              & $\O(Nn^2)$ \\ 
 \hline
 \stexttt{StiefelAdam}     & ---                                                   &  \eqref{sec:application_env:algorithm_name:ModifiedAdam}                                    & $\O(Nn^2)$ \\
 \hline
 \stexttt{retract}         & Cayley                                                &  \eqref{sec:manifold_env:align_name:compactstiefelSMWformulacayleywithrightmultiplication}  & $\O(Nn^2)$ \\
 \hline
 \stexttt{transport}       & submanifold transport                                 &  \eqref{sec:manifold_env:align_name:compactstiefelvectranssub}                              & $\O(Nn^2)$ \\
                          & differentiated transport                              &  \eqref{sec:manifold_env:align_name:compactstiefelvectransdiffSMW}                          & $\O(Nn^2)$ \\
\end{tabular}
\caption{Computational costs of the single components of an~update step of StiefelAdam on $\St(n,N)$ in Algorithm~\ref{sec:application_env:algorithm_name:PSDLayerModifiedAdamStep}.}
\label{sec:application_env:table_name:costsstiefeladamupdatestep}
\end{table}

\subsubsubsection{Summary}
We summarize the default optimizer configuration of GML and our modifications and extensions in Table~\ref{sec:application_env:table_name:variantsNetwork}.
Note that in our manifold update step configuration, we have implemented two different choices for the Riemannian gradient 
and the vector transport (which we both tested) summing up to four possible parameter constellations. 
Since we tested our optimizer both with and without a~decay, we end up with a~total of eight possible constellations.
\begin{table}[ht!]
\centering
\begin{tabular}{|p{3.5cm}||p{4.7cm}|p{5cm}|} 
 \hline
 \multicolumn{3}{|c|}{Network configuration values} \\
 \hline
 \textbf{Element}  & \textbf{Reference configuration} & \textbf{New configuration} \\
 \hline
 Manifold update step        & Rgrad: Canonical                         & Rgrad: Canonical/Euclidean                      \\
                             & Optimizer: Adam                          & Optimizer: StiefelAdam                          \\
                             & Retraction $R_{I_N}^{\Orth(N)}$: Cayley      & Retraction $R_X^{\St(n,N)}$: Cayley             \\
                             & ---                                      & Vector transport: sub/diff                          \\
                             & ---                                      & Decay rate: ($\eta \leftarrow \eta \cdot 0.9995$)          \\
\hline
\hline
 \multicolumn{3}{|c|}{Functional GML modifications} \\
 \hline
 \textbf{Element}  & \textbf{Reference configuration} & \textbf{New configuration} \\
 \hline
 Manifold update step        & ---                                      & \small{\stexttt{optimizer.jl, retractions.jl, stiefel\_manifold.jl, init\_optimizer\_chache.jl, utils.jl, batch.jl}}                      \\
 \hline
 \hline
 \multicolumn{3}{|c|}{Computational costs} \\
 \hline
 \textbf{Element}  & \textbf{Reference configuration} & \textbf{New configuration} \\
 \hline
 Manifold update step        & $\O(N^2(N-n)+Nn^2)$      & $\O(Nn^2)$                 \\
 \hline
\end{tabular}
\caption{Variants of network configurations.}
\label{sec:application_env:table_name:variantsNetwork}
\end{table}

\subsubsection{Training configurations}

Regarding the training setup for our autoencoder network, we applied two changes in comparison to the reference setup in \cite{main:brantner2023}:
\begin{itemize}
  \item We applied an~epoch-wise training routine.
  \item We used normalized data.
\end{itemize}

\newpage
\subsubsubsection{Epoch-wise training}
At the time the reference configuration implemented in \cite[scripts/symplectic\_autoencoder]{githubGML:brantner2024} was published, the latest release 
version of GML was v$0.1.0$.
This version did not yet offer epoch-wise learning as it is suggested for the training of neural networks in general. 
First, we implemented an~exact reproduction of the learning setup presented in \cite{main:brantner2023}. 
Here, the training is conducted as follows:
\begin{itemize}
  \item[1.] Specify the number of \enquote{epochs} in the variable \stexttt{n\_epochs}.
  \item[2.] Specify a~batch size in the variable \stexttt{batch\_size}.
  \item[3.] Calculate the number of iterations: \\ $\tt n\_training\_iterations = \tt n\_epochs \cdot \tfrac{\text{\#(number of snapshots)}}{\tt batch\_size}$.
  \item[4.] Draw a~batch from the whole data set in each iteration \stexttt{1:n\_training\_iterations}.
\end{itemize}
The code for this is displayed in Figure~\ref{sec:application_env:jllisting:non-epoch-wisetraining}. 
The problem with this setup is that, unlike in normal epoch-wise training, we do not ensure that we see the entire data set once per epoch.
Here, we basically just use \stexttt{n\_epochs} as a~multiplicator of how often we want to draw a~batch from the whole data set.
Thus, we will refer to this setup as \textit{non-epoch-wise} training.
\begin{figure}[ht!] 
    \centering
  \begin{jllisting}
    n_training_iterations = Int(ceil(n_epochs*dl.n_params/batch_size))

    for _ in 1:n_training_iterations
      draw_batch!(batch, data)
      ...
    end
  \end{jllisting}
    \caption{Non-epoch-wise training.}
    \label{sec:application_env:jllisting:non-epoch-wisetraining}
\end{figure}

A~correct epoch-wise training routine was added in GML release version v$0.2.0$. In an~alternative configuration, we used the available 
GML building blocks to implement a~proper epoch-wise training routine for our network. The code for this is displayed in Figure~\ref{sec:application_env:jllisting:epoch-wisetraining}.
\begin{figure}[ht!] 
    \centering
  \begin{jllisting}
    for _ in 1:n_epochs
        optimize_for_one_epoch!(optimizer_instance, model, ps, dl, batch)
        ...
    end
  \end{jllisting}
    \caption{Epoch-wise training.}
    \label{sec:application_env:jllisting:epoch-wisetraining}
\end{figure}

\subsubsubsection{Training data}
The reference setup from \cite{main:brantner2023} uses unnormalized data. In our case, this means the training set
$$\X(\P_{\text{train}}) = \left\{ x^k_\mu \; \big| \; 0 \leq k \leq K, \mu \in \P_{\text{train}} \right\}$$
consists of unnormalized model snapshots $x^k_\mu$. Similar to the PSD method, a~reasonable choice for the reduced initial value for unnormlized  
snapshots is $x_{r,\mu}^0 :=e(x_\mu^0)$.

As stated in Section~\ref{sec:application_subsubsubsec:approachestosolvetheoptprob}, autoencoders should be trained with normalized data. 
Thus, we implemented an~alternative configuration using a~normalized training set 
$$\X(\P_{\text{train}}) = \left\{ x^k_\mu - x^0_\mu \; \big| \; 0 \leq k \leq K, \mu \in \P_{\text{train}} \right\}$$
as we proposed in \eqref{sec:application_env:align_name:trainingset}. Here, we choose $x_{r,\mu}^0 :=e(0)$.
The choice of normalized data itself did not affect GML because the choice of training data is not handled by GML.
But the different choices of the reduced initial value required some functional adaptions in \stexttt{reduced\_system.jl}.

\subsubsubsection{Summary}
We summarize the reference configuration and our modifications in Table~\ref{sec:application_env:table_name:variantsTraining}.
Note that we combine our normalized data configuration only with the ROM setup that uses a~reference state $x_{\reff,\mu}$ (see Table~\ref{sec:application_env:table_name:variantsFOMROM})
and the unnormalized setup only with the ROM that does not use a~reference state.
\begin{table}[ht!]
\centering
\begin{tabular}{|p{3.5cm}||p{4.7cm}|p{4.5cm}|} 
 \hline
 \multicolumn{3}{|c|}{Training configuration values} \\
 \hline
 \textbf{Element}  & \textbf{Reference configuration} & \textbf{New configuration} \\
 \hline
 Training data $\X$          & Unnormalized                 &  Normalized   \\
                             & $x_{r, \mu}^0=e(x_{\mu}^0)$  & $x_{r, \mu}^0=e(0)$  \\
 [0.2cm]
 Epoch-wise training         & No            & Yes           \\
\hline
\hline
 \multicolumn{3}{|c|}{Functional GML modifications} \\
 \hline
 \textbf{Element}  & \textbf{Reference configuration} & \textbf{New configuration} \\
 \hline
 Training data $\X$          & ---                    & \small{\stexttt{reduced\_system.jl}} \\
 [0.2cm]
 Epoch-wise training         & ---                    & ---                    \\
 \hline
\end{tabular}
\caption{Variants of training configurations.}
\label{sec:application_env:table_name:variantsTraining}
\end{table}

\subsubsection{Testing configurations}
After training the autoencoder, we perform a~quality assessment of the obtained approximations 
by computing two different approximation errors and comparing them with the results of a~reference method.

\subsubsubsection{Approximation errors}
Similar to the reference setup, we use the 
\begin{itemize}
  \item \textit{reduction error} $e_{\red}(\mu)$ and 
  \item \textit{projection error} $e_{\proj}(\mu)$
\end{itemize}
to estimate the quality of the obtained approximations.
The reduction error takes the discrete solutions of the ROM $(x_{r,\mu}^k)_{k=0}^K$ for a~fixed system parameter $\mu \in \P_{\train}$ 
and compares the reconstructions via the decoder $d$ to the exact solutions $(x_\mu^k)_{k=0}^K$ of the FOM. For the ROM setup without the reference state $x_{\reff,\mu}$, this is: 
$$e_{\red}(\mu) = \sqrt{\tfrac{\sum_{k=0}^{K}\norm{x_{\mu}^k-d \left(x_{r, \mu}^k\right)}^2_2 }{ \sum_{k=0}^{K} \norm{x_{\mu}^k}^2_2 }}.$$ 
For our proposed setup with the reference state, we get
$$e_{\red}(\mu) = \sqrt{ \frac{ \sum_{k=0}^{K} \norm{x_{\reff,\mu} + d \left(x_{r,\mu}^k\right) - x_{\mu}^k}^2_2}{ \sum_{k=0}^{K} \norm{x_{\mu}^k}^2_2} }.$$
The projection error takes the discrete exact solutions of the FOM $(x_\mu^k)_{k=0}^K$ for a~fixed system parameter $\mu \in \P_{\train}$ 
and computes the error between the exact solutions and the approximations obtained by the application of the target function $((d \circ e)(x_\mu^k))_{k=0}^K$. 
For the reference configuration without the reference state $x_{\reff,\mu}$, this is 
$$e_{\proj}(\mu) = \sqrt{ \tfrac{ \sum_{k=0}^{K} \norm{x_{\mu}^k - (d \circ e)\left(x_{\mu}^k\right)}^2_2 }{ \sum_{k=0}^{K} \norm{x_{\mu}^k}^2_2}}.$$
For the setup with the reference state, we have
$$e_{\red}(\mu) = \sqrt{ \frac{ \sum_{k=0}^{K} \norm{x_{\reff,\mu} + (d \circ e)\left(x_{\mu}^k - x_{\reff,\mu}\right) - x_{\mu}^k}^2_2}{ \sum_{k=0}^{K} \norm{x_{\mu}^k}^2_2} }.$$

\noindent Note that both errors are normalized by the norm of the exact FOM solutions. 
To get the new ROM configuration work with GML, we had to modify the GML-file \stexttt{reduced\_system.jl}.

\subsubsubsection{Comparison method}
We compare our error results (in all configurations) to the errors obtained for the PSD reduction method discussed in Section~\ref{sec:application_subsubsubsec:approachestosolvetheoptprob}.
This method  was also chosen for the reference configuration tested in \cite{main:brantner2023} using an~unnormalized snapshot matrix for the encoder and decoder 
calculation and a~ROM without a~reference state. 
In our implementation of the \stexttt{sineGordon} example, we conducted some tests with the PSD method working on normalized snapshots and a~ROM 
with reference state. But, since we wanted to keep a~consistent and unchanged comparison method, we did not include these tests in this thesis 
and decided to always use the PSD comparison method with unnormalized data and no reference state in all tested variants.
The comparison method is not part of GML, and thus we did not need to modify any GML-files.

\subsubsubsection{Summary}
We summarize the testing configurations in Table~\ref{sec:application_env:table_name:variantsTesting}.
\begin{table}[ht!]
\resizebox{\textwidth}{!}{%
\centering
\begin{tabular}{|p{4cm}||p{4.7cm}|p{5.2cm}|} 
 \hline
 \multicolumn{3}{|c|}{Testing configuration values} \\
 \hline
 \textbf{Element}  & \textbf{Reference configuration} & \textbf{New configuration} \\
 \hline 
 Reduction error $e_{\text{red}}(\mu)$            & {\small$\sqrt{\tfrac{\sum_{k=0}^{K}\norm{x_{\mu}^k-d \left(x_{r, \mu}^k\right)}^2_2 }{ \sum_{k=0}^{K} \norm{x_{\mu}^k}^2_2 }}$}     & {\small$\sqrt{ \frac{ \sum_{k=0}^{K} \norm{x_{\reff,\mu} + d \left(x_{r,\mu}^k\right) - x_{\mu}^k}^2_2}{ \sum_{k=0}^{K} \norm{x_{\mu}^k}^2_2} }$}               \\
 [0.3cm]
 Projection error $e_{\text{proj}}(\mu)$          & {\small$\sqrt{ \tfrac{ \sum_{k=0}^{K} \norm{x_{\mu}^k - (d \circ e)\left(x_{\mu}^k\right)}^2_2 }{ \sum_{k=0}^{K} \norm{x_{\mu}^k}^2_2}}$}               & {\small$\sqrt{ \frac{ \sum_{k=0}^{K} \norm{x_{\reff,\mu} + (d \circ e)\left(x_{\mu}^k - x_{\reff,\mu}\right) - x_{\mu}^k}^2_2}{ \sum_{k=0}^{K} \norm{x_{\mu}^k}^2_2} }$}   \\
 [0.5cm]
 Comparison method                                & PSD method                                                                                                                                & PSD method                                                                                                                                 \\
\hline
\hline
 \multicolumn{3}{|c|}{Functional GML modifications} \\
 \hline
 \textbf{Element}  & \textbf{Reference configuration} & \textbf{New configuration} \\
 \hline
 Reduction error             & ---                    & \small{\texttt{reduced\_system.jl}}  \\
 [0.2cm]
 Projection error            & ---                    & \small{\texttt{reduced\_system.jl}}  \\
 [0.2cm]
 Comparison method           & ---                    & ---  \\
 \hline
\end{tabular}
}
\caption{Variants of testing configurations.}
\label{sec:application_env:table_name:variantsTesting}
\end{table}

\newpage
\subsubsection{Complete learning configurations}
We display some complete learning setups in Table~\ref{sec:application_env:table_name:completelearningsetups}. 
The first configuration V$1$, referred to as \enquote{Non-epoch-wise+unnormalized+lossGML}, is the exact reference setup described in \cite{main:brantner2023} and 
implemented in \cite[scripts/symplectic\_autoencoder]{githubGML:brantner2024}. \\
Configuration V$2$ (\enquote{Epoch-wise+unnormalized+lossGML}) changes the training routine to an~epoch-wise training. Setup V$3$ (\enquote{Epoch-wise+normalized+lossGML}) 
keeps the epoch-wise routine and additionally uses normalized data for the training. As proposed by \cite{weaklysymp:buchfink2021}, we combine this 
with the new ROM configuration using the reference state and consequently using the new reduction and projection errors in the testing stage,
which make use of the reference state.  Since our numerical experiments (see chapter~\ref{sec:numericalexperiments}) yield better results 
for this configuration than V$1$ (and V$2$), we keep this configuration as the new basis for further variants.
Variant V$4$ (\enquote{Epochs-wise+normalized+lossTheory}) tries the alternative loss we proposed for the theoretical problem setup. \\
Configurations V$5$-V$10$ all use the new optimizer update step StiefelAdam in different variants. 
Since the numerical experiments did not yield any improvements for the alternative loss, we returned to the original relative-error loss function implemented in GML for most of these variants. \\
First, in variant V$5$ (\enquote{StiefelAdam(can,sub)+epochs-wise+normalized+lossGML}), we apply the StiefelAdam with the canonical Riemannian gradient and the submanifold vector transport without a decay. \\
Next, in V$6$ (\enquote{StiefelAdamWithDecay(can,sub)+epochs-wise+normalized+lossGML}), we add a~decay. To highlight the difference to StiefelAdam without decay, 
we named this optimizer \linebreak StiefelAdamWithDecay.
In variants V$7$-V$9$, we try all other combinations of the Riemannian gradients and vector transports. 
Since V$6$ with decay worked far better than V$5$ without decay, we used StiefelAdamWithDecay for all these variants. \\
\noindent Finally, in variant V$10$ (\enquote{StiefelAdamWithDecay(can,sub)+epochs-wise+normalized+lossTheory}), we combine the alternative loss with the new optimizer approach.
\begin{table}[ht!]
\definecolor{lightblue}{RGB}{173, 216, 230}
\rowcolors{3}{lightblue}{lightblue}
\resizebox{\textwidth}{!}{%
\centering
\begin{tabular}{c|c|c|c|c|c|c} 
   \textbf{Variant}       & \multicolumn{2}{c|}{\textbf{Problem setup}}       & \textbf{Network}                                      & \multicolumn{2}{c|}{\textbf{Training}}                          & \textbf{Testing} \\
                          & \textbf{ROM}          & \textbf{Loss}             &                                                       & \textbf{Data}                  & \textbf{Epoch}                 &                  \\
 \hline
 \cellcolor{white} V$1$     & Ref.                  & Ref.                      & Ref.                                                  & Ref.                           & Ref.                           & Ref.        \\ 
 \hline               
 \cellcolor{white} V$2$     & Ref.                  & Ref.                      & Ref.                                                  & Ref.                           & \cellcolor{green} New          & Ref.        \\ 
 \hline               
 \cellcolor{white} V$3$     & \cellcolor{green} New & Ref.                      & Ref.                                                  & \cellcolor{green} New          & \cellcolor{green} New          & \cellcolor{green} New \\ 
 \hline               
 \cellcolor{white} V$4$     & \cellcolor{green} New & \cellcolor{green} New     & Ref.                                                  & \cellcolor{green} New          & \cellcolor{green} New          & \cellcolor{green} New \\ 
 \hline               
 \cellcolor{white} V$5$     & \cellcolor{green} New &  Ref.                     & \cellcolor{green} New (canonical,sub-trans)           & \cellcolor{green} New          & \cellcolor{green} New          & \cellcolor{green} New \\ 
 \hline               
 \cellcolor{white} V$6$     & \cellcolor{green} New &  Ref.                     & \cellcolor{green} New (canonical,sub-trans,decay)     & \cellcolor{green} New          & \cellcolor{green} New          & \cellcolor{green} New \\ 
 \hline               
 \cellcolor{white} V$7$     & \cellcolor{green} New &  Ref.                     & \cellcolor{green} New (canonical,diff-trans,decay)    & \cellcolor{green} New          & \cellcolor{green} New          & \cellcolor{green} New \\ 
 \hline               
 \cellcolor{white} V$8$     & \cellcolor{green} New &  Ref.                     & \cellcolor{green} New (euclidean,sub-trans,decay)     & \cellcolor{green} New          & \cellcolor{green} New          & \cellcolor{green} New \\ 
 \hline               
 \cellcolor{white} V$9$     & \cellcolor{green} New &  Ref.                     & \cellcolor{green} New (euclidean,diff-trans,decay)    & \cellcolor{green} New          & \cellcolor{green} New          & \cellcolor{green} New \\ 
 \hline                
 \cellcolor{white} V$10$    & \cellcolor{green} New &  \cellcolor{green} New    & \cellcolor{green} New (canonical,sub-trans,decay)     & \cellcolor{green} New          & \cellcolor{green} New          & \cellcolor{green} New \\ 
 \hline
\end{tabular}
}
\caption{Overview of tested configurations.}
\label{sec:application_env:table_name:completelearningsetups}
\end{table}

\newpage
\subsubsection{Nonfunctional GML modifications}
In the course of the GML adjustments, we also had to make some nonfunctional changes. These include
\begin{itemize}
  \item adjustments that provide more process information, such as integration time and 3D plots of the reconstructed solutions,
  \item technical adjustments that made it possible to integrate the use of different options to choose the Riemannian gradient, 
  the vector transport and the loss function,
  \item code errors that we needed to resolve to make the code work with our data setup, etc. and
  \item efficiency improvements that we implemented to make some parts of GML work faster.
\end{itemize}
We do not want to describe all of these adjustments in detail here. We only point out some of our efficiency changes:
\begin{itemize}
  \item \stexttt{adam\_optimizer.jl}: A~classical Adam update step requires the exponential $\beta_i^t$ of the Adam parameters $\beta_i$ for $i=1,2$ in every iteration $t$. 
  In GML, this is done by calculating $\beta_i^t$ several times in each step. This becomes an~expensive task especially for large iteration numbers $t$. 
  We did not change this for the built-in GML optimizer step since we did not consider the efficiency gain to have a large enough impact to justify the modification of the GML library here. 
  But for our StiefelAdam optimizer step, where we had to implement the optimizer step anyway, we changed this by storing two more Adam parameters 
  $(\beta\_t)_i$ for $i=1,2$, and updating these parameters in every step by a~single multiplication operation 
  $$ (\beta\_t)_i \leftarrow (\beta\_t)_i \cdot \beta_i$$
  implemented in \stexttt{fn\_optimizer.jl}. Note that in GML (and our adjustments) the Adam parameters $\beta_i$ are named $\rho_i$.
  \item \stexttt{reduced\_system.jl}: The method \\ 
  \stexttt{reduced\_vector\_field\_from\_full\_explicit\_vector\_field(...)}
  conducts full matrix multiplications with the Poisson matrix (both in the full dimension $2N$ and the reduced dimension $2n$) to obtain the reduced Hamiltonian. 
  This becomes very expensive for high dimensions $N \in \N$, 
  especially since the computation of the vector field is needed for the reduced integration. We changed this calculation by using the 
  shape of the Poisson matrix, see Figure~\ref{sec:application_env:jllisting:largematrixmulit}. 
  Note that this inefficiency has been removed in GML release version v$0.3.0$.
\begin{figure}[ht!] 
    \centering
  \begin{jllisting}
function reduced_vector_field_from_full_explicit_vector_field(full_explicit_vector_field, decoder, N::Integer, n::Integer)
    function reduced_vector_field(v, t, ξ, params)
        x = full_explicit_vector_field(t, decoder(ξ), params)
        Del = ForwardDiff.jacobian(decoder, ξ)'
        v .=  vcat(-Del[n+1:2*n,:], Del[1:n,:]) * vcat(x[N+1:2*N],-x[1:N])
        # Implementation in v0.2.0:
        # v .= -SymplecticPotential(2*n) * ForwardDiff.jacobian(decoder, ξ)' * SymplecticPotential(2*N) * full_explicit_vector_field(t, decoder(ξ), params)
    end
    reduced_vector_field
end
  \end{jllisting}
    \caption{Efficiency improvements applied in GML to compute the reduced Hamiltonian.}
    \label{sec:application_env:jllisting:largematrixmulit}
\end{figure}
\end{itemize}

\clearpage
\section{Numerical experiments} \label{sec:numericalexperiments}

In this section, we present the numerical results obtained for the different learning setups we tested. 
First, we compare the two different manifold update step approaches in terms of numerical efficiency.
Second, we present the results for the \textit{1D linear wave equation} for different learning configurations. 
Then we do the same for the \textit{1D sine-Gordon equation}.

\subsection{Numerical efficiency of the manifold update step} \label{sec:numericalexperiments_subsec:speedtest}
We compared the numerical efficiency of a~single manifold update step using the HomogeneousAdam optimizer implemented in GML 
to our update step approach using the StiefelAdamOptimizerWithDecay in all implemented Riemannian metric and vector transport combinations.
To do this, we created two scripts 
\begin{itemize}
  \item \stexttt{fn\_speed\_test\_homogeneous\_optimizer\_step.jl} and 
  \item \stexttt{fn\_speed\_test\_alternative\_optimizer\_step.jl}.
\end{itemize}
The setup of these files is pretty similar: First, we chose a~pair $(N,n)$ of a~very large full dimension $N$ and a~small reduced dimension $n\ll N$. 
Next, we instantiate the required objects for each optimizer step with standard parameters and initialize  
a~PSDLayer with the default initial values. Then we add a~one-matrix as the update vector used by the (modified) Adam algorithm.
Finally, we conduct two update iterations using the \stexttt{optimization\_step!} function defined separately for each setup.  
When the function is called for the first time, all objects must be loaded initially. 
This distorts the actual calculation time of an~update step. Therefore, we first performed a~pre-iteration before we measured the time of an~\stexttt{optimization\_step!} iteration. 
We displayed these setups in Figure~\ref{sec:application_env:jllisting:speedtestAdamHom} and Figure~\ref{sec:application_env:jllisting:speedtestStiefelAdam}.

The test results for different choices of $N$ and $n$ are displayed in Table~\ref{sec:application_env:table_name:speedtestResults}. 
The calculations were performed on a~\textit{NVIDIA Quadro T2000} GPU for Datatype \stexttt{Float64}. 
We can observe that all StiefelAdam update step variants are multiple times faster than the homogeneous Adam update step for all tested dimension pairs $(N,n)$.
The time difference gets larger the smaller $n$ gets compared to $N$. This is in line with the calculated costs per update step: 
For one homogeneous space update step, we have costs of $\O(N^2(N-n)+Nn^2)$, see Table~\ref{sec:application_env:table_name:variantsNetwork}, 
whereas one StiefelAdam update step costs $\O(Nn^2)$.
Thus, we get the largest speed difference in our test for $N=10000$ and $n=10$.
Another interesting observation is that doubling the full dimension $N$ did not cause any significant speed reduction of the StiefelAdam approach for all variants, 
whereas for the HomogeneousAdam approach, we exceeded the available GPU memory of $4$ Gigabyte, referred to as \textit{GPU Out of Memory} (GPU OOM), and thus did not get any results.
Comparing the different variants of StiefelAdam, we see that the vector transport is the dominating cost factor. The variants using the 
submanifold vector transport perform worse than the differential vector transport variants for very small reduced dimensions $n$, but perform better for increasing reduced dimensions.
\begin{figure}[ht!] 
    \centering
  \begin{jllisting}
# one pre-iteration as the first iteration always takes longer than a~normal intermediate iteration
optimization_step!(optimizer_instance, TESTPSDLayer, ps, optimizer_instance.cache, dx)

println("Calculating HomogeneousAdam step...")
timer = 
@timed begin
optimization_step!(optimizer_instance, TESTPSDLayer, ps, optimizer_instance.cache, dx)
end

println("***Time needed: (timer.time)***")
println()
  \end{jllisting}
    \caption{Speed test setup for a~single manifold update step using Adam on homogeneous spaces.}
    \label{sec:application_env:jllisting:speedtestAdamHom}
\end{figure}

\begin{figure}[ht!] 
    \centering
  \begin{jllisting}
# one preiteration as the first iteration always takes longer than a~normal intermediate iteration
optimization_step!(optimizer_instance, met, vt, TESTPSDLayer, ps, optimizer_instance.cache, dx)

println("Calculating StiefelAdamWithDecay step...")
timer = 
@timed begin
optimization_step!(optimizer_instance, met, vt, TESTPSDLayer, ps, optimizer_instance.cache, dx)
end

println("***Time needed: (timer.time)***")
println()
  \end{jllisting}
    \caption{Speed test setup for a~single manifold update step using StiefelAdamWithDecay on the Stiefel manifold directly.}
    \label{sec:application_env:jllisting:speedtestStiefelAdam}
\end{figure}
\begin{table}[ht!]
\resizebox{\textwidth}{!}{%
\centering
\begin{tabular}{llllll}
  \hline
  & \textbf{$(10000,10)$} & \textbf{$(20000,10)$} & \textbf{$(30000,10)$} & \textbf{$(10000,100)$} & \textbf{$(10000,1000)$}  \\
  \hline  
  HomogeneousAdam         & 13.788589597          & GPU OOM          & GPU OOM           & 14.292333164          & GPU OOM  \\
  StiefelAdamWithDecay(can, sub)    & 0.063947959          & 0.064044896           & 0.065288955            & 0.100374425           & 3.727560833  \\
  StiefelAdamWithDecay(can,diff)    & 0.01757327           &  0.022213855          & 0.020285467            & 0.147214102           & 6.709052409  \\
  StiefelAdamWithDecay(eucl,sub)    & 0.061689715          &  0.060874198         & 0.060847897            & 0.099162161           & 3.614985748  \\
  StiefelAdamWithDecay(eucl,diff)   & 0.016686088          &  0.022605729         & 0.019532122            & 0.159957555           & 7.009566782  \\
  \hline
\end{tabular}
}
\caption{Computational time for a~single manifold update step for different choices of $(N,n)$ in seconds.}
\label{sec:application_env:table_name:speedtestResults}
\end{table}

\subsection{1D linear wave equation} \label{sec:numericalexperiments_subsec:waveequation}

We consider a~1D linear wave equation  
\begin{align} \label{sec:numericalexperiments_env:align_name:pde1Dwaveequation}
  \partial_{tt}^2u(t,\xi) &=\mu^2 \partial_{\xi \xi}^2u(t,\xi)                            & \quad\quad &\text{in } \I \times \Omega,  \\
  u(0,\xi) &= u_0(\xi) & \quad\quad &\text{on } \Omega, \nonumber \\
  u_t(0,\xi) &= u_1(\xi) & \quad\quad &\text{on } \Omega, \nonumber \\
  u(t,a) &= \varphi(t) & \quad\quad &\text{on } \I, \nonumber \\
  u(t,b) &= \psi(t) & \quad\quad &\text{on } \I \nonumber  
\end{align}
with initial conditions determined by $u_0(\xi), u_1(\xi)$ and boundary conditions determined by $\varphi(t),\psi(t)$.
We define 
\begin{itemize}
  \item $q(t,\xi) := u(t,\xi)$ and 
  \item $p(t,\xi) := \partial_t q(t,\xi) = \partial_t u(t,\xi)$
\end{itemize}
(and thus $\partial_tp(t,\xi) = \mu^2 \partial_{\xi \xi}^2u(t,\xi) = \mu^2 \partial_{\xi \xi}^2q(t,\xi)$) 
and obtain the \textit{continuous Hamiltonian} 
\begin{align} \label{sec:numericalexperiments_env:align_name:continuoushamiltonianlinearwave}
  \H_{cont}(q,p) &= \frac{1}{2}\int_{\Omega}  \mu^2 \left(\partial_{\xi} q(t,\xi)\right)^2 + p(t,\xi)^2  d\xi. 
\end{align}

\subsubsection{Spatial discretization}

Given an~$N\in \N$, we discretize the spatial domain $\Omega=[a,b]$ into $(N+2)$ points $\xi_i := ih_{N} + a$ for $i = 0,...,N+1$ and $h_{N} := \tfrac{b-a}{N+1}$. 
For the discretization of $q(t,\xi)$ and $p(t,\xi)$ we define:
\begin{itemize}
  \item $q_i(t) := q(t,\xi_i) \quad i=0,...,N+1$,
  \item $p_i(t) := p(t,\xi_i) \quad i=0,...,N+1$.
\end{itemize} 
We want to obtain the discretized problem as an~ODE of the form
\begin{align} \label{ode:generaldiscretizedproblem}
  \begin{cases}
    \dot{q} = p \\
    \dot{p} = \L q \\
    q(0) = q_0, p(0) = p_0,
  \end{cases}
\end{align}
where we set 
$q := [q_0,...,q_{N+1}]^T$ and $p := [p_0,...,p_{N+1}]^T$. 
Here, $\L$ denotes a~linear operator that relies on the method used to discretize the continuous Hamiltonian (see \eqref{sec:numericalexperiments_env:align_name:continuoushamiltonianlinearwave}) and the initial and boundary
conditions used. We use the \textit{rectangular rule} to approximate the integral:
\begin{align} 
  \H_{cont}(q,p) 
  =& \frac{1}{2} \int_{\Omega} \mu^2 \left(\partial_{\xi} q(t,\xi)\right)^2 + p(t,\xi)^2 d\xi \label{calc:midpointrule}\\
  \approx& \frac{h_N}{2} \sum_{i=1}^{N} \left( \mu^2 \left(\partial_{\xi} q_i(t)\right)^2 +  p_i(t)^2 \right) + \frac{h_N}{2}(\varphi(t) + \partial_t\varphi(t)).   \nonumber 
\end{align}
Before we continue, we first need to find an~approximation for $\left(\partial_{\xi} q_i(t)\right)^2$. We use a~combination of \textit{forward} and 
\textit{backward difference methods} to obtain
$$2 \left(\partial_{\xi} q_i(t)\right)^2 \approx \left(\frac{q_{i}(t)- q_{i-1}(t)}{h_N}\right)^2 + \left(\frac{q_{i+1}(t)- q_i(t)}{h_N}\right)^2, \quad i=1,...,N.$$
Now we apply this to \eqref{calc:midpointrule} and get
\begin{align*}
  \H_{cont}(q,p)
 &\approx\frac{h_N}{2} \sum_{i=1}^{N} \left( \mu^2 \frac{(q_i(t)-q_{i-1}(t))^2 + (q_{i+1}(t)-q_{i}(t))^2}{2h_N^2} +  p_i(t)^2 \right)  \\
  &+\frac{h_N}{2}(\varphi(t) + \partial_t\varphi(t)) \nonumber \\
  &= q(t)^TKq(t) + \frac{h_N}{2} p(t)^Tp(t) +\frac{h_N}{2}(\varphi(t) + \partial_t\varphi(t)) \nonumber \\
  &=: \H_{h_N}(q,p), \nonumber
\end{align*}
where we set
\begin{align*}
  K &:=
  \frac{\mu^2}{h_N}
\begin{bmatrix}
  \tfrac{1}{4}  &              &              & &  \\
  -\frac{1}{2}  & \tfrac{3}{4} & -\frac{1}{2} & &  \\
   & \ddots & \ddots & \ddots & \\
   & &-\frac{1}{2}  & \tfrac{3}{4} & -\frac{1}{2}  \\
    &              &              & &\tfrac{1}{4}  \\
\end{bmatrix} \in \R^{N+2}.
\end{align*}
Calculating the gradient of $\H_{h_N}(q,p)$ with respect to $[q^T,p^T]^T$ gives us
\begin{align*}
  \nabla\H_{h_N}(q,p) &= 
  \begin{bmatrix}
     K + K^T && 0 \\
    0 && h_NI_{N+2} \\
  \end{bmatrix}
  \begin{bmatrix}
    q \\ p
  \end{bmatrix}.
\end{align*}  
We now obtain a~\textit{Hamiltonian system} of the form \eqref{ode:generaldiscretizedproblem}. For $\L := -\tfrac{1}{h_N}(K + K^T)$, we get
\begin{align} 
  \begin{bmatrix}
    \dot{q} \\ \dot{p}
  \end{bmatrix}
  &=
  \begin{bmatrix}
    p \\ \L q  
  \end{bmatrix} \label{formula:almosthamiltonianwaveequation} \\
  &= J_{2(N+2)} 
  \begin{bmatrix}
    -\L & 0 \\
    0 & I_{N+2} \\
  \end{bmatrix}
  \begin{bmatrix}
    q \\ p
  \end{bmatrix} \nonumber \\
  &= J_{2(N+2)} \frac{1}{h_N} \nabla \H_{h_N}(q,p). \nonumber
\end{align}
This system almost looks like a~Hamiltonian system for the discretized wave equation. The only problem remaining is the scaling 
factor $\tfrac{1}{h_N}$. To get rid of this, we can apply a~so-called \textit{state-space transformation} by redefining $q,p$ as
\begin{align*} 
  \begin{bmatrix}
    \hat{q} \\ \hat{p}
  \end{bmatrix}
  &:= \sqrt{h_N}
  \begin{bmatrix}
    q \\ p
  \end{bmatrix}
  ,\qquad
  \begin{bmatrix}
    q \\ p
  \end{bmatrix}
  = \sqrt{\frac{1}{h_N}}
  \begin{bmatrix}
    \hat{q} \\ \hat{p}
  \end{bmatrix}.
\end{align*}
This leads to a~\textit{transformed Hamiltonian} without a~scaling factor. 
From a~numerical perspective it does not make any difference if we solve the transformed Hamiltonian 
system or the not transformed \enquote{almost} Hamiltonian system.  
Thus, analogous to \cite{main:brantner2023}, we decided to implement the not transformed system \eqref{formula:almosthamiltonianwaveequation} in our implementation project 
and omitted the details of the transformation.

\subsubsection{Initial and boundary conditions}

Let $\I = [0,T]$ denote the time interval and $\Omega = [a,b]$ the spatial domain. 
If we choose the initial conditions (IC) 
\begin{itemize}
  \item $u_0(\xi) = h(s(\xi))$  
   with
  $ s(\xi) =
    28 |\xi + \frac{1}{2}|$
  and
  $h(s) =
        \begin{cases}
          1 - \frac{3}{2} s^2 + \frac{3}{4}s^3 & 0 \leq s \leq 1 \\
          \frac{1}{4}(2-s)^3 & 1 < s \leq 2 \\
          0 & \text{else,} 
        \end{cases}$,
  \item $u_1(\xi) = -\mu\partial_{\xi} u_0(\xi)$
\end{itemize}
and the boundary conditions (BC)
$$\varphi(t) \equiv \psi(t) \equiv 0$$
we get the so-called \textit{homogeneous Dirichlet} boundary conditions. Similar to \cite{main:brantner2023}, we have used these conditions in our implementation.
Defining $x_\mu(t) := [q(t)^T,p(t)^T]^T \in \R^{2(N+2)}$, we have to solve the $\mu$-dependent FOM
\begin{align}
  \dot{x}_{\mu} &= J_{2(N+2)} \frac{1}{h_{N}} \nabla_{x}\H_{h_N} (x_{\mu}; \mu) x_{\mu} \label{ode:examplewaveequation} \\ 
  &=  
  \begin{bmatrix}
    0 && I_{N+2} \\
    -\frac{1}{h_{N}}(K + K^T) && 0 \\
  \end{bmatrix}
  x_{\mu}, \nonumber \\
  x_{\mu}(0) &= x_{\mu}^0 = 
  \begin{bmatrix}
    (h(s(\xi_i,\mu)))_{i=0}^{N+1} \\ (-\mu \partial_{\xi}h(s(\xi_i, \mu)))_{i=0}^{N+1} \\
  \end{bmatrix} \in \R^{2(N+2)}. \nonumber 
\end{align}

\subsubsection{Implementation}

\begin{figure}[ht]
    \centering
    \begin{tikzpicture}[scale=0.7] 
    \foreach \col in {0, 5, 10} {
        \draw (\col, 0) rectangle (\col + 5, -6);
    }
    \node at (2.5, -3) {\small{solution for $\mu_1$}};
    \node at (7.5, -3) {...};
    \node at (12.5, -3) {\small{solution for $\mu_{\tt{n\_params}}$}};
    \draw[decorate,decoration={brace,amplitude=10pt,mirror}] (0, -6.5) -- (15, -6.5) node[midway, yshift=-15pt] {$\tt{n\_params}$ $\times (K+1)$};
    \draw[decorate,decoration={brace,amplitude=10pt,mirror}] (-0.5, 0) -- (-0.5, -6) node[midway, xshift=-15pt, rotate=90] {$2(N+2)$};
    \end{tikzpicture}
    \caption{A~schematic depiction of the $2D$-training array with a~specified number \stexttt{n\_params} of variable system 
    parameters $\mu$ and time steps $K$.}
    \label{fig:schematic_data_matrix}
\end{figure}
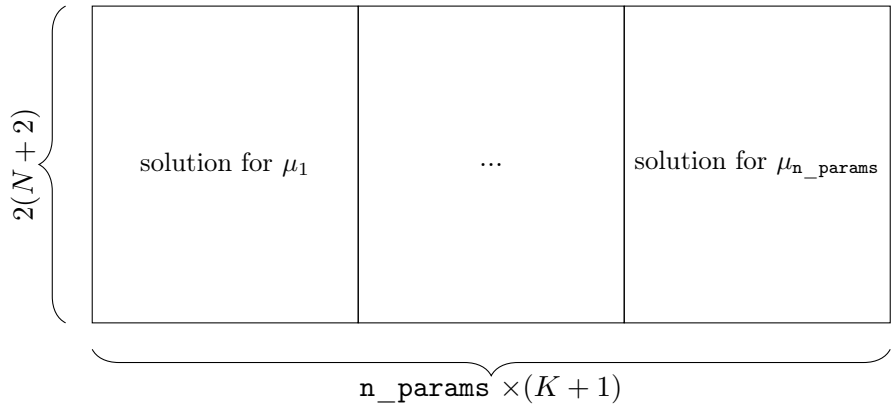
\noindent We implemented all 1D linear wave equation related objects like the Hamiltonian and the initial conditions in 
the folder {\small\texttt{reproduction/general}}. \\
We use integrated solutions for the training data obtained by the \textit{implicit midpoint integrator}. 
The data is generated in {\small\texttt{reproduction/main/fn\_integration\_v020.jl}} and stored as a~$2$D-array. 
The number of rows is the full system dimension $2(N+2)$. The columns are composed of the number of chosen $\mu$-parameters times
the discrete $\mu$-dependent solutions at each time point $t_k = \tfrac{T}{K+1}$, where $K$ denotes the number of discrete time steps chosen. 
A~schematic depiction of the training data matrix is displayed in Figure~\ref{fig:schematic_data_matrix}. \\
The file {\small\texttt{reproduction/main/fn\_generate\_data\_normalized.jl}} produces normalized training data from the 
unnormalized training data matrix.
The errors of the comparison method are computed in \stexttt{reproduction/main/fn\_generate\_psd.jl}. 
The actual training is performed in \stexttt{fn\_training\_v020\_reproduction.jl} with the GML built-in homogeneous space optimizer and in 
\stexttt{fn\_training\_v020\_alternative\_optimizer\_reproduction.jl} with our alternative optimizer approach.
Note that the files containing \enquote{\_v020} in its name are modified versions of existing files from the reference setup implementation 
in \cite[scripts/symplectic\_autoencoder]{githubGML:brantner2024}.

The necessary system parameters and the choices of the training variants (see Table~\ref{sec:application_env:table_name:completelearningsetups})
need to be specified in 
\begin{itemize}
  \item \scripttexttt{reproduction/general/fn\_global\_consts.jl} and
  \item \scripttexttt{alternative\_optimizer/general/fn\_global\_consts\_alternative\_optimizer\_specific.jl}.
\end{itemize}
We summarize the variant specific options in Table~\ref{sec:application_env:table_name:globalconstsrepro} and the 
1D linear wave equation specific parameters in Table~\ref{sec:application_env:table_name:globalconstsrepro1Dspecific}.
The time interval $\I$ and the spatial domain $\Omega$ cannot be chosen as a~parameter. Analogous to \cite{main:brantner2023}, we 
decided to use $\I = [0,1]$ and $\Omega = \left[-\tfrac{1}{2}, \tfrac{1}{2}\right]$ in all test runs for the $1$D linear wave equation.
We use the datatype \stexttt{Float64} for the training snapshots as well as for the actual learning and the testing.
A~description of the chronological steps to be taken to conduct a~training run is displayed in {\small \stexttt{fn\_global\_consts.jl}}.
\begin{table}[ht!]
\resizebox{\textwidth}{!}{%
\centering
\begin{tabular}{c|c|c|c|c|c|c|c|c|c|c} 
 \textbf{Option} & \multicolumn{10}{c}{\textbf{Variant}} \\
                                       & \textbf{V1}       & \textbf{V2}       & \textbf{V3}    & \textbf{V4}     & \textbf{V5}        & \textbf{V6}       & \textbf{V7}       & \textbf{V8}       & \textbf{V9}     & \textbf{V10}\\
 \hline 
 & \multicolumn{10}{c}{{\small \stexttt{fn\_global\_consts.jl}}} \\
 \hline                                                                          
\texttt{LossGML\_global}               & \texttt{true}     & \texttt{true}     & \texttt{true}  & \texttt{false}  & \texttt{true}     & \texttt{true}     & \texttt{true}     & \texttt{true}     & \texttt{true}   & \texttt{false}\\
 \hline                                                                                
\texttt{LossModified\_global}          & \texttt{false}    & \texttt{false}    & \texttt{false} & \texttt{true}   & \texttt{false}    & \texttt{false}    & \texttt{false}    & \texttt{false}    & \texttt{false}  & \texttt{true}\\
 \hline 
\texttt{epochwise\_global}             & \texttt{false}    & \texttt{true}     & \texttt{true}  & \texttt{true}   & \texttt{true}     & \texttt{true}     & \texttt{true}     & \texttt{true}     & \texttt{true}   & \texttt{true}\\
 \hline                                                                                
\texttt{normalized\_global}            & \texttt{false}    & \texttt{false}    & \texttt{true}  & \texttt{true}   & \texttt{true}     & \texttt{true}     & \texttt{true}     & \texttt{true}     & \texttt{true}   & \texttt{true}\\
 \hline                                                                          
 & \multicolumn{10}{c}{{\small \texttt{fn\_global\_consts\_alternative\_optimizer\_specific.jl}}} \\
 \hline                                                                          
\texttt{StiefelAdam\_global}           & \texttt{---}    & \texttt{---}        & \texttt{---}   & \texttt{---}    & \texttt{true}     & \texttt{false}    & \texttt{false}    & \texttt{false}    & \texttt{false}  & \texttt{false}\\
 \hline                                                                              
\texttt{StiefelAdamWithDacay\_global}  & \texttt{---}    & \texttt{---}        & \texttt{---}   & \texttt{---}    & \texttt{false}    & \texttt{true}     & \texttt{true}     & \texttt{true}     & \texttt{true}   & \texttt{true}\\
 \hline                                                                              
\texttt{CanonicalMetric\_global}       & \texttt{---}    & \texttt{---}        & \texttt{---}   & \texttt{---}    & \texttt{true}     & \texttt{true}     & \texttt{true}     & \texttt{false}    & \texttt{false}  & \texttt{true}\\
 \hline                                                                                   
\texttt{EuclideanMetric\_global}       & \texttt{---}    & \texttt{---}        & \texttt{---}   & \texttt{---}    & \texttt{false}    & \texttt{false}    & \texttt{false}    & \texttt{true}     & \texttt{true}   & \texttt{false}\\
 \hline                                                                              
\texttt{SubmanifoldTransport\_global}  & \texttt{---}    & \texttt{---}        & \texttt{---}   & \texttt{---}    & \texttt{true}     & \texttt{true}     & \texttt{false}    & \texttt{true}     & \texttt{false}  & \texttt{true}\\
 \hline                                                                              
\texttt{DifferentialTransport\_global} & \texttt{---}    & \texttt{---}        & \texttt{---}   & \texttt{---}    & \texttt{false}    & \texttt{false}    & \texttt{true}     & \texttt{false}    & \texttt{true}   & \texttt{false}\\
\end{tabular}
}
\caption{Learning options to be set for the different variants tested.}
\label{sec:application_env:table_name:globalconstsrepro}
\end{table}

\begin{table}[ht!]
\resizebox{\textwidth}{!}{%
\centering
\begin{tabular}{c|p{10.5cm}} 
 \textbf{Parameter}             & \textbf{Explanation}         \\
 \hline 
 {\small\texttt{N\_global}}                        & Number of spatial discretization points $N$; the full dimension of the FOM is $2(N+2)$.                                                                                          \\
 \hline 
 {\small\texttt{n\_range\_global}}                 & Range of reduced dimensions that are supposed to be tested.                                                                                                                       \\
 \hline 
 {\small\texttt{n\_epochs\_global}}                & Number of epochs to be learned. If non-epoch-wise learning is configured, this is a~multiplicator for the number of times we draw a~batch.                                         \\
 \hline 
 {\small\texttt{batch\_size\_global}}              & The size of a~single batch to be drawn.                                                                                                                                          \\
 \hline 
 {\small\texttt{time\_steps\_global}}              & The number of discrete time steps equally distributed over the time interval $\I$.                                                                                                \\
 \hline 
 {\small\texttt{$\mu$\_left\_global}}              & The left boundary of the interval $\P$ from which the variable system parameter $\mu$ is to be chosen.                                                                           \\
 \hline 
 {\small\texttt{$\mu$\_right\_global}}             & The right boundary of the interval $\P$ from which the variable system parameter $\mu$ is to be chosen.                                                                          \\
 \hline 
 {\small\texttt{n\_params\_global}}                & The number of variable system parameters $\mu$ to build the training data. The values are equally distributed  over the interval ${\P :=[\mu_{\text{left}}, \mu_{\text{right}}]}$.  \\
 \hline 
 {\small\texttt{$\mu$\_testing\_range\_global}}    & The range of $\mu$ values to test the quality of the learned autoencoder.                                                                                                         \\
\end{tabular}
}
\caption{Summary of all $1$D linear wave equation system parameters needed to be specified in the file {\small\texttt{fn\_global\_consts.jl}}.}
\label{sec:application_env:table_name:globalconstsrepro1Dspecific}
\end{table}

\subsubsection{Numerical results}
We summarize all test runs and the obtained results in Table~\ref{sec:application_env:table_name:waveequationtestruns}. 
\begin{table}[ht!]
\resizebox{\textwidth}{!}{%
\centering
\begin{tabular}{c|c|c|c|c|c|c|c|c|c|c|c} 
 \textbf{Test run}      & \textbf{Variant}    & \multicolumn{9}{c}{\textbf{1D linear wave equation parameter values}}  & \textbf{Results}       \\
                        &                     & {\small\texttt{N}} & {\small\texttt{n\_range}} & {\small\texttt{epochs}}& {\small\texttt{batch}} & {\small\texttt{time\_steps}} & {\small\texttt{$\mu_{\text{left}}$}} & {\small\texttt{$\mu_{\text{right}}$}} & {\small\texttt{params}} & {\small\texttt{$\mu$\_testing\_range}} & \\
 \hline 
 Run 1                  & V$1$                & $128$ & $2,3,...,15$ & $100$ & $32$ & $200$ & $5/12$ & $2/3$ & $20$ & $0.47, 0.51, 0.55, 0.625$ & {\small F\ref{fig:1Dwave_red_errors_N=128_100epochs_PSD+V1-V4}, F\ref{fig:1Dwave_proj_N=128_100epochs_PSD+V1-V4}}\\
 \hline 
 Run 2                  & V$2$                & $128$ & $2,3,...,15$ & $100$ & $32$ & $200$ & $5/12$ & $2/3$ & $20$ & $0.47, 0.51, 0.55, 0.625$ & F\ref{fig:1Dwave_red_errors_N=128_100epochs_PSD+V1-V4}, F\ref{fig:1Dwave_proj_N=128_100epochs_PSD+V1-V4}\\
 \hline 
 Run 3                  & V$3$                & $128$ & $2,3,...,15$ & $100$ & $32$ & $200$ & $5/12$ & $2/3$ & $20$ & $0.47, 0.51, 0.55, 0.625$ & F\ref{fig:1Dwave_red_errors_N=128_100epochs_PSD+V1-V4}, F\ref{fig:1Dwave_proj_N=128_100epochs_PSD+V1-V4}, F\ref{fig:1Dwave_epoch_losses_N=128_100epochs_V3_V5_V6}\\
 \hline 
 Run 4                  & V$4$                & $128$ & $2,3,...,15$ & $100$ & $32$ & $200$ & $5/12$ & $2/3$ & $20$ & $0.47, 0.51, 0.55, 0.625$ & F\ref{fig:1Dwave_red_errors_N=128_100epochs_PSD+V1-V4}, F\ref{fig:1Dwave_proj_N=128_100epochs_PSD+V1-V4}\\
 \hline 
 Run 5                  & V$5$                & $128$ & $2,3,...,15$ & $100$ & $32$ & $200$ & $5/12$ & $2/3$ & $20$ & $0.47, 0.51, 0.55, 0.625$ & F\ref{fig:1Dwave_epoch_losses_N=128_100epochs_V3_V5_V6} \\
 \hline 
 Run 6                  & V$6$                & $128$ & $2,3,...,15$ & $100$ & $32$ & $200$ & $5/12$ & $2/3$ & $20$ & $0.47, 0.51, 0.55, 0.625$ & F\ref{fig:1Dwave_epoch_losses_N=128_100epochs_V3_V5_V6}\\
 \hline 
 Run 7                  & V$3$                & $128$ & $2,3,...,15$ & $50$ & $32$ & $200$ & $5/12$ & $2/3$ & $20$ & $0.47, 0.51, 0.55, 0.625$ & F\ref{fig:1Dwave_epoch_losses_N=128_50epochs_V3_V6-V10}, F\ref{fig:1Dwave_red_errors_N=128_50epochs_PSD+V3_V6_V10}\\
 \hline 
 Run 8                  & V$6$                & $128$ & $2,3,...,15$ & $50$ & $32$ & $200$ & $5/12$ & $2/3$ & $20$ & $0.47, 0.51, 0.55, 0.625$ & F\ref{fig:1Dwave_epoch_losses_N=128_50epochs_V3_V6-V10}, F\ref{fig:1Dwave_red_errors_N=128_50epochs_PSD+V3_V6_V10}\\
 \hline 
 Run 9                  & V$7$                & $128$ & $2,3,...,15$ & $50$ & $32$ & $200$ & $5/12$ & $2/3$ & $20$ & $0.47, 0.51, 0.55, 0.625$ & F\ref{fig:1Dwave_epoch_losses_N=128_50epochs_V3_V6-V10}, F\ref{fig:1Dwave_red_errors_N=128_50epochs_PSD+V3_V6_V10}\\
 \hline 
 Run 10                  & V$8$               & $128$ & $2,3,...,15$ & $50$ & $32$ & $200$ & $5/12$ & $2/3$ & $20$ & $0.47, 0.51, 0.55, 0.625$ & F\ref{fig:1Dwave_epoch_losses_N=128_50epochs_V3_V6-V10}, F\ref{fig:1Dwave_red_errors_N=128_50epochs_PSD+V3_V6_V10}\\
 \hline 
 Run 11                  & V$9$               & $128$ & $2,3,...,15$ & $50$ & $32$ & $200$ & $5/12$ & $2/3$ & $20$ & $0.47, 0.51, 0.55, 0.625$ & F\ref{fig:1Dwave_epoch_losses_N=128_50epochs_V3_V6-V10}, F\ref{fig:1Dwave_red_errors_N=128_50epochs_PSD+V3_V6_V10}\\
 \hline 
 Run 12                  & V$10$              & $128$ & $2,3,...,15$ & $50$ & $32$ & $200$ & $5/12$ & $2/3$ & $20$ & $0.47, 0.51, 0.55, 0.625$ & F\ref{fig:1Dwave_epoch_losses_N=128_50epochs_V3_V6-V10}, F\ref{fig:1Dwave_red_errors_N=128_50epochs_PSD+V3_V6_V10}\\
 \hline
 Run 13                  & V$9$               & $256$ & $10,...,20$ & $50$ & $32$ & $200$ & $5/12$ & $2/3$ & $20$ & $0.47, 0.51, 0.55, 0.625$ & F\ref{fig:1Dwave_red_errors_N=256_50epochs_V9vsPSD}, F\ref{reconstruction comparison}, T\ref{table:integration_time}\\
\end{tabular}
}
\caption{$1$D linear wave equation test runs.}
\label{sec:application_env:table_name:waveequationtestruns}
\end{table}
First, we aim to reproduce the results published in \cite{main:brantner2023}.
The system parameters are
\begin{itemize}
\item $N=128$,
\item $n=2,3,...,15$ and
\item $\mu_{\tt testing}=0.47, 0.51, 0.55, 0.625$.
\end{itemize}
We use 20 different $\mu$ parameters equally spaced over the interval $\P = [\tfrac{5}{12}, \tfrac{2}{3}]$ for the training set and
chose 200 time steps for the discretization of the time interval.
The learning is done for 100 epochs using a~constant batch size of $32$.
The first tested variant is V$1$ which is the exact training configuration used in \cite{main:brantner2023}.
Next, we learned variants V$2$-V$4$. These variants all use the homogeneous Adam optimizer implemented in GML.
V$2$ implements a~proper epoch-wise learning routine, V$3$ uses an~epoch-wise routine with normalized data and V$4$ performs epoch-wise and normalized learning
using an~alternative loss function.
All these variants are learned for $100$ epochs.
We computed both the reduction errors (see Figure~\ref{fig:1Dwave_red_errors_N=128_100epochs_PSD+V1-V4}) and 
the projection errors (see Figure~\ref{fig:1Dwave_proj_N=128_100epochs_PSD+V1-V4}) for these runs and compared them to the PSD method errors.
The projection errors are quite similar for all runs except for V$4$, where we get a~slightly higher error for small dimensions $\leq 10$ for all test values $\mu$.
Furthermore, we can observe a~monotonically decreasing error for all variants V$1$-V$4$ in all $\mu$ values.
The reduction errors are more interesting. Here, we see the same pattern for all variants. For small reduced dimensions the PSD error values are constantly smaller
than the learned autoencoder errors, whereas for reduced dimensions $\geq 10$ the autoencoder errors are almost constantly smaller than the PSD errors. The only exception is
$\mu=0.625$ where all variants V$1$-V$4$ perform rather poorly.
Comparing the different variants, we can see that the normalized data setup (with the ROM using a~reference state) works best (i.e. variants V$3$ and V$4$).
The chosen loss function and the epoch-wise learning setup do not seem to have much impact on the quality of the results.
Next, we tested the StiefelAdam optimizer with the same system parameters as before and compared it to the best homogeneous Adam variant V$3$.
Since the use of an~alternative loss did not yield any improvements, we decided to use the StiefelAdam optimizer with the GML loss.
We started with variant V$5$ using the canonical metric Riemannian gradient and the submanifold vector transport and no decay.
Again, we learned for $100$ epochs.
Here, we observe something very interesting. Looking at the average loss function value of each epoch (see Figure~\ref{fig:1Dwave_epoch_losses_N=128_100epochs_V3_V5_V6}), 
we see at first a~massive loss decrease that drops from an~initial value around $1$ to a~value around $0.12$ within only four epochs. 
This initial decrease is much faster than what we observe for V$3$, where we have an~average epoch loss around $0.62$ after four epochs.
Astonishingly, after the third epoch, the epoch loss of V$5$ suddenly increases again and explodes within a~few epochs. This behavior implies that the StiefelAdam optimizer makes good learning progress in the beginning
but then gets worse before it goes completely off course. This suggests the assumption that, as we approach a~(local) minimum of the error function,
we need to reduce the step size of an~update vector in order to improve. Otherwise, we will just jump over the minimum and drift away.
With this in mind, the logical next step is to add a~decay to the learning rate $\eta$ in each step. As described in Section~\ref{sec:application_subsubsec:networkconfiguration}, we added the decay $\eta \leftarrow \eta \cdot
0.9995$ in each iteration. The results of this modification were applied in V$6$. As seen in Figure~\ref{fig:1Dwave_epoch_losses_N=128_100epochs_V3_V5_V6}, this prevents the epoch loss from exploding,
and we get a~monotonous loss decrease.
Nevertheless, we have a~strong bend after four epochs and the loss decrease slows down considerably, so that we are approaching the loss of the homogeneous space optimizer.
Further research could be helpful here to investigate the reason for this behavior and thus possibly further improve the performance of the StiefelAdam optimizer.
In a~next step, we tested all combinations of the Riemannian metrics and vector transports using the StiefelAdam optimizer with decay (variants V$6$-V$9$) for the same system parameters as above
with the only difference that we tested only the reduced dimensions $n=6,7,...,15$.
We also tested the canonical metric and submanifold vector transport setup with the alternative loss in V$10$.
Since we observed a~very fast loss decrease at the beginning for V$6$, we now learned for only $50$ epochs instead of $100$ epochs.
We displayed the reduction error results of these runs in Figure~\ref{fig:1Dwave_red_errors_N=128_50epochs_PSD+V3_V6_V10} and 
the average epoch losses in Figure~\ref{fig:1Dwave_epoch_losses_N=128_50epochs_V3_V6-V10}.
When comparing to the PSD reduction errors, we can observe that the autoencoders outperform the PSD method for all $\mu$-values except for $\mu=0.625$.
Here, we see that the StiefelAdam variants can keep up with the PSD results for reduced dimensions $\geq 11$ whereas the error of variant V$3$ 
explodes for $n=15$.
Comparing the different autoencoder variants, we see that the different StiefelAdam configurations (V$6$-V$9$) work better than the HomogeneousAdam variant V$3$ on average.
This is in line with the faster loss decrease for these variants (see Figure~\ref{fig:1Dwave_epoch_losses_N=128_50epochs_V3_V6-V10}).
The four different configurations of the Riemannian metrics and vector transports are all pretty similar. We can observe a~slightly better performance for V$8$ (Euclidean metric + submanifold vector transport)
and V$9$ (Euclidean metric + differential vector transport).

At the end, we performed one more training and testing run for V$9$ in the doubled full dimension $N=256$ and the reduced dimensions $n=10,...,20$. 
We used the same $\mu$ values as before for the training data and the testing. 
In Figure~\ref{fig:1Dwave_red_errors_N=256_50epochs_V9vsPSD}, we compared the obtained reduction errors to the PSD reduction errors.
We make two interesting observations here.
First, for the reduced dimensions $n=10,...,16$, the autoencoder errors can keep up with the PSD errors. This changes for dimensions $n=17,...,20$. Here, we see
a massive error decrease for the PSD method, whereas the autoencoder errors increase. We display a~3D-reconstruction comparison for $n=17$ in Figure~\ref{reconstruction comparison}.
Second, we see that the autoencoder errors get bigger the larger $\mu$ gets, whereas this is not the case for the PSD errors. Especially for large reduced dimensions between $18$ and $20$, we see that
the reduction errors of the autoencoder explode in some cases. 
This is different from what we observed for $N=128$, where V$9$ could keep up and in most cases outperformed the PSD method for all $\mu$ values tested 
and reduced dimensions $n\geq 10$. 
Since the use of a~reduction method is usually applied to very high FOM dimensions $N$, it would be interesting to conduct more 
experiments for high FOM dimensions to see why the PSD method delivers so much better results than V$9$ here.
Another interesting observation is the large integration time for the ROM when using the nonlinear decoder function 
of the autoencoder compared to the PSD method (see Table~\ref{table:integration_time}). 
There may be several reasons for this. On the one hand, of course, this may be due to the fact that the learned autoencoder system already occupies 
many system resources (memory, etc.) and therefore fewer resources remain for integration, which can lead to a~loss of speed.
On the other hand, this may be due to the fact that we have to calculate the derivative of $d$ in each integration step of the ROM. 
For a~linear function (PSD method) this is an~easy task, but for a~nonlinear decoder $d$ (like the autoencoder functions) this can be quite expensive.
\begin{table}[ht!]
\resizebox{\textwidth}{!}{%
\centering
\begin{tabular}{llllllllllllllllllll}
  \hline
                    & \textbf{$n=10$} & \textbf{$n=11$} & \textbf{$n=12$} & \textbf{$n=13$} & \textbf{$n=14$} & \textbf{$n=15$}  & \textbf{$n=16$} & \textbf{$n=17$} & \textbf{$n=18$} & \textbf{$n=19$} & \textbf{$n=20$}  \\
  \hline  
  V$9$              &  99             & 121             & 167             & 221             & 242             & 247              & 290             & 384             & 861             & 777             & 1208 \\
  PSD               &  8              & 10              & 14              & 19              & 25              & 28               & 38              & 41              & 39              & 47              & 59   \\
  \hline
\end{tabular}
}
\caption{Integration time needed to solve the ROM in different reduced dimensions using the decoder from variant V$9$ and the PSD method for $\mu=0.625$ and $N=256$ in seconds.}
\label{table:integration_time}
\end{table}

\subsection{1D sine-Gordon equation} \label{sec:numericalexperiments_subsec:sinegordon}

The \textit{sine-Gordon equation} is a~nonlinear partial differential equation (PDE) and can be viewed as a~nonlinear extension of the 
\textit{linear wave equation} with unity wave speed $\mu = 1$.

We consider a~1D sine-Gordon equation  
\begin{align} \label{sec:numericalexperiments_env:align_name:pde1Dsinegordon}
  \partial_{tt}^2u(t,\xi) &= \partial_{\xi \xi}^2u(t,\xi) - \sin \left(u(t,\xi)\right) & \quad\quad &\text{in } \I \times \Omega,  \\
  u(0,\xi) &= u_0(\xi) & \quad\quad &\text{on } \Omega, \nonumber \\
  u_t(0,\xi) &= u_1(\xi) & \quad\quad &\text{on } \Omega, \nonumber \\
  u(t,a) &= \varphi(t) & \quad\quad &\text{on } \I, \nonumber \\
  u(t,b) &= \psi(t) & \quad\quad &\text{on } \I \nonumber  
\end{align}
with initial conditions determined by $u_0(\xi), u_1(\xi)$ and boundary conditions determined by $\varphi(t),\psi(t)$.
Unlike the linear wave equation \eqref{sec:numericalexperiments_env:align_name:pde1Dwaveequation}, the wave speed 
factor $\mu^2$ disappears, and instead we subtract the nonlinearity $\sin(u(t,\xi))$.
We define 
\begin{itemize}
  \item $q(t,\xi) := u(t,\xi)$ and 
  \item $p(t,\xi) := \partial_t q(t,\xi) = \partial_t u(t,\xi)$
\end{itemize}
(and thus $\partial_tp(t,\xi) = \partial_{\xi \xi}^2u(t,\xi) - \sin \left(u(t,\xi)\right) = \partial_{\xi \xi}^2q(t,\xi) - \sin \left(q(t,\xi)\right)$), 
and obtain the \textit{continuous Hamiltonian} 
\begin{align} \label{sec:numericalexperiments_env:align_name:continuoushamiltonian}
  \H_{cont}(q,p) &= \int_{\Omega} \frac{1}{2} \left(\partial_{\xi} q(t,\xi)\right)^2 + \frac{1}{2} p(t,\xi)^2 + (1 - \cos(q(t,\xi))) d\xi. 
\end{align}

\subsubsection{Spatial discretization}

Given an~$N\in \N$, we discretize the spatial domain $\Omega=[a,b]$ into $(N+2)$ points $\xi_i := ih_{N} + a$ for $i = 0,...,N+1$ and $h_{N} := \tfrac{b-a}{N+1}$. 
For the discretization of $q(t,\xi)$ and $p(t,\xi)$ we define:
\begin{itemize}
  \item $q_i(t) := q(t,\xi_i) \quad i=0,...,N+1$,
  \item $p_i(t) := p(t,\xi_i) \quad i=0,...,N+1$.
\end{itemize} 
We want to obtain the discretized problem as an~ODE of the form
\begin{align} \label{sec:numericalexperiments_env:align_name:odegeneraldiscretizedproblem}
  \begin{cases}
    \dot{q} = p \\
    \dot{p} = \L q - f(q) \\
    q(0) = q_0, p(0) = p_0,
  \end{cases}
\end{align}
where we set 
$q := [q_1,...,q_{N}]^T$ and $p := [p_1,...,p_{N}]^T$. 
Here, $\L$ denotes a~linear operator and $f$ is the nonlinearity.
Both $\L$ and $f$ rely on the method used to discretize the continuous Hamiltonian \eqref{sec:numericalexperiments_env:align_name:continuoushamiltonian} 
and the initial/boundary conditions used. We use the \textit{trapezoidal rule}: 
\begin{align} \label{sec:numericalexperiments_env:align_name:calctrapezoidalrule}
  \H_{cont}(q,p) 
  &= \int_{\Omega} \frac{1}{2} \left(\partial_{\xi} q(t,\xi)\right)^2 + \frac{1}{2} p(t,\xi)^2 + (1 - \cos(q(t,\xi))) d\xi \\
  &\approx \frac{h_N}{2} \sum_{i=0}^N \bigg(\frac{1}{2} \left(\partial_{\xi} q_i(t)\right)^2 + \frac{1}{2} p_i(t)^2 + (1 - \cos(q_i(t)))  \nonumber \\
  &\phantom{\frac{h_N}{4} \sum_{i=0}^N \bigg(} +\frac{1}{2} \left(\partial_{\xi} q_{i+1}(t)\right)^2 + \frac{1}{2} p_{i+1}(t)^2 + (1 - \cos(q_{i+1}(t))) \bigg) \nonumber \\
  &=\frac{h_N}{4} \left( \left(\partial_{\xi} q_{0}(t)\right)^2 + \left(\partial_{\xi} q_{N+1}(t)\right)^2 \right) + \frac{h_N}{4} \sum_{i=1}^N 2\left(\partial_{\xi} q_{i}(t)\right)^2 \nonumber \\
  &+\frac{h_N}{4} \left( p_{0}(t)^2 + p_{N+1}(t)^2 \right) + \frac{h_N}{2} \sum_{i=1}^N p_{i}(t)^2 \nonumber \\
  &+\frac{h_N}{2} \left( (1 - \cos(q_{0}(t))) + (1 - \cos(q_{N+1}(t))) \right) + h_N \sum_{i=1}^N (1 - \cos(q_{i}(t))). \nonumber 
\end{align}
Before we continue, we first need to find an~approximation for $\left(\partial_{\xi} q_i(t)\right)^2$. We use a~combination of \textit{forward and 
backward difference methods} to obtain
\begin{itemize}
  \item $\left(\partial_{\xi} q_0(t)\right)^2 \approx \left(\frac{q_1(t)- q_0(t)}{h_N}\right)^2$,
  \item $\left(\partial_{\xi} q_{N+1}(t)\right)^2 \approx \left(\frac{q_{N+1}(t)- q_{N}(t)}{h_N}\right)^2$,
  \item $2\left(\partial_{\xi} q_i(t)\right)^2 \approx \left(\frac{q_{i}(t)- q_{i-1}(t)}{h_N}\right)^2 + \left(\frac{q_{i+1}(t)- q_i(t)}{h_N}\right)^2, \quad i=1,...,N$.
\end{itemize} 
Now, we apply this to \eqref{sec:numericalexperiments_env:align_name:calctrapezoidalrule} and get
\begin{align*}   
  \H_{cont}(q,p)
  &\approx\frac{h_N}{4} \sum_{i=0}^N 2\left( \frac{q_{i+1}(t) - q_i(t)}{h_N} \right)^2  \\
  &+\frac{h_N}{4} \left( p_{0}(t)^2 + p_{N+1}(t)^2 \right) + \frac{h_N}{2} \sum_{i=1}^N p_{i}(t)^2 \nonumber \\
  &+\frac{h_N}{2} \left( (1 - \cos(q_{0}(t))) + (1 - \cos(q_{N+1}(t))) \right) + h_N \sum_{i=1}^N (1 - \cos(q_{i}(t))) \nonumber \\
  &=\frac{h_N}{2} \left( \frac{(q_1(t) - \varphi(t))^2}{h_N^2} + \frac{(\psi(t) - q_N(t))^2}{h_N^2} + \sum_{i=1}^{N-1} \frac{(q_{i+1}(t) - q_i(t))^2}{h_N^2} \right) \nonumber \\
  &+\frac{h_N}{4} \left( \dot{\varphi}(t)^2 + \dot{\psi}(t)^2 \right) + \frac{h_N}{2} \sum_{i=1}^N p_{i}(t)^2 \nonumber \\
  &+\frac{h_N}{2} \left( (1 - \cos(\varphi(t))) + (1 - \cos(\psi(t))) \right) + h_N \sum_{i=1}^N (1 - \cos(q_{i}(t))) \nonumber \\
  &=\frac{h_N}{2} \left( -q^T \L q - \frac{2q_1(t)\varphi(t)}{h_N^2} + \frac{\varphi(t)^2}{h_N^2} + \frac{\psi(t)^2}{h_N^2} - \frac{2q_N(t)\psi(t)}{h_N^2} \right) \nonumber \\
  &+\frac{h_N}{4} \left( \dot{\varphi}(t)^2 + \dot{\psi}(t)^2 \right) + \frac{h_N}{2} p^Tp \nonumber \\
  &+\frac{h_N}{2} \left( (1 - \cos(\varphi(t))) + (1 - \cos(\psi(t))) \right) + h_N \sum_{i=1}^N (1 - \cos(q_{i}(t))) \nonumber \\
  &=-\frac{h_N}{2} q^T \L q + \frac{h_N}{2} p^Tp \nonumber \\
  &+\frac{h_N}{2} \left( -\frac{2q_1(t)\varphi(t)}{h_N^2} + \frac{\varphi(t)^2}{h_N^2} + \frac{\psi(t)^2}{h_N^2} - \frac{2q_N(t)\psi(t)}{h_N^2} \right) + \frac{h_N}{4} \left( \dot{\varphi}(t)^2 + \dot{\psi}(t)^2 \right) \nonumber \\
  &+\frac{h_N}{2} \left( (1 - \cos(\varphi(t))) + (1 - \cos(\psi(t))) \right) + h_N \sum_{i=1}^N (1 - \cos(q_{i}(t))) \nonumber \\
  &=: \H_{h_N}(q,p), \nonumber
\end{align*}
where we set  
\begin{align*}
  \L &:=
  \frac{1}{h_N^2}
\begin{bmatrix}
  -2 & 1 & & &  \\
  1 & -2 & 1 & &  \\
   & \ddots & \ddots & \ddots & \\
   &  &  &  & 1 \\
   & & & 1 & -2 \\
\end{bmatrix} \in \R^{N}.
\end{align*}
Calculating the gradient of $\H_{h_N}(q,p)$ gives us
\begin{align*} 
  \nabla\H_{h_N}(q,p) &= 
  \begin{bmatrix}
    -h_N \L & 0 \\
    0 & h_NI_N 
  \end{bmatrix}
  \begin{bmatrix}
    q \\ p
  \end{bmatrix}
  + 
  \begin{bmatrix}
    h_Nf(q) \\ 0
  \end{bmatrix},
\end{align*}  
where
\begin{align*} 
  f(q) &= 
  \begin{bmatrix}
    \begin{aligned}
     &\sin(q_1) - \frac{\varphi(t)}{h_N^2} \\
     &\sin(q_2) \\ 
     &\phantom{..} \vdots \\
     &\sin(q_N) - \frac{\psi(t)}{h_N^2}
     &\end{aligned}
  \end{bmatrix}.
\end{align*}

We now obtain a~\textit{Hamiltonian system} of the form \eqref{sec:numericalexperiments_env:align_name:odegeneraldiscretizedproblem}. We have
\begin{align} 
  \begin{bmatrix}
    \dot{q} \\ \dot{p}
  \end{bmatrix}
  &=
  \begin{bmatrix}
    p \\ \L q - f(q) 
  \end{bmatrix} \label{sec:numericalexperiments_env:align_name:almosthamiltonian} \\
  &= J_{2N} 
  \left( 
  \begin{bmatrix}
    -\L & 0 \\
    0 & I_N \\
  \end{bmatrix}
  \begin{bmatrix}
    q \\ p
  \end{bmatrix}
  +
  \begin{bmatrix}
    f(q) \\ 0
  \end{bmatrix}
  \right) \nonumber \\
  &=  J_{2N} \frac{1}{h_N}\nabla \H_{h_N}(q,p). \nonumber
\end{align}
This system almost looks like a~Hamiltonian system for the discretized sine-Gordon equation. The only remaining problem is the scaling 
factor $\tfrac{1}{h_N}$. Similar to the 1D wave equation, we omitted the state-space transformation and implemented the not transformed, 
scaled problem \eqref{sec:numericalexperiments_env:align_name:almosthamiltonian}.

Note that the dimension of the Hamiltonian system \eqref{sec:numericalexperiments_env:align_name:almosthamiltonian} is $2N$, 
unlike for the 1D linear wave equation, where we had the full dimension $2(N+2)$. 
Now we have \enquote{cut off} the edges $q_0, q_{N+1}, p_0, p_{N+1}$ of the vector $[q^T,p^T]^T$.
This is due to the different discretization methods we used and does not cause any problems 
since the boundary values are determined by the initial and boundary conditions.

\subsubsection{Initial and boundary conditions}

As described in \eqref{sec:numericalexperiments_env:align_name:pde1Dsinegordon}, we have to make choices for the initial conditions $u_0(\xi), u_1(\xi)$ (IC) 
and the boundary conditions $\varphi(t),\psi(t)$ (BC).
We decided to implement two different condition pairs: The so-called \textit{single soliton} conditions and the \textit{soliton-soliton doublets} conditions.
Unlike the 1D wave equation, we have analytical solutions for these condition pairs. 
Although this makes these examples kind of useless for a~real world model reduction scenario, this is exactly what makes it a~good choice for testing a~new network configuration,
since we have the exact FOM solutions to compute the error and do not have to rely on inexact numerical solutions.
The analytical solutions presented in this section are taken from \cite{sGAnalytical2:mittal2014} and \cite{sGanalytical1:uddin2012}.

\subsubsubsection{Single soliton} \label{subsec:theory_singlesoliton}
Let $\I:=[0,T]$ and $\Omega:=[a,b]$. 
The single soliton initial conditions are given by
  \begin{align*}
    u_0(\xi) &= u(0,\xi) := 4\arctan \left( \exp \left( \frac{\xi}{\sqrt{1- \nu^2}} \right) \right),  \\
    u_1(\xi) &= \partial_t u(0,\xi) := \frac{-4\nu \exp \left( \frac{\xi}{\sqrt{1- \nu^2}} \right)}{\sqrt{1- \nu^2} \left( 1 + \exp \left( \frac{2\xi}{\sqrt{1- \nu^2}} \right) \right)}.  
  \end{align*}
The boundary conditions are
  \begin{align*}
    \varphi(t) &= u(t,a) := 4\arctan \left( \exp \left( \frac{a - \nu t}{\sqrt{1- \nu^2}} \right) \right), \\ 
    \psi(t) &= u(t,b) := 4\arctan \left( \exp \left( \frac{b - \nu t}{\sqrt{1- \nu^2}} \right) \right). 
  \end{align*}
These initial and boundary conditions lead to the exact solution
\begin{align}
  u(t,\xi) = 4\arctan \left( \exp \left( \frac{\xi - \nu t}{\sqrt{1- \nu^2}} \right) \right). \label{singlesolitonsolution_exactsol}
\end{align}
In contrast to the 1D linear wave equation, we have now denoted the variable system parameter as $\nu$ instead of $\mu$.
Defining $x_\nu(t) := [q(t)^T,p(t)^T]^T \in \R^{2N}$, we have to solve the $\nu$-dependent FOM
\begin{align}
  \dot{x}_{\nu} &= J_{2N} \frac{1}{h_{N}} \nabla_{x}\H_{h_N} (x_{\nu}; \nu) \label{ode:examplesinglesoliton} \\ 
  &=  
  \begin{bmatrix}
    0 && I_N \\
    \L && 0 \\
  \end{bmatrix}
  x_{\nu}
  - 
  \begin{bmatrix}
    \begin{aligned}
     & \hphantom{\sin(q_1) - \tfrac{1}{h_N^2} 444444} 0 \\
     & \hphantom{\sin(q_1) - \tfrac{1}{h_N^2} 444444} \vdots \\
     & \hphantom{\sin(q_1) - \tfrac{1}{h_N^2} 444444} 0 \\
     &\sin(q_1) - \tfrac{1}{h_N^2} 4\arctan \left( \exp \left( \tfrac{a - \nu t}{\sqrt{1- \nu^2}} \right) \right)\\
     &\sin(q_2) \\ 
     &\phantom{..} \vdots \\
     &\sin(q_N) - \tfrac{1}{h_N^2} 4\arctan \left( \exp \left( \tfrac{b - \nu t}{\sqrt{1- \nu^2}} \right) \right)
     &\end{aligned}
  \end{bmatrix}, \nonumber \\
  x_{\nu}(0) &=   
  \begin{bmatrix}
    \left(4\arctan \left( \exp \left( \tfrac{\xi_i}{\sqrt{1- \nu^2}} \right) \right) \right)_{i=1}^N \\[1.0em] \left(\frac{-4\nu \exp \left( \tfrac{\xi_i}{\sqrt{1- \nu^2}} \right)}{\sqrt{1- \nu^2} \left( 1 + \exp \left( \tfrac{2\xi_i}{\sqrt{1- \nu^2}} \right) \right)} \right)_{i=1}^N \\
  \end{bmatrix} \in \R^{2N}. \nonumber 
\end{align} 

\clearpage
\subsubsubsection{Soliton-soliton doublets} \label{subsec:theory_solitonsolitondoublets}
Denote the \textit{hyperbolic secant} function by 
$$\sech(x) := \frac{2}{e^x + e^{-x}}.$$
The soliton-soliton doublets initial conditions are given by
  \begin{align*}
    u_0(\xi) &= u(0,\xi) := 4\arctan \left( \nu \sinh \left( \frac{\xi }{\sqrt{1 - \nu^2}} \right) \right),  \\
    u_1(\xi) &= \partial_t u(0,\xi) := 0. 
  \end{align*}
The boundary conditions are
  \begin{align*}
    \varphi(t) &= u(t,a) := 4\arctan \left( \nu \sech \left( \frac{\nu t}{\sqrt{1 - \nu^2}} \right) \sinh \left( \frac{a}{\sqrt{1 - \nu^2}} \right) \right), \\
    \psi(t) &= u(t,b) := 4\arctan \left( \nu \sech \left( \frac{\nu t}{\sqrt{1 - \nu^2}} \right) \sinh \left( \frac{b}{\sqrt{1 - \nu^2}} \right) \right). 
  \end{align*}
These initial and boundary conditions lead to the exact solution
\begin{align}
  u(t,\xi) = 4\arctan \left( \nu \sech \left( \frac{\nu t}{\sqrt{1 - \nu^2}} \right) \sinh \left( \frac{\xi}{\sqrt{1 - \nu^2}} \right) \right). \label{solitonsolitiondoublets_exactsol}
\end{align}
Again, we used $\nu$ as the variable to denote the variable system parameter. \\
Defining ${x_\nu(t) := [q(t)^T,p(t)^T]^T \in \R^{2N}}$, we have to solve the $\nu$-dependent FOM
\begin{align}
  \dot{x}_{\nu} &= J_{2N} \frac{1}{h_{N}} \nabla_{x}\H_{h_N} (x_{\nu}; \nu)  \label{ode:examplesolitionsolitondoublets} \\ 
  &=  
  \begin{bmatrix}
    0 && I_N \\
    \L && 0 \\
  \end{bmatrix}
  x_{\nu}
  - 
  \begin{bmatrix}
    \begin{aligned}
     & \hphantom{\sin(q_1) - \tfrac{1}{h_N^2} 44444444444} 0 \\
     & \hphantom{\sin(q_1) - \tfrac{1}{h_N^2} 44444444444} \vdots \\
     & \hphantom{\sin(q_1) - \tfrac{1}{h_N^2} 44444444444} 0 \\
     &\sin(q_1) - \tfrac{1}{h_N^2} 4\arctan \left( \nu \sech \left( \tfrac{\nu t}{\sqrt{1 - \nu^2}} \right) \sinh \left( \tfrac{a}{\sqrt{1 - \nu^2}} \right) \right)\\
     &\sin(q_2) \\ 
     &\phantom{..} \vdots \\
     &\sin(q_N) - \tfrac{1}{h_N^2} 4\arctan \left( \nu \sech \left( \tfrac{\nu t}{\sqrt{1 - \nu^2}} \right) \sinh \left( \tfrac{b}{\sqrt{1 - \nu^2}} \right) \right)
     &\end{aligned}
  \end{bmatrix}, \nonumber \\
  x_{\nu}(0) &=   
  \begin{bmatrix}
    \left(4\arctan \left( \nu \sinh \left( \tfrac{\xi_i}{\sqrt{1 - \nu^2}} \right) \right) \right)_{i=1}^N \\ \left( 0 \right)_{i=1}^N \\
  \end{bmatrix} \in \R^{2N}. \nonumber 
\end{align}

\subsubsection{Implementation}

We implemented all single soliton and soliton-soliton doublets specific parts in the folder \\
\stexttt{sineGordon/examples}.
The Hamiltonian is implemented in the folder \stexttt{sineGordon/general}.
The training data matrix with the analytical solution snapshots is generated in \\
{\small\texttt{sineGordon/main/fn\_generate\_analytic\_data.jl}}. 
The data matrix is constructed in the same way as for the $1$D linear wave equation (see \eqref{fig:schematic_data_matrix}), except that we now have the 
variable system parameter $\nu$ instead of $\mu$.
The normalized training data is produced in the file {\small\texttt{sineGordon/main/fn\_generate\_data\_normalized.jl}}.
The errors of the comparison method are computed in \stexttt{sineGordon/main/fn\_generate\_psd.jl}. 
The actual training is performed in \stexttt{fn\_training\_v020\_sineGordon.jl} with the GML built-in homogeneous space optimizer and in 
\stexttt{fn\_training\_v020\_alternative\_optimizer\_sineGordon.jl} with our alternative optimizer approach.
The necessary system parameters and the choices of the training variants (see Table~\ref{sec:application_env:table_name:completelearningsetups})
need to be specified in 
\begin{itemize}
  \item \scripttexttt{sineGordon/general/fn\_global\_consts.jl} and
  \item \scripttexttt{alternative\_optimizer/generalfn\_global\_consts\_alternative\_optimizer\_specific.jl}.
\end{itemize}
The variant specific options are the same as for the $1$D linear wave equation (displayed in Table~\ref{sec:application_env:table_name:globalconstsrepro}).
The sine-Gordon equation specific parameters are shown in Table~\ref{sec:application_env:table_name:globalconstssineGordon1Dspecific}.
Unlike what we did for the wave equation, we now have to specify the time interval $\I$ and the spatial domain $\Omega$ by setting the 
specific parameters.
Again, we use the datatype \stexttt{Float64} for the training snapshots as well as for the actual learning and the testing.
A~description of the chronological steps to be taken to perform a~complete training routine can be found in {\small \stexttt{fn\_global\_consts.jl}}.
\begin{table}[ht!]
\resizebox{\textwidth}{!}{%
\centering
\begin{tabular}{c|p{10.5cm}} 
 \textbf{Parameter}             & \textbf{Explanation}          \\
 \hline 
 {\small\texttt{N\_global}}                                             & Number of spatial discretization points $N$; the full dimension of the FOM is $2N$.                                                                                          \\
 \hline 
 {\small\texttt{range\_n\_global}}                                      & Range of reduced dimensions that are supposed to be tested.                                                                                                                       \\
 \hline 
 {\small\texttt{n\_epochs\_global}}                                     & Number of epochs to be learned.                                                                                                                                                  \\
 \hline 
 {\small\texttt{batch\_size\_global}}                                   & The size of a~single batch to be drawn.                                                                                                                                          \\
 \hline 
 {\small\texttt{start\_time\_global}}                                   & The left boundary of the time interval $\I=[\text{t}_0,\text{t}_1]$.                                                                                                                                          \\
 \hline 
 {\small\texttt{end\_time\_global}}                                     & The right boundary of the time interval $\I=[\text{t}_0,\text{t}_1]$.                                                                                                                                          \\
 \hline 
 {\small\texttt{a\_global}}                                             & The left boundary of the spatial domain $\Omega=[a,b]$.                                                                                                                                          \\
 \hline 
 {\small\texttt{b\_global}}                                             & The right boundary of the spatial domain $\Omega=[a,b]$.                                                                                                                                          \\
 \hline 
 {\small\texttt{time\_steps\_global}}                                   & The number of discrete time steps equally distributed over the time interval $\I$ (short form \stexttt{t\_steps}).                                                                                              \\
 \hline 
 {\small\texttt{$\nu$\_global\_range\_single\_soliton}}                 & The variable system parameters $\nu$ to build the training data for the single soliton conditions. Unlike for the wave equation, we now specify these values directly and not by the number of values (short form \stexttt{$\nu$\_range}).    \\
 \hline 
 {\small\texttt{$\nu$\_global\_range\_soliton\_soliton\_doublets}}      & The variable system parameters $\nu$ to build the training data for the soliton-soliton doublets conditions (short form \stexttt{$\nu$\_range}).                                                                    \\
 \hline 
 {\small\texttt{$\nu$\_testing\_range\_global}}                         & The range of $\nu$ values to test the quality of the learned autoencoder.                                                                                                         \\
 \hline 
 {\small\texttt{example\_in\_use\_global}}                              & Specifies the example to be specified as a~string. For single soliton this is the string \stexttt{single\_soliton}, for soliton-soliton doublets it is \stexttt{soliton\_soliton\_doublets}. \\
\end{tabular}
}
\caption{Summary of all $1$D sine-Gordon equation system parameters needed to be specified in the file {\small\texttt{fn\_global\_consts.jl}}.}
\label{sec:application_env:table_name:globalconstssineGordon1Dspecific}
\end{table}

\subsubsection{Numerical results: single soliton}
We summarize all test runs and the obtained results in Table~\ref{sec:application_env:table_name:singlesolitontestruns}. 
\begin{table}[ht!]
\resizebox{\textwidth}{!}{%
\centering
\begin{tabular}{c|c|c|c|c|c|c|c|c|c|c|c|c|c} 
 \textbf{Test run}      & \textbf{Variant}    & \multicolumn{11}{c|}{\textbf{Single soliton parameter values}}  & \textbf{Results}       \\
                        &                     & {\small\texttt{N}} & {\small\texttt{range\_n}} & {\small\texttt{epochs}}& {\small\texttt{batch}} & {\stexttt{t0}} & {\stexttt{t1}} & {\stexttt{a}} & {\stexttt{b}} & {\small\texttt{t\_steps}} & {\small\texttt{$\nu$\_range}} & {\small\texttt{$\nu$\_testing\_range}} & \\
 \hline 
 Run 1                  & V$2$                & $128$ & $3,4,...,15$ & $50$ & $32$ & $0.0$ & $4.0$ &-10.0 & 10.0 & $200$  & \stexttt{range}$(-0.98,-0.72,20)$ & $[-0.97,-0.92,-0.85,-0.73]$ & F\ref{fig:1DsG_red_errors_N=128_50epochs_singlesoliton}, F\ref{fig:1DsG_proj_errors_N=128_50epochs_singlesoliton}, F\ref{fig:1DsineGordon_epoch_losses_N=128_50epochs_singlesoliton}, F\ref{reconstruction_comparison_single_soliton} \\
 \hline 
 Run 2                  & V$3$                & $128$ & $3,4,...,15$ & $50$ & $32$ & $0.0$ & $4.0$ &-10.0 & 10.0 & $200$  & \stexttt{range}$(-0.98,-0.72,20)$ & $[-0.97,-0.92,-0.85,-0.73]$ &  F\ref{fig:1DsG_red_errors_N=128_50epochs_singlesoliton}, F\ref{fig:1DsG_proj_errors_N=128_50epochs_singlesoliton}, F\ref{fig:1DsineGordon_epoch_losses_N=128_50epochs_singlesoliton}, F\ref{reconstruction_comparison_single_soliton} \\
 \hline                                          
 Run 3                  & V$6$                & $128$ & $3,4,...,15$ & $50$ & $32$ & $0.0$ & $4.0$ &-10.0 & 10.0 & $200$  & \stexttt{range}$(-0.98,-0.72,20)$ & $[-0.97,-0.92,-0.85,-0.73]$ & F\ref{fig:1DsG_red_errors_N=128_50epochs_singlesoliton}, F\ref{fig:1DsG_proj_errors_N=128_50epochs_singlesoliton}, F\ref{fig:1DsineGordon_epoch_losses_N=128_50epochs_singlesoliton} \\
 \hline                                          
 Run 4                  & V$7$                & $128$ & $3,4,...,15$ & $50$ & $32$ & $0.0$ & $4.0$ &-10.0 & 10.0 & $200$  & \stexttt{range}$(-0.98,-0.72,20)$ & $[-0.97,-0.92,-0.85,-0.73]$ & F\ref{fig:1DsG_red_errors_N=128_50epochs_singlesoliton}, F\ref{fig:1DsG_proj_errors_N=128_50epochs_singlesoliton}, F\ref{fig:1DsineGordon_epoch_losses_N=128_50epochs_singlesoliton} \\
 \hline                                          
 Run 5                  & V$8$                & $128$ & $3,4,...,15$ & $50$ & $32$ & $0.0$ & $4.0$ &-10.0 & 10.0 & $200$  & \stexttt{range}$(-0.98,-0.72,20)$ & $[-0.97,-0.92,-0.85,-0.73]$ & F\ref{fig:1DsG_red_errors_N=128_50epochs_singlesoliton}, F\ref{fig:1DsG_proj_errors_N=128_50epochs_singlesoliton}, F\ref{fig:1DsineGordon_epoch_losses_N=128_50epochs_singlesoliton}, F\ref{reconstruction_comparison_single_soliton} \\
 \hline                                          
 Run 6                  & V$9$                & $128$ & $3,4,...,15$ & $50$ & $32$ & $0.0$ & $4.0$ &-10.0 & 10.0 & $200$  & \stexttt{range}$(-0.98,-0.72,20)$ & $[-0.97,-0.92,-0.85,-0.73]$ & F\ref{fig:1DsG_red_errors_N=128_50epochs_singlesoliton}, F\ref{fig:1DsG_proj_errors_N=128_50epochs_singlesoliton}, F\ref{fig:1DsineGordon_epoch_losses_N=128_50epochs_singlesoliton} \\
\end{tabular}
}
\caption{Single soliton test runs.}
\label{sec:application_env:table_name:singlesolitontestruns}
\end{table}
We learned the variants V$2$, V$3$ and V$6$-V$9$ for $50$ epochs and compared the results to the PSD method. 
The system parameters we used, are
\begin{itemize}
\item $N=128$,
\item $n=3,...,15$,
\item $\I=[0.0, 4.0]$,
\item $\Omega=[-10.0,10.0]$,
\item $\nu_{\tt testing}=-0.97, -0.92, -0.85, -0.73$.
\end{itemize}
We use $\P=[-0.98, -0.72]$ as our $\nu$-range. We took $20$ equally spaced $\nu$ values from this range for the training set.
Note that the possible $\nu$-range for the single soliton conditions is the interval $(-1,1)$.  
However, since the FOM solutions are symmetric about zero except for the sign and since the (discretized) FOM solution for $\nu=0$ is constant (in $t$), 
we decided to use the negative partial interval $[-0.98, -0.72] \subseteq (-1,1)$ near $-1$ for our tests. 
This is particularly interesting because the larger the norm of a~value $\nu$ is, the less \enquote{trivial}, i.e. less constant, the solution is.
The reduction errors are displayed in Figure~\ref{fig:1DsG_red_errors_N=128_50epochs_singlesoliton}, the projection errors in Figure~\ref{fig:1DsG_proj_errors_N=128_50epochs_singlesoliton}
and the average epoch loss in Figure~\ref{fig:1DsineGordon_epoch_losses_N=128_50epochs_singlesoliton}.
Compared to the 1D wave equation, we were able to determine some interesting differences in the behavior of the individual autoencoder variants.
The reduction errors of the autoencoder variants outperform the PSD errors for small reduced dimensions $n\leq 10$ for all $\nu$-values, whereas 
for $n \geq 11$ the PSD method shows a~better error decrease and outperforms the autoencoder variants especially for small $\nu$ values.
The homogeneous Adam variants V$2$ and V$3$ show the best performance for very small dimensions $n$, but cannot improve a~lot for larger reduced 
dimensions $n$. We have displayed some reconstructions for $\nu=-0.73$ in Figure~\ref{reconstruction_comparison_single_soliton}. 
It shows the good performance of V$2$ and V$3$ in comparison to variant V$8$ and the PSD method in dimension $n=4$. 
It is also noticeable that the results of V$2$ and V$3$ become significantly worse the larger the norm of $\nu$ gets. 
This is particularly noticeable for $\nu=-0.97$, where we get significantly better results for the StiefelAdam variants V$6$-V$9$. 
This is also the only $\nu$ value where V$6$-V$9$ constantly give better results than the PSD comparison method. 
It suggests that the learned nonlinear autoencoder pair copes better with the increasing complexity of the system,
but has more problems to achieve as good results as the PSD method for rather constant solution reconstructions.
The projection error results show a~similar behavior. Here, we also have generally better results the lower the $\nu$-values are.
Again, we have the best performance for variants V$6$-V$9$, especially for $\nu=-0.97$. 
The average epoch loss shows similar results to the $1$D wave equation. We have a~very steep loss decrease at the beginning for all variants except 
V$3$. Interestingly, the unnormalized homogeneous Adam approach V$2$ shows a~significantly faster loss decrease than the normalized approach V$3$
and, for small reduced dimensions (like $n=5$), even a~faster loss decrease than the StiefelAdam variants V$6$-V$9$. 
Surprisingly, this faster loss decrease does not translate into better error values.

\subsubsection{Numerical results: soliton-soliton doublets}
We summarize all test runs and the obtained results in Table~\ref{sec:application_env:table_name:solitonsolitondoubletstestruns}. 
\begin{table}[ht!]
\resizebox{\textwidth}{!}{%
\centering
\begin{tabular}{c|c|c|c|c|c|c|c|c|c|c|c|c|c} 
 \textbf{Test run}      & \textbf{Variant}    & \multicolumn{11}{c|}{\textbf{Soliton-soliton doublets parameter values}}  & \textbf{Results}       \\
                        &                     & {\small\texttt{N}} & {\small\texttt{range\_n}} & {\small\texttt{epochs}}& {\small\texttt{batch}} & {\stexttt{t0}} & {\stexttt{t1}} & {\stexttt{a}} & {\stexttt{b}} & {\small\texttt{t\_steps}} & {\small\texttt{$\nu$\_range}} & {\small\texttt{$\nu$\_testing\_range}} & \\
 \hline 
 Run 1                  & V$2$                & $128$ & $3,4,...,15$ & $50$ & $32$ & $0.0$ & $4.0$ &-10.0 & 10.0 & $200$  & \stexttt{range}$(0.72,0.98,20)$ & $[0.73,0.85,0.92,0.97]$ & F\ref{fig:1DsG_red_errors_N=128_50epochs_solitonsolitondoublets}, F\ref{fig:1DsG_proj_errors_N=128_50epochs_solitonsolitondoublets}, F\ref{fig:1DsineGordon_epoch_losses_N=128_50epochs_solitonsolitiondoublets}, F\ref{reconstruction_comparison_solitonsolitiondoublets} \\
 \hline 
 Run 2                  & V$3$                & $128$ & $3,4,...,15$ & $50$ & $32$ & $0.0$ & $4.0$ &-10.0 & 10.0 & $200$  & \stexttt{range}$(0.72,0.98,20)$ & $[0.73,0.85,0.92,0.97]$ & F\ref{fig:1DsG_red_errors_N=128_50epochs_solitonsolitondoublets}, F\ref{fig:1DsG_proj_errors_N=128_50epochs_solitonsolitondoublets}, F\ref{fig:1DsineGordon_epoch_losses_N=128_50epochs_solitonsolitiondoublets}, F\ref{reconstruction_comparison_solitonsolitiondoublets} \\
 \hline 
 Run 3                  & V$6$                & $128$ & $3,4,...,15$ & $50$ & $32$ & $0.0$ & $4.0$ &-10.0 & 10.0 & $200$  & \stexttt{range}$(0.72,0.98,20)$ & $[0.73,0.85,0.92,0.97]$ & F\ref{fig:1DsG_red_errors_N=128_50epochs_solitonsolitondoublets}, F\ref{fig:1DsG_proj_errors_N=128_50epochs_solitonsolitondoublets}, F\ref{fig:1DsineGordon_epoch_losses_N=128_50epochs_solitonsolitiondoublets} \\
 \hline 
 Run 4                  & V$7$                & $128$ & $3,4,...,15$ & $50$ & $32$ & $0.0$ & $4.0$ &-10.0 & 10.0 & $200$  & \stexttt{range}$(0.72,0.98,20)$ & $[0.73,0.85,0.92,0.97]$ & F\ref{fig:1DsG_red_errors_N=128_50epochs_solitonsolitondoublets}, F\ref{fig:1DsG_proj_errors_N=128_50epochs_solitonsolitondoublets}, F\ref{fig:1DsineGordon_epoch_losses_N=128_50epochs_solitonsolitiondoublets} \\
 \hline 
 Run 5                  & V$8$                & $128$ & $3,4,...,15$ & $50$ & $32$ & $0.0$ & $4.0$ &-10.0 & 10.0 & $200$  & \stexttt{range}$(0.72,0.98,20)$ & $[0.73,0.85,0.92,0.97]$ & F\ref{fig:1DsG_red_errors_N=128_50epochs_solitonsolitondoublets}, F\ref{fig:1DsG_proj_errors_N=128_50epochs_solitonsolitondoublets}, F\ref{fig:1DsineGordon_epoch_losses_N=128_50epochs_solitonsolitiondoublets}, F\ref{reconstruction_comparison_solitonsolitiondoublets} \\
 \hline 
 Run 6                  & V$9$                & $128$ & $3,4,...,15$ & $50$ & $32$ & $0.0$ & $4.0$ &-10.0 & 10.0 & $200$  & \stexttt{range}$(0.72,0.98,20)$ & $[0.73,0.85,0.92,0.97]$ & F\ref{fig:1DsG_red_errors_N=128_50epochs_solitonsolitondoublets}, F\ref{fig:1DsG_proj_errors_N=128_50epochs_solitonsolitondoublets}, F\ref{fig:1DsineGordon_epoch_losses_N=128_50epochs_solitonsolitiondoublets} \\
\end{tabular}
}
\caption{Soliton-soliton doublets test runs.}
\label{sec:application_env:table_name:solitonsolitondoubletstestruns}
\end{table}
Analogous to the single soliton conditions, we learned the variants V$2$, V$3$ and V$6$-V$9$ for $50$ epochs and compared the results to the PSD method. 
We used the system parameters
\begin{itemize}
\item $N=128$,
\item $n=3,...,15$,
\item $\I=[0.0, 4.0]$,
\item $\Omega=[-10.0,10.0]$,
\item $\nu_{\tt testing}=0.73, 0.85, 0.92, 0.97$.
\end{itemize}
In contrast to single soliton, we now use the $\nu$-range $\P=[0.72, 0.98] \subseteq (-1,1)$ to build our training set. 
Similar to single soliton, the possible $\nu$-range for the soliton-soliton doublets conditions is the interval $(-1,1)$.  
Again, we have a~faster movement in time for $\nu$-values around $-1$ and $1$.
Thus, we now decided to choose $\P$ as a~sub-interval near $1$. 
The reduction errors are displayed in Figure~\ref{fig:1DsG_red_errors_N=128_50epochs_solitonsolitondoublets}, the projection errors in Figure~\ref{fig:1DsG_proj_errors_N=128_50epochs_solitonsolitondoublets}
and the average epoch loss in Figure~\ref{fig:1DsineGordon_epoch_losses_N=128_50epochs_solitonsolitiondoublets}.
The results are similar to that we got for the single soliton conditions. Compared to the PSD method, 
we see that the autoencoder errors are smaller for small reduced dimensions and that the PSD method catches up the larger the dimensions become.
We also see a~tendency towards better results for small $\nu$-values. 
For $\nu=0.92$ and $\nu=0.97$, we observe a~better performance of the StiefelAdam variants in comparison to the homogeneous Adam variants V$2$ and V$3$.
These performance differences can also be observed in the reconstructions for $\nu=0.97$ displayed in Figure~\ref{reconstruction_comparison_solitonsolitiondoublets}. 
The average epoch loss shows the same pattern as for the single soliton conditions. We have the fastest loss decrease for V$2$ in small dimensions 
(here $n=5$) and the slowest decrease for V$3$.

\newpage
\section{Conclusion}

This master thesis dealt with the structure-preserving model reduction of Hamiltonian systems with the goal of solving them efficiently. 
We discussed an~autoencoder network, first introduced by \cite{main:brantner2023}, that learns a~symplectic encoder-decoder pair 
to conduct the model reduction. 
In particular, we presented our own variations of the network and training setup. We paid special attention to the manifold 
update step of the network and introduced a~new StiefelAdam manifold optimizer step that works directly on the Stiefel manifold by modifying the 
classical Adam optimizer and avoids the use of homogeneous spaces as proposed in \cite{main:brantner2023}. 
We compared the two different approaches (as well as other network and training modifications) and 
were able to identify some advantages of the StiefelAdam optimizer over the HomogeneousAdam approach.
The significantly better numerical efficiency, the faster loss decrease in the first epochs 
and the better performance for larger variable system parameter values (i.e. $\mu=0.625$ and $\nu= 0.97$) stood out in particular.
But we could also observe some drawbacks of our approach. 
The most surprising drawback is the strong slow-down of the loss decrease 
after only a~few epochs, which we did not observe for the HomogeneousAdam approach. 
Further improvements can also be made by using a~more elaborate decay of the learning rate by utilizing 
information of previous iterations. Testing other system parameter constellations may yield some further improvements. 
Furthermore, it might also be worthwhile to take a~look at the symplectic preprocessing layers and try to improve in terms of numerical efficiency.
This offers many opportunities for further investigation and research on this optimizer and the network in general to address these issues 
and improve the results.

\newpage
\printbibliography[heading=bibintoc]


\newpage \appendix

\section{Appendix: Plots}

\begin{figure}[ht!]
\resizebox{\textwidth}{!}{%
  \centering
  \begin{subfigure}{0.43\textwidth}
    \centering
    \rotatebox{90}{\includegraphics[width=1.15\textwidth, height=.26\textheight]{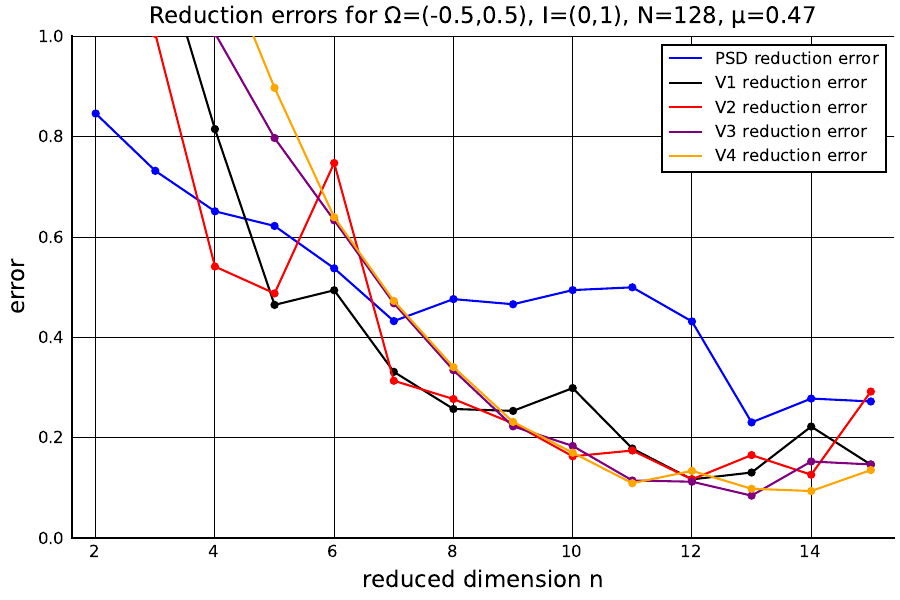}}
    \caption{$\mu=0.47$}
    \label{img:1Dwave_red_errors_N=128_100epochs_PSD+V1-V4_047}
  \end{subfigure} 
  \begin{subfigure}{0.43\textwidth}
    \centering
    \rotatebox{90}{\includegraphics[width=1.15\textwidth, height=.26\textheight]{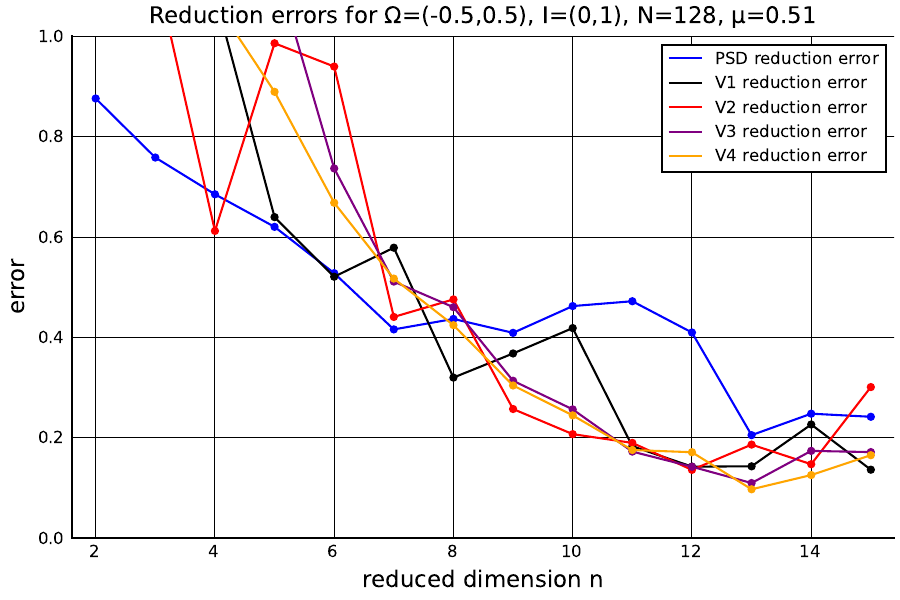}}
    \caption{$\mu=0.51$}
    \label{img:1Dwave_red_errors_N=128_100epochs_PSD+V1-V4_051}
  \end{subfigure}
}  
\resizebox{\textwidth}{!}{%
  \begin{subfigure}{0.43\textwidth}
    \centering
    \rotatebox{90}{\includegraphics[width=1.15\textwidth, height=.26\textheight]{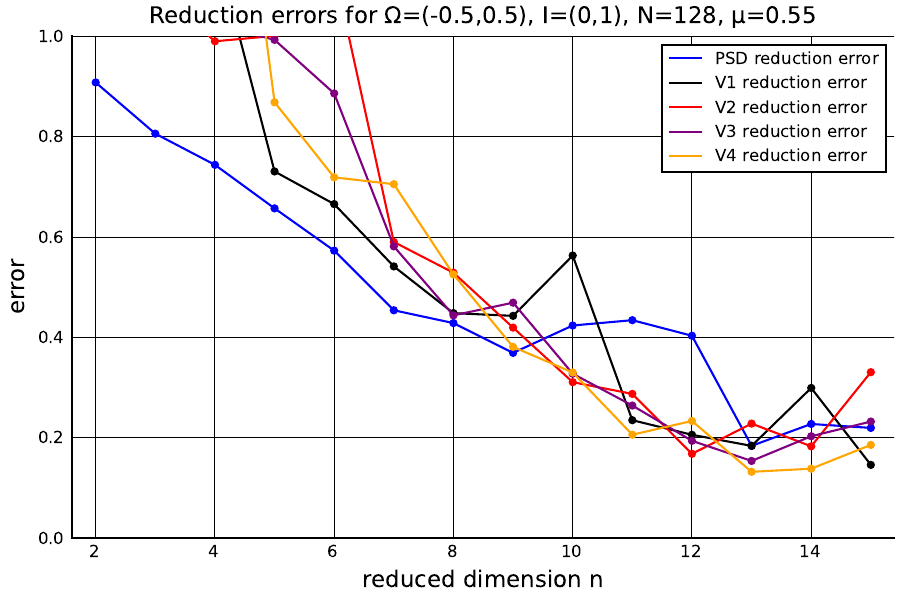}}
    \caption{$\mu=0.55$}
    \label{img:1Dwave_red_errors_N=128_100epochs_PSD+V1-V4_055}
  \end{subfigure}
  \begin{subfigure}{0.43\textwidth}
    \centering
    \rotatebox{90}{\includegraphics[width=1.15\textwidth, height=.26\textheight]{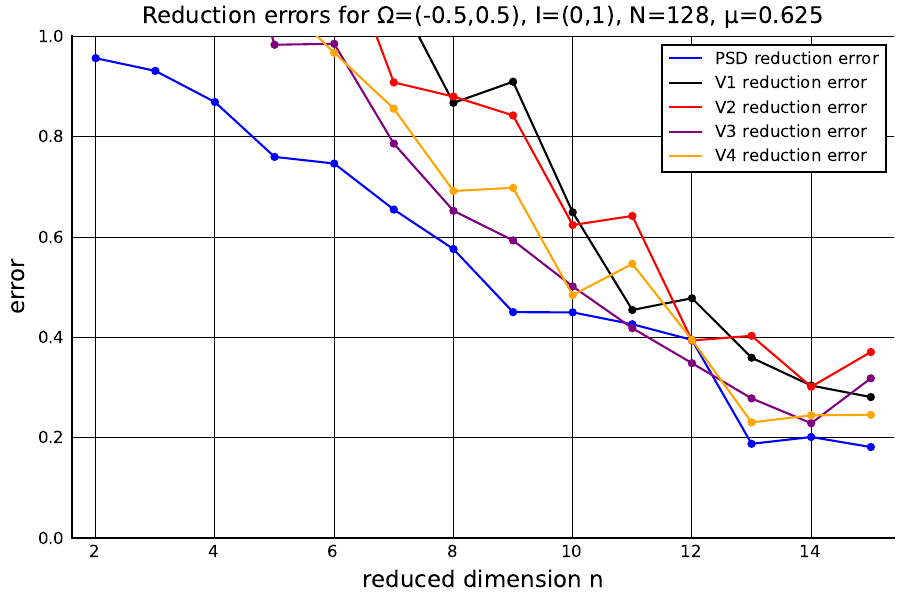}}
    \caption{$\mu=0.625$}
    \label{img:1Dwave_red_errors_N=128_100epochs_PSD+V1-V4_0625}
  \end{subfigure}
}
\caption{$1$D linear wave equation reduction errors for variants V$1$-V$4$ and the PSD method in dimensions $N=128$ and $n=2,3,...,15$ for $100$ epochs test runs.}
  \label{fig:1Dwave_red_errors_N=128_100epochs_PSD+V1-V4}
\end{figure}

\begin{figure}[ht!] 
\resizebox{\textwidth}{!}{%
  \centering
  \begin{subfigure}{0.43\textwidth}
    \centering
    \rotatebox{90}{\includegraphics[width=1.2\textwidth, height=.27\textheight]{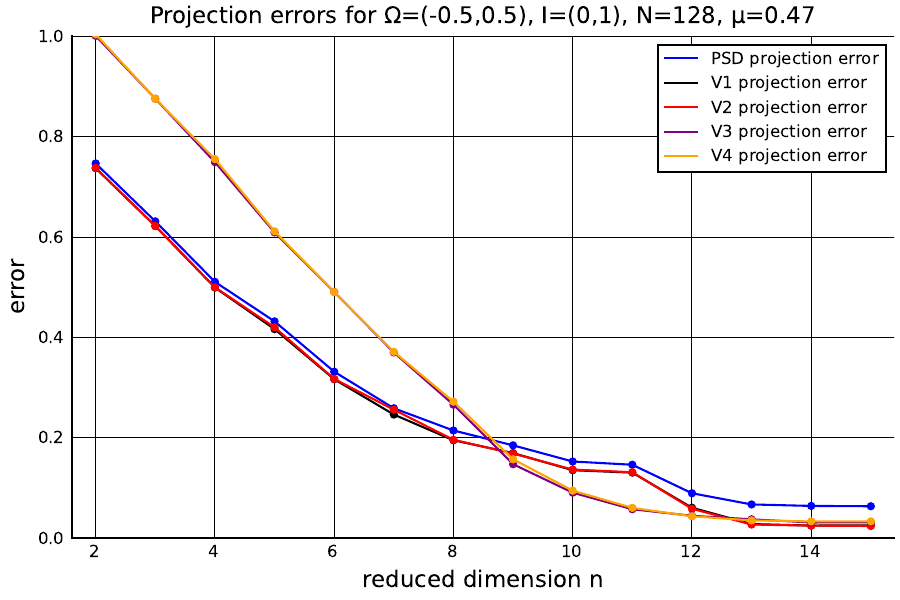}}
    \caption{$\mu=0.47$}
    \label{img:1Dwave_proj_errors_N=128_100epochs_PSD+V1-V4_047}
  \end{subfigure} 
  \begin{subfigure}{0.43\textwidth}
    \centering
    \rotatebox{90}{\includegraphics[width=1.2\textwidth, height=.27\textheight]{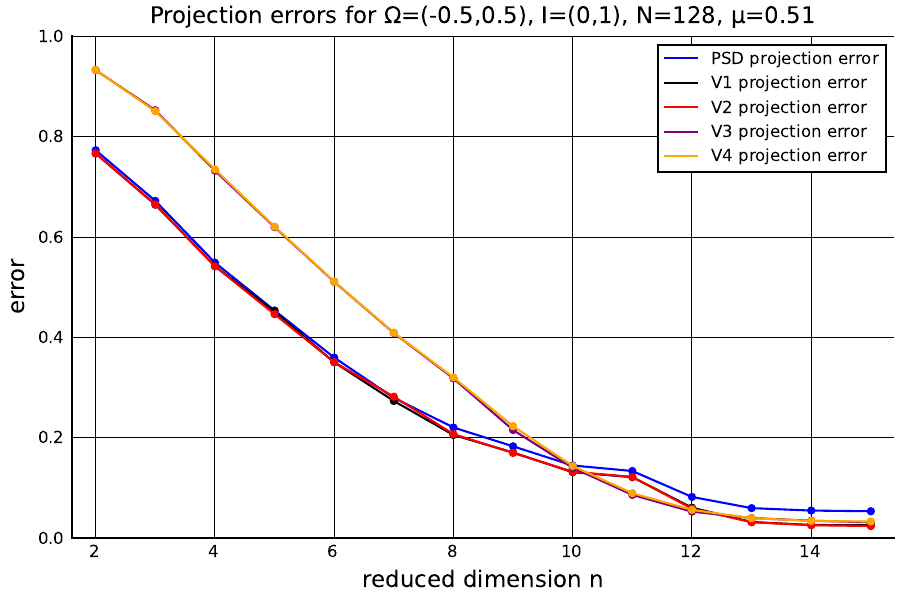}}
    \caption{$\mu=0.51$}
    \label{img:1Dwave_proj_errors_N=128_100epochs_PSD+V1-V4_051}
  \end{subfigure}
}  
\resizebox{\textwidth}{!}{%
  \begin{subfigure}{0.43\textwidth}
    \centering
    \rotatebox{90}{\includegraphics[width=1.2\textwidth, height=.27\textheight]{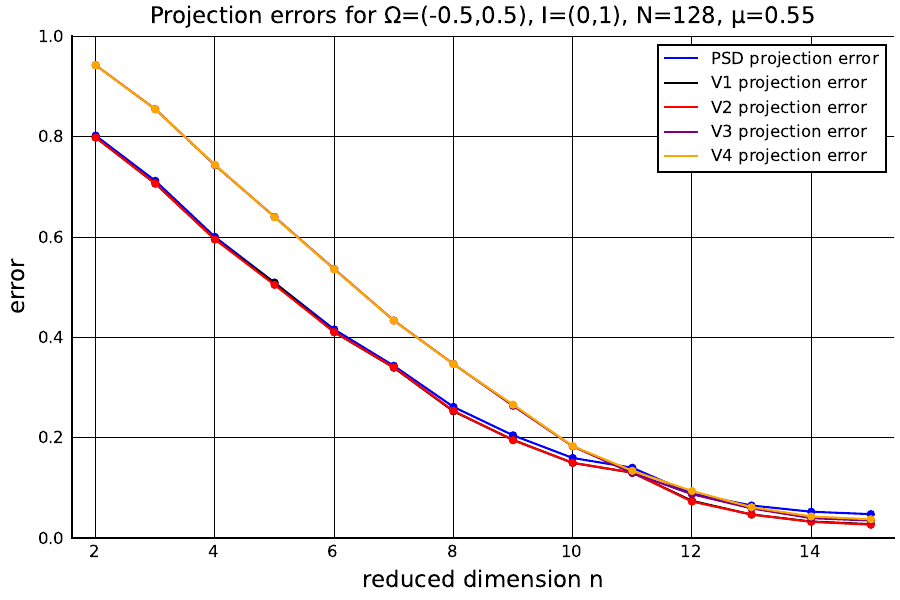}}
    \caption{$\mu=0.55$}
    \label{img:1Dwave_proj_errors_N=128_100epochs_PSD+V1-V4_055}
  \end{subfigure}
  \begin{subfigure}{0.43\textwidth}
    \centering
    \rotatebox{90}{\includegraphics[width=1.2\textwidth, height=.27\textheight]{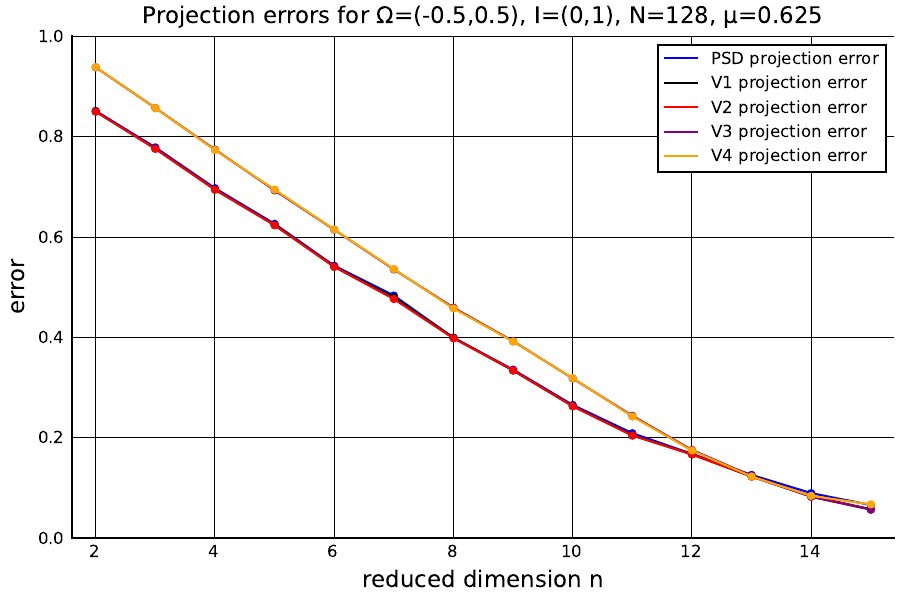}}
    \caption{$\mu=0.625$}
    \label{img:1Dwave_proj_errors_N=128_100epochs_PSD+V1-V4_0625}
  \end{subfigure}
}
\caption{$1$D linear wave equation projection errors for variants V$1$-V$4$ and the PSD method in dimensions $N=128$ and $n=2,3,...,15$ for $100$ epochs test runs.}
  \label{fig:1Dwave_proj_N=128_100epochs_PSD+V1-V4}
\end{figure}

\begin{figure}[ht!] 
\resizebox{\textwidth}{!}{%
  \centering
  \begin{subfigure}{0.43\textwidth}
    \centering
    \rotatebox{90}{\includegraphics[width=1.2\textwidth, height=.27\textheight]{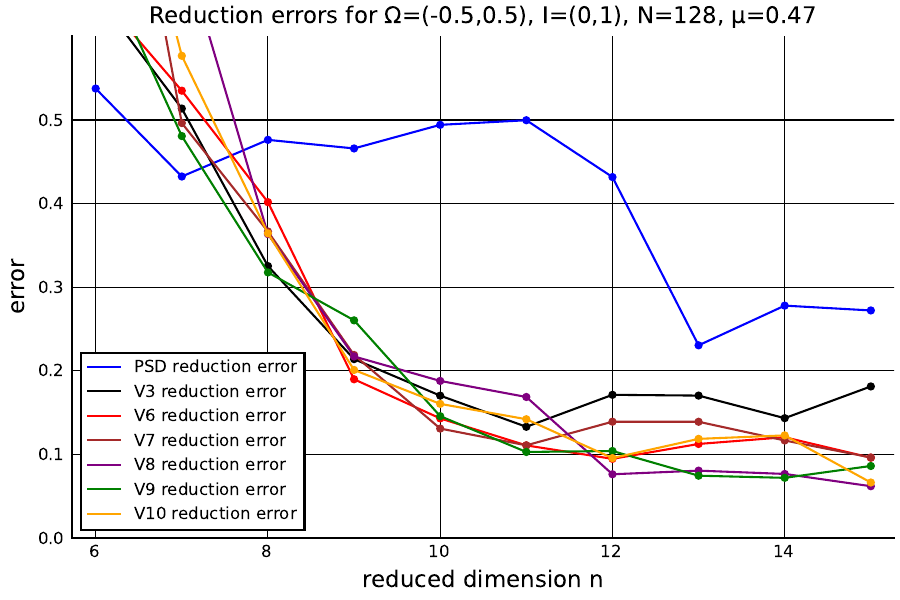}}
    \caption{$\mu=0.47$}
    \label{img:1Dwave_red_errors_N=128_50epochs_PSD+V1-V4_047}
  \end{subfigure} 
  \begin{subfigure}{0.43\textwidth}
    \centering
    \rotatebox{90}{\includegraphics[width=1.2\textwidth, height=.27\textheight]{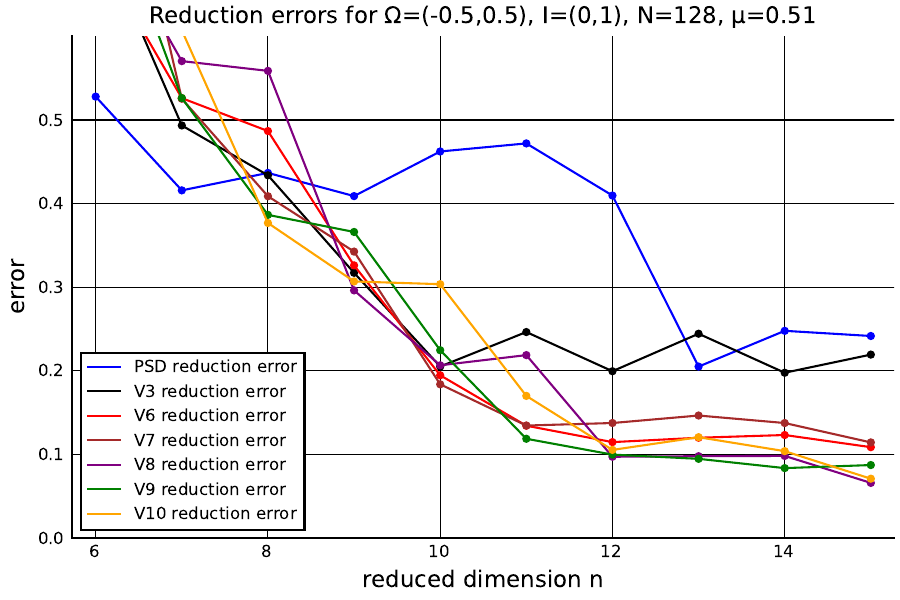}}
    \caption{$\mu=0.51$}
    \label{img:1Dwave_red_errors_N=128_50epochs_PSD+V1-V4_051}
  \end{subfigure}
}  
\resizebox{\textwidth}{!}{%
  \begin{subfigure}{0.43\textwidth}
    \centering
    \rotatebox{90}{\includegraphics[width=1.2\textwidth, height=.27\textheight]{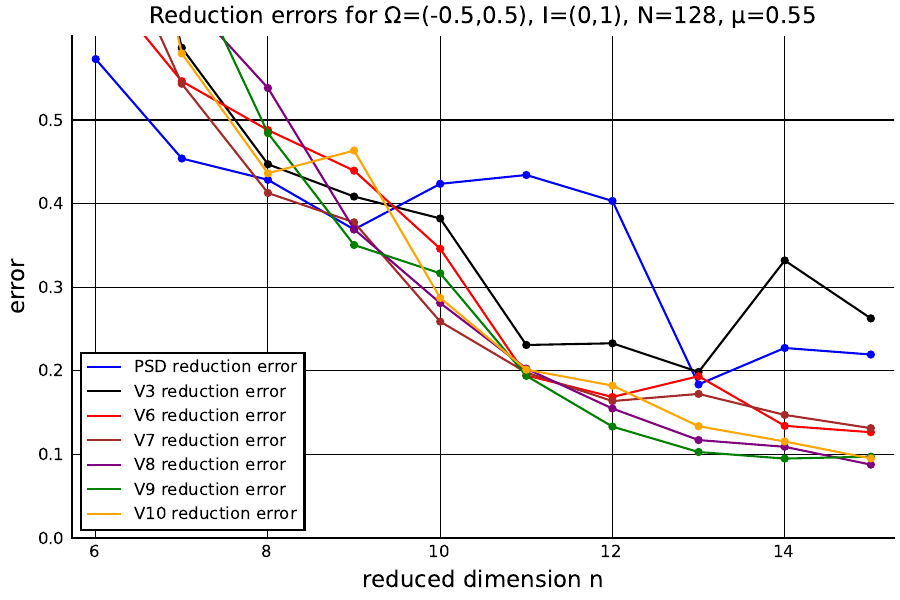}}
    \caption{$\mu=0.55$}
    \label{img:1Dwave_red_errors_N=128_50epochs_PSD+V1-V4_055}
  \end{subfigure}
  \begin{subfigure}{0.43\textwidth}
    \centering
    \rotatebox{90}{\includegraphics[width=1.2\textwidth, height=.27\textheight]{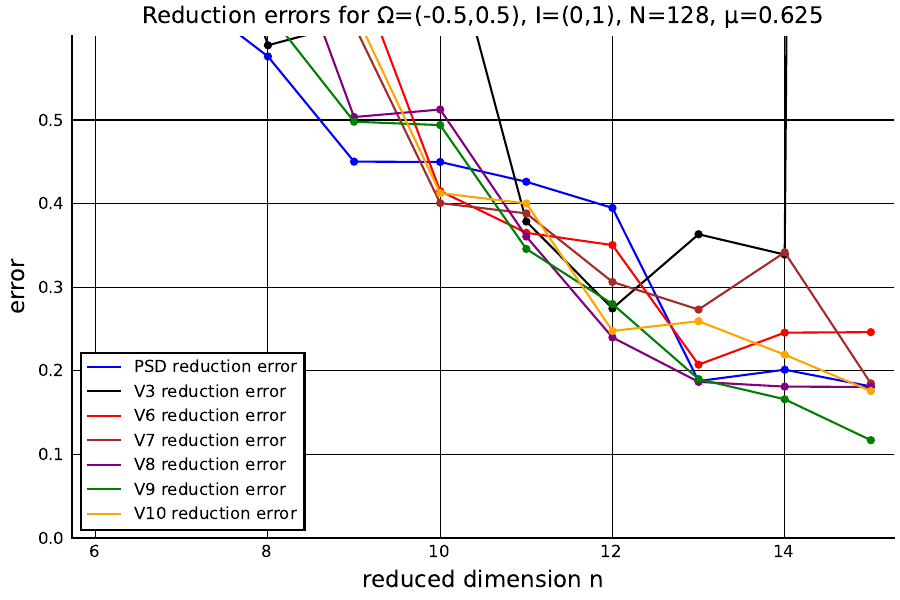}}
    \caption{$\mu=0.625$}
    \label{img:1Dwave_red_errors_N=128_50epochs_PSD+V1-V4_0625}
  \end{subfigure}
}
\caption{$1$D linear wave equation reduction errors for variants V$3$, V$6$-V$10$ and the PSD method in dimensions $N=128$ and $n=6,7,...,15$ for $50$ epochs test runs.}
  \label{fig:1Dwave_red_errors_N=128_50epochs_PSD+V3_V6_V10}
\end{figure}

\begin{figure}[ht!] 
\resizebox{\textwidth}{!}{%
  \centering
    \includegraphics[width=0.05\textwidth]{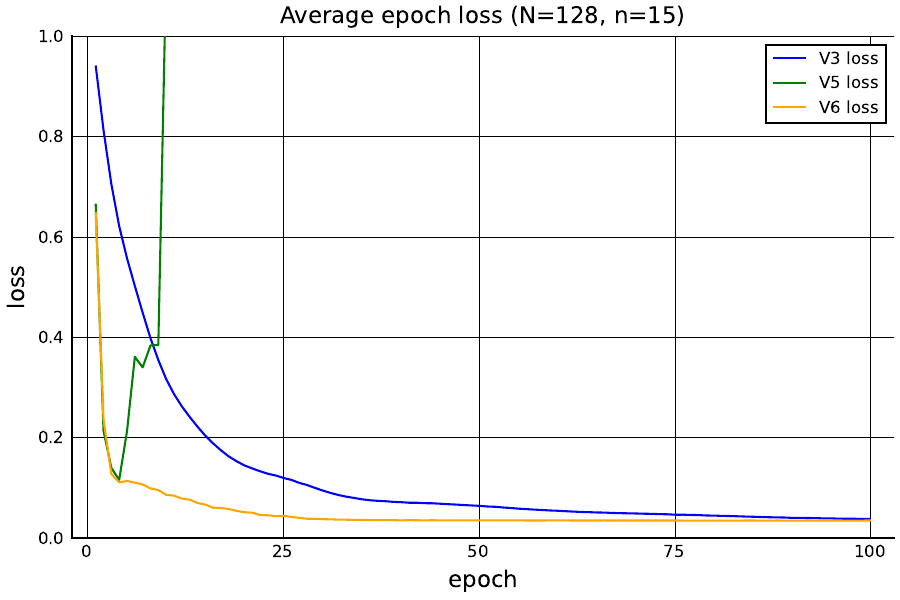}
}
\caption{$1$D linear wave equation average epoch target function losses for variants V$3$, V$5$ and V$6$ in dimensions $N=128$ and $n=15$ for $100$ epochs test runs.}
  \label{fig:1Dwave_epoch_losses_N=128_100epochs_V3_V5_V6}
\end{figure}

\begin{figure}[ht!] 
\resizebox{\textwidth}{!}{%
  \centering
    \includegraphics[width=0.05\textwidth]{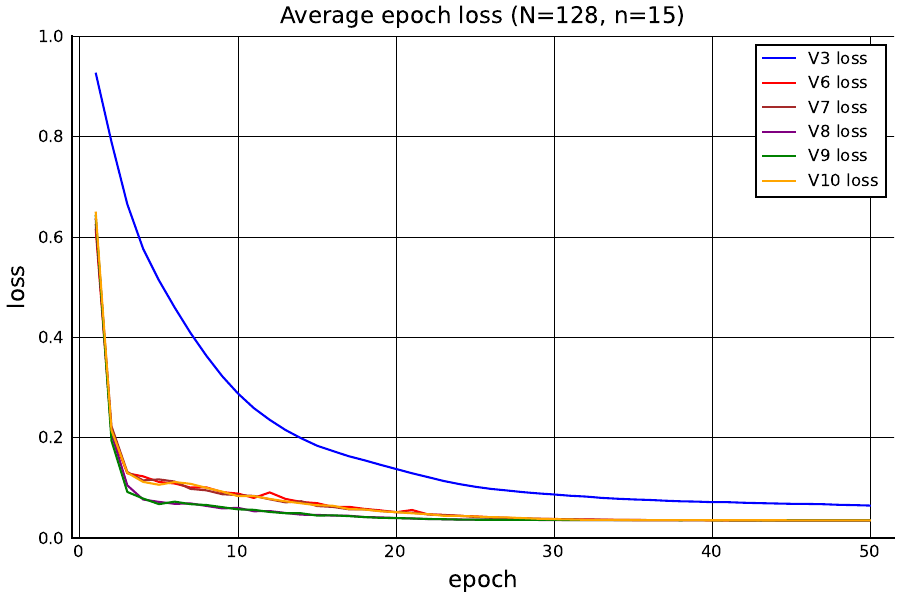}
}
\caption{$1$D linear wave equation average epoch target function losses for variants V$3$ and V$6$-V$10$ in dimensions $N=128$ and $n=15$ for $50$ epochs test runs.}
  \label{fig:1Dwave_epoch_losses_N=128_50epochs_V3_V6-V10}
\end{figure}

\begin{figure}[ht!] 
\resizebox{\textwidth}{!}{%
  \centering
  \begin{subfigure}{0.43\textwidth}
    \centering
    \rotatebox{90}{\includegraphics[width=1.2\textwidth, height=.27\textheight]{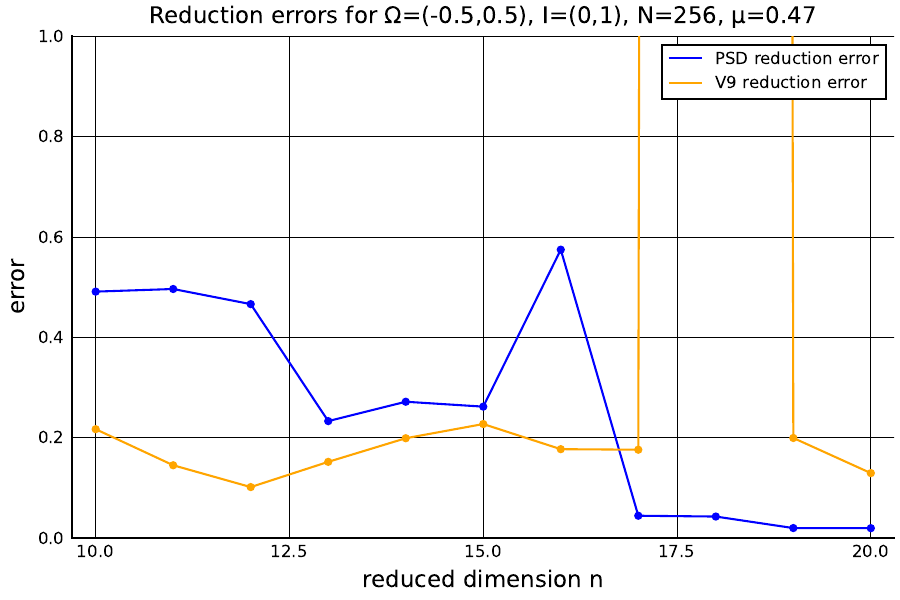}}
    \caption{$\mu=0.47$}
  \end{subfigure} 
  \begin{subfigure}{0.43\textwidth}
    \centering
    \rotatebox{90}{\includegraphics[width=1.2\textwidth, height=.27\textheight]{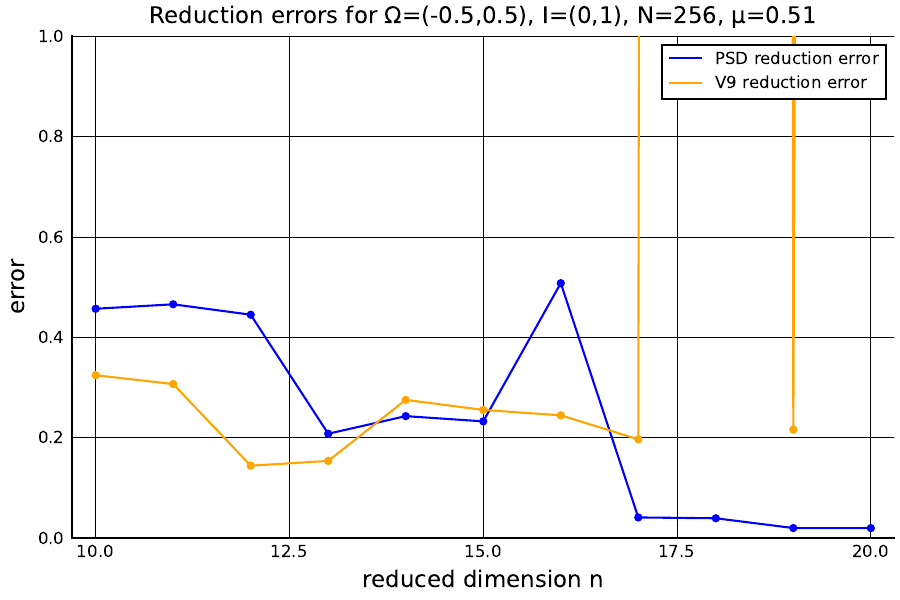}}
    \caption{$\mu=0.51$}
  \end{subfigure}
}  
\resizebox{\textwidth}{!}{%
  \begin{subfigure}{0.43\textwidth}
    \centering
    \rotatebox{90}{\includegraphics[width=1.2\textwidth, height=.27\textheight]{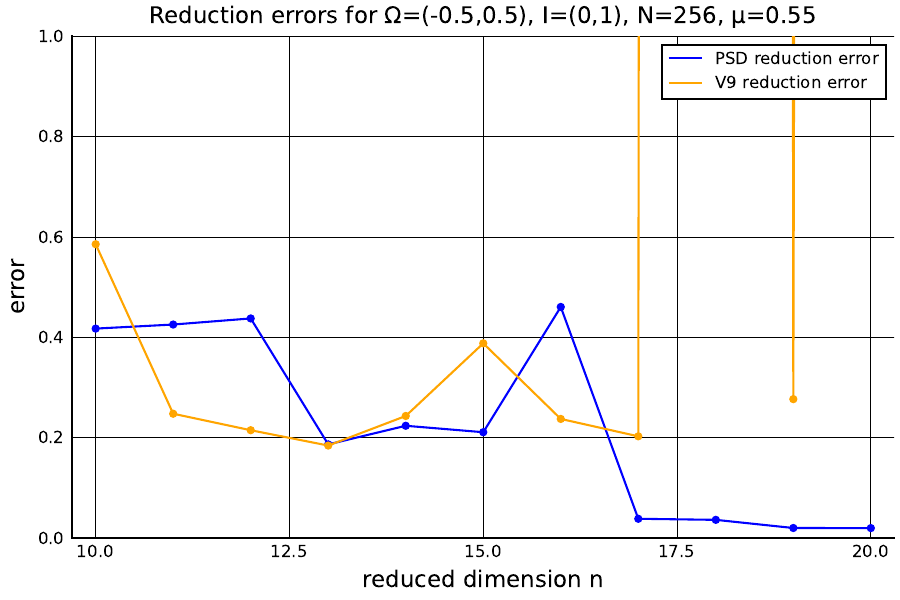}}
    \caption{$\mu=0.55$}
  \end{subfigure}
  \begin{subfigure}{0.43\textwidth}
    \centering
    \rotatebox{90}{\includegraphics[width=1.2\textwidth, height=.27\textheight]{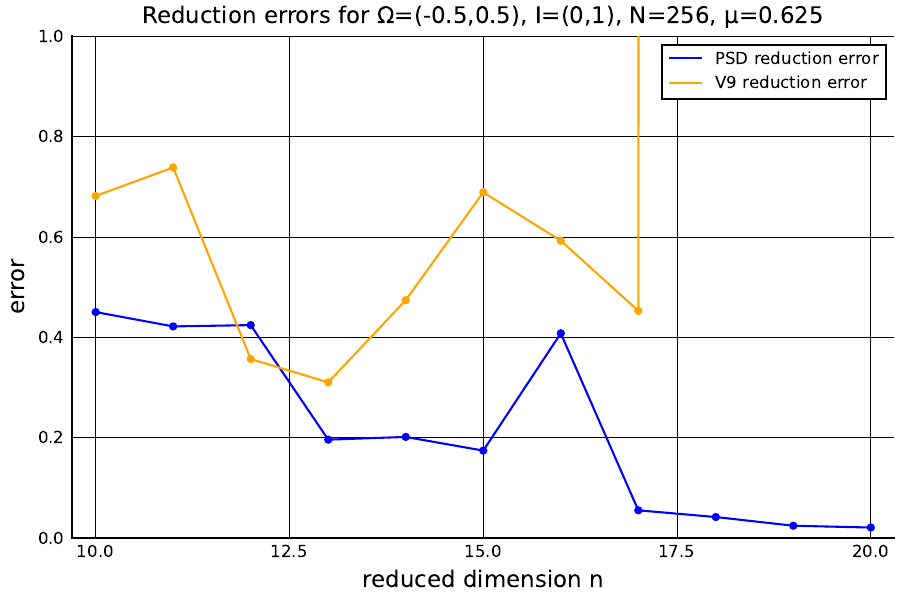}}
    \caption{$\mu=0.625$}
  \end{subfigure}
}
\caption{Direct comparison of the $1$D linear wave equation reduction errors for variant V$9$ and the PSD method in dimensions $N=256$ and $n=10,...,20$ for $50$ epochs test runs.}
  \label{fig:1Dwave_red_errors_N=256_50epochs_V9vsPSD}
\end{figure}

\begin{figure}[ht!] 
  \centering
\begin{subfigure}{0.45\linewidth}
    \centering
    \includegraphics[width=\linewidth, scale=.35]{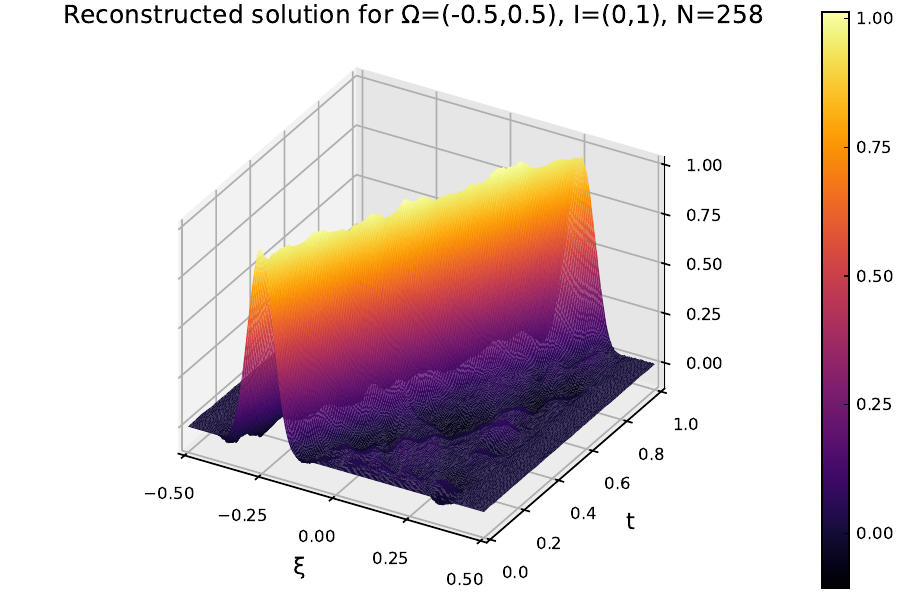}
    \caption{$\mu=0.51$, V$9$.}
  \end{subfigure}
  \hfill
  \begin{subfigure}{0.45\linewidth}
    \centering
    \includegraphics[width=\linewidth, scale=.35]{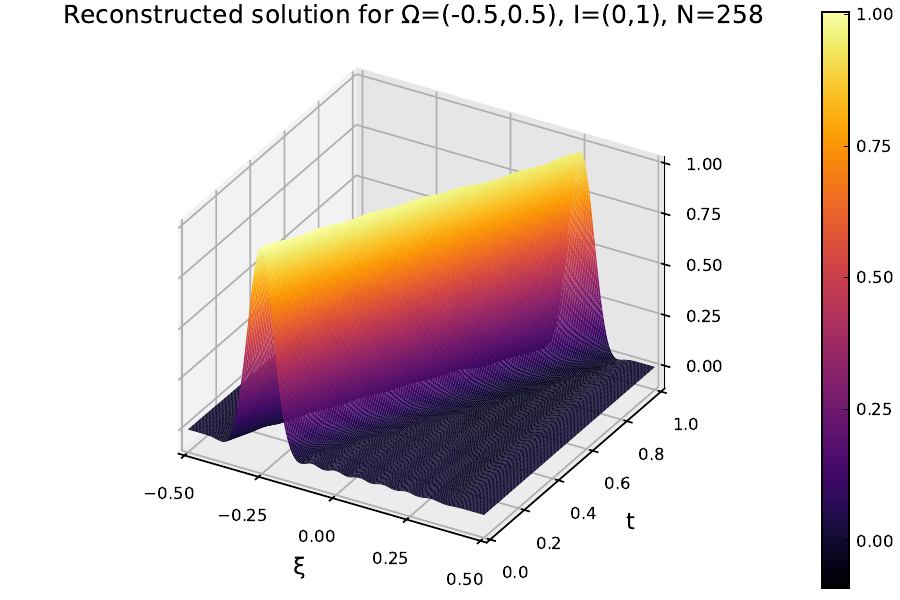}
    \caption{$\mu=0.51$, PSD.}
  \end{subfigure}

  \begin{subfigure}{0.45\linewidth}
    \centering
    \includegraphics[width=\linewidth, scale=.35]{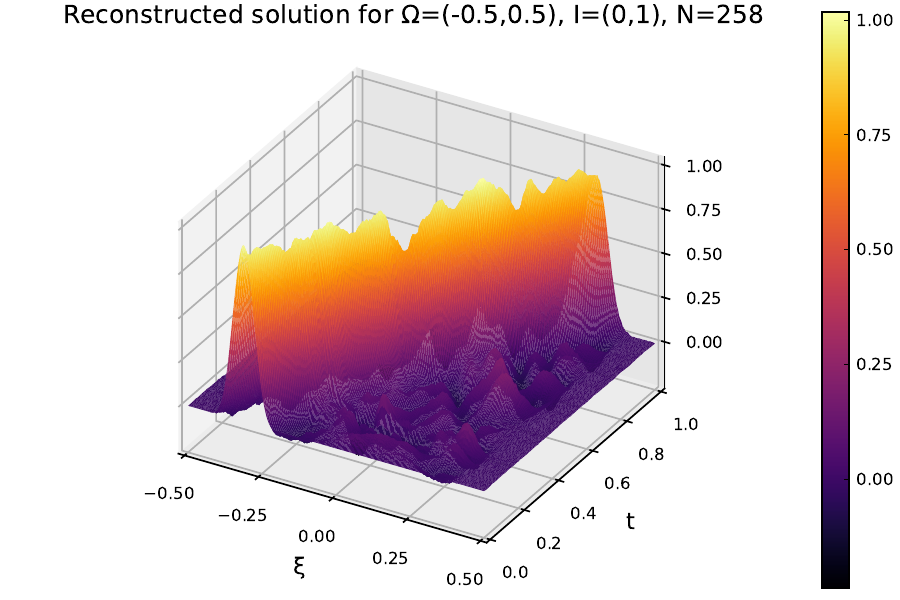}
    \caption{$\mu=0.625$, V$9$.}
  \end{subfigure}
  \hfill
  \begin{subfigure}{0.45\linewidth}
    \centering
    \includegraphics[width=\linewidth, scale=.35]{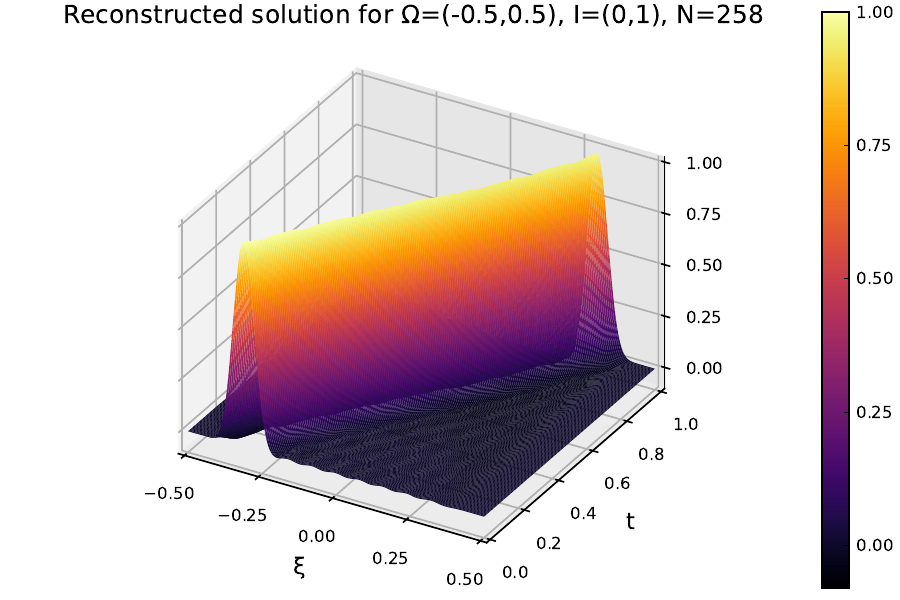}
    \caption{$\mu=0.625$, PSD.}
  \end{subfigure}
  \caption{Comparison of the reconstruction of the $1$D wave equation from the integrated ROM solution in dimension $n=17$ to the full dimension $N=256$ using the autoencoder variant V$9$ ($50$ epochs) and the PSD method.}
  \label{reconstruction comparison}
\end{figure}

\begin{figure}[ht!] 
\resizebox{\textwidth}{!}{%
  \centering
  \begin{subfigure}{0.43\textwidth}
    \centering
    \rotatebox{90}{\includegraphics[width=1.2\textwidth, height=.27\textheight]{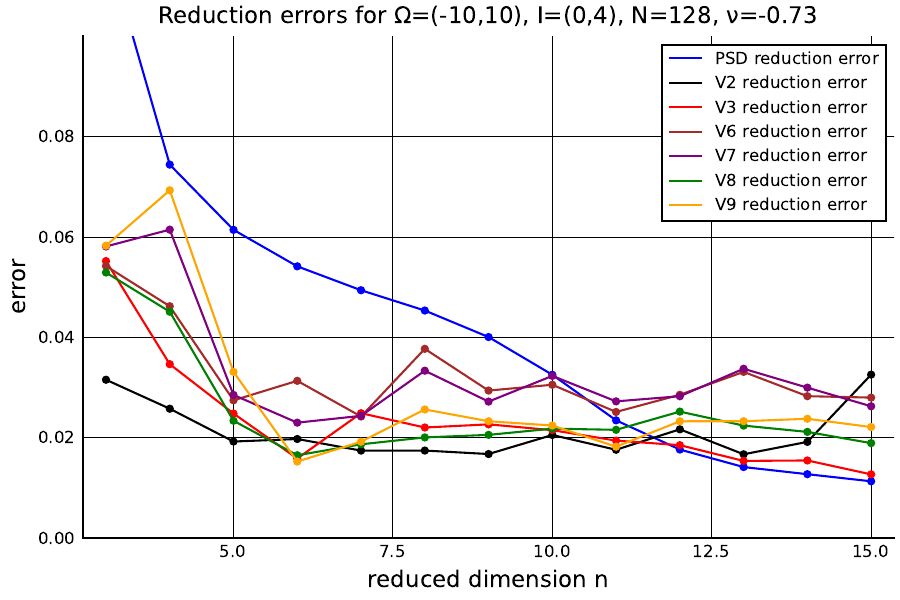}}
    \caption{$\nu=-0.73$}
  \end{subfigure} 
  \begin{subfigure}{0.43\textwidth}
    \centering
    \rotatebox{90}{\includegraphics[width=1.2\textwidth, height=.27\textheight]{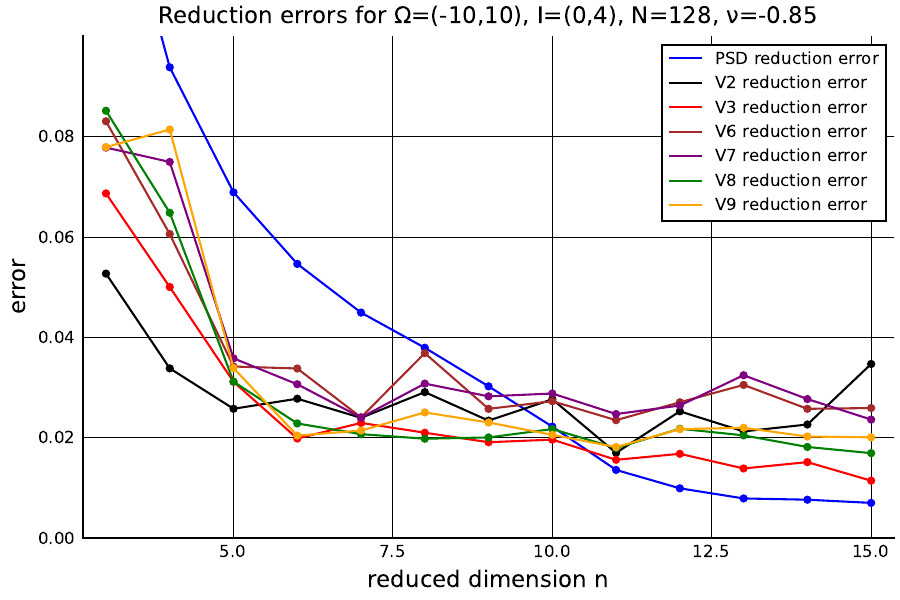}}
    \caption{$\nu=-0.85$}
  \end{subfigure}
}  
\resizebox{\textwidth}{!}{%
  \begin{subfigure}{0.43\textwidth}
    \centering
    \rotatebox{90}{\includegraphics[width=1.2\textwidth, height=.27\textheight]{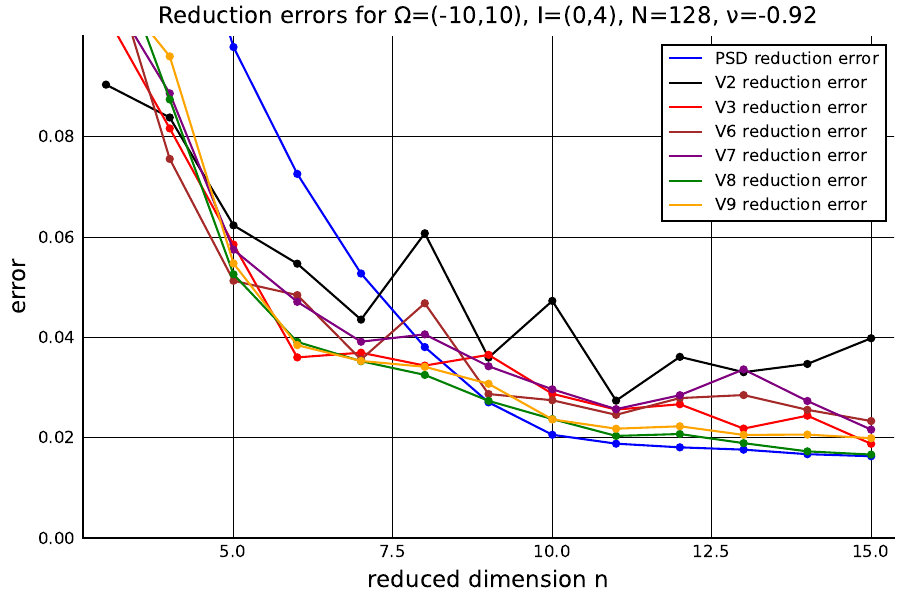}}
    \caption{$\nu=-0.92$}
  \end{subfigure}
  \begin{subfigure}{0.43\textwidth}
    \centering
    \rotatebox{90}{\includegraphics[width=1.2\textwidth, height=.27\textheight]{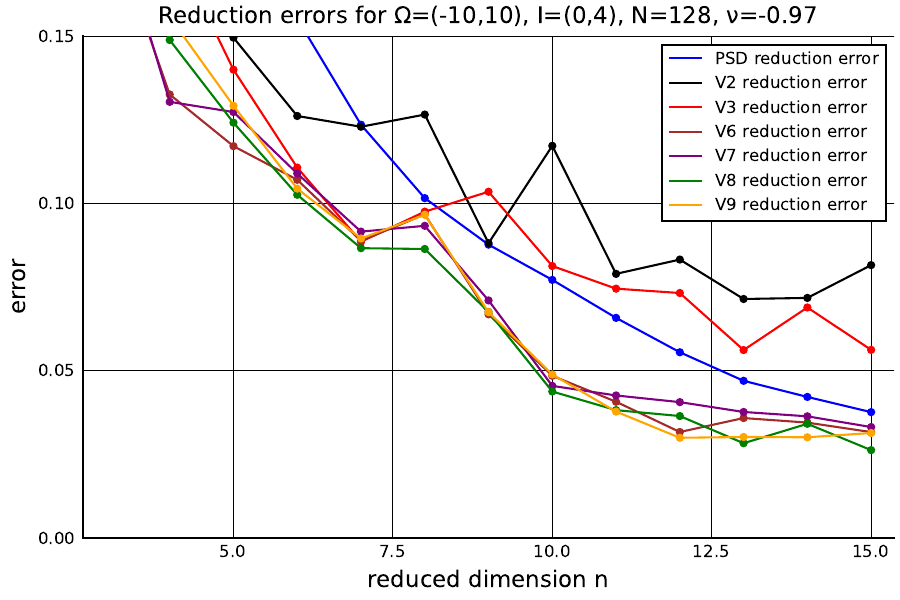}}
    \caption{$\nu=-0.97$}
  \end{subfigure}
}
\caption{Reduction errors for the $1$D sine-Gordon equation with single soliton boundary conditions in dimensions $N=128$ and $n=3,4,...,15$ and $50$ epochs test runs.}
  \label{fig:1DsG_red_errors_N=128_50epochs_singlesoliton}
\end{figure}

\begin{figure}[ht!] 
\resizebox{\textwidth}{!}{%
  \centering
  \begin{subfigure}{0.43\textwidth}
    \centering
    \rotatebox{90}{\includegraphics[width=1.2\textwidth, height=.27\textheight]{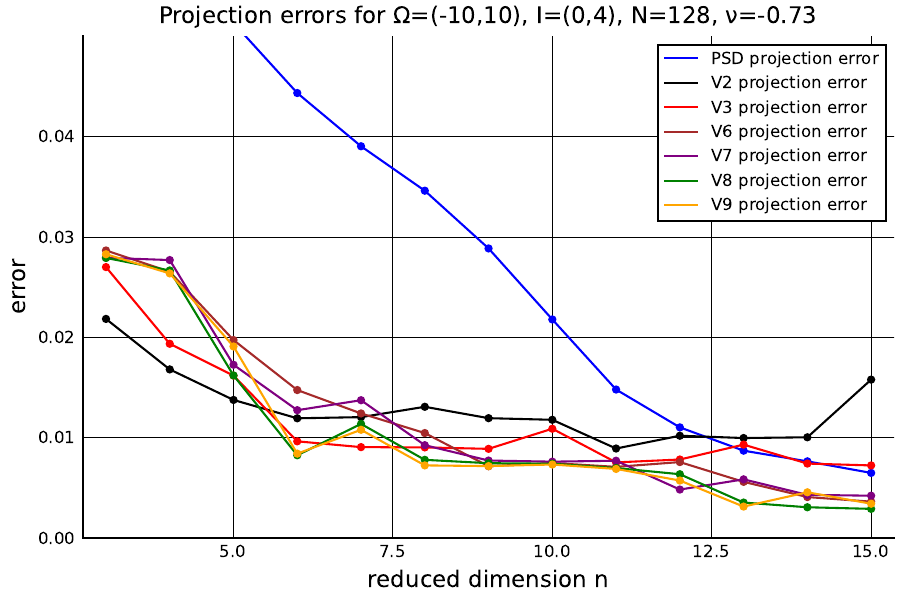}}
    \caption{$\nu=-0.73$}
  \end{subfigure} 
  \begin{subfigure}{0.43\textwidth}
    \centering
    \rotatebox{90}{\includegraphics[width=1.2\textwidth, height=.27\textheight]{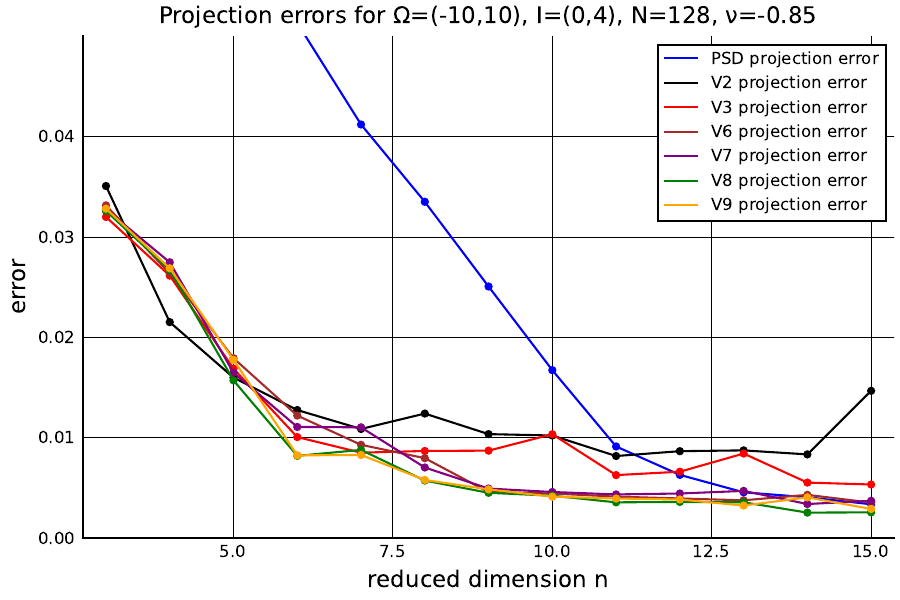}}
    \caption{$\nu=-0.85$}
  \end{subfigure}
}  
\resizebox{\textwidth}{!}{%
  \begin{subfigure}{0.43\textwidth}
    \centering
    \rotatebox{90}{\includegraphics[width=1.2\textwidth, height=.27\textheight]{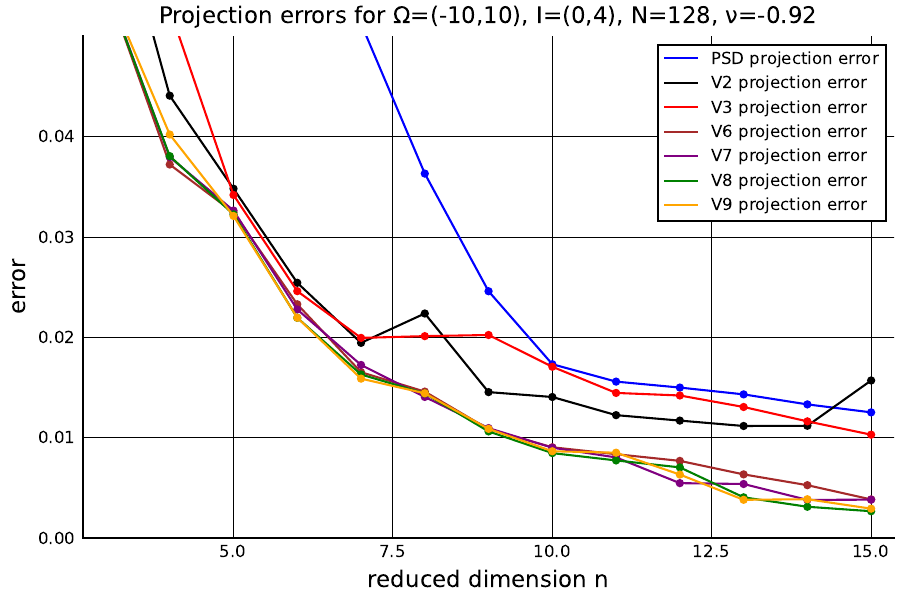}}
    \caption{$\nu=-0.92$}
  \end{subfigure}
  \begin{subfigure}{0.43\textwidth}
    \centering
    \rotatebox{90}{\includegraphics[width=1.2\textwidth, height=.27\textheight]{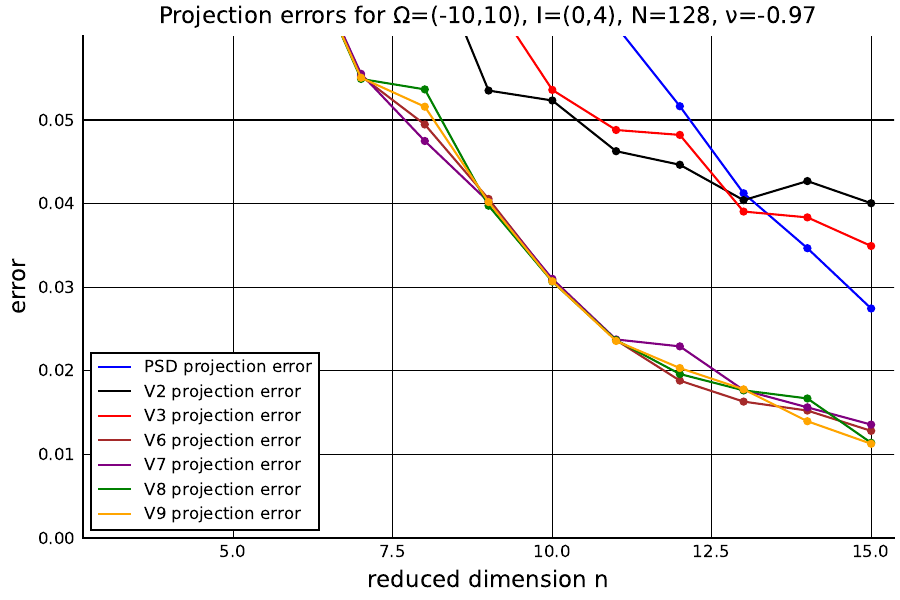}}
    \caption{$\nu=-0.97$}
  \end{subfigure}
}
\caption{Projection errors for the $1$D sine-Gordon equation with single soliton boundary conditions in dimensions $N=128$ and $n=3,4,...,15$ and $50$ epochs test runs.}
  \label{fig:1DsG_proj_errors_N=128_50epochs_singlesoliton}
\end{figure}

\begin{figure}[ht!] 
  \begin{subfigure}{\textwidth}
    \resizebox{\textwidth}{!}{%
      \centering
        \includegraphics[width=0.05\textwidth]{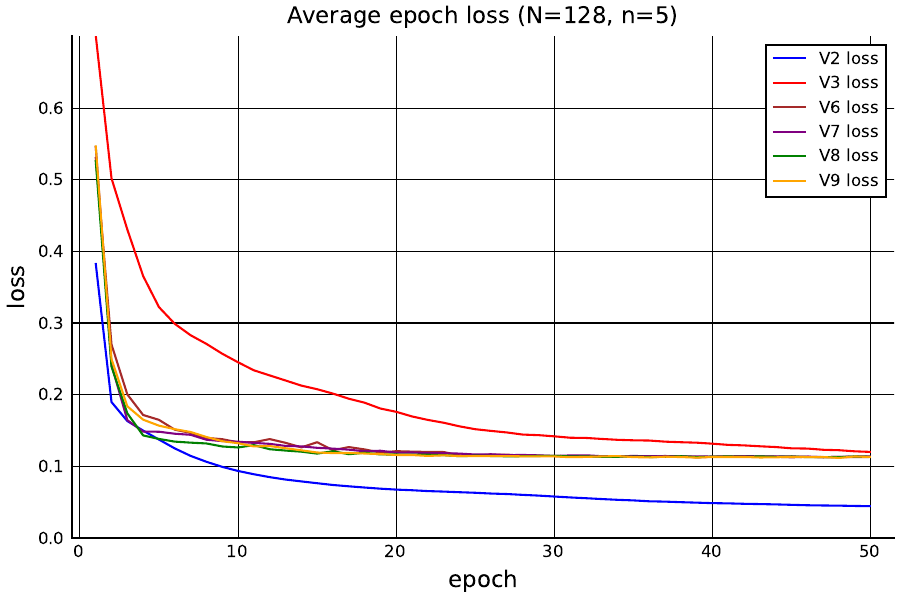}
    }
    \caption{$n=5$}
  \end{subfigure}
  \begin{subfigure}{\textwidth}
    \resizebox{\textwidth}{!}{%
      \centering
        \includegraphics[width=0.05\textwidth]{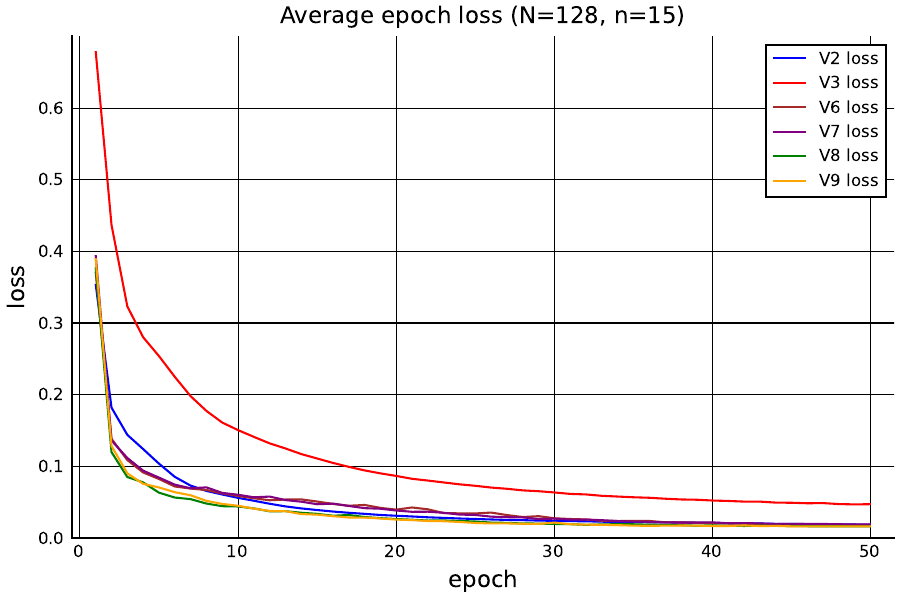}
    }
    \caption{$n=15$}
  \end{subfigure}
\caption{Comparison of $1$D sine-Gordon equation (single soliton) average epoch target function losses for variants V$2$, V$3$ and V$6$-V$9$ in dimension $N=128$ for $50$ epochs test runs.}
  \label{fig:1DsineGordon_epoch_losses_N=128_50epochs_singlesoliton}
\end{figure}

\begin{figure}[ht!] 
  \centering
\begin{subfigure}{0.45\linewidth}
    \centering
    \includegraphics[width=\linewidth, scale=.35]{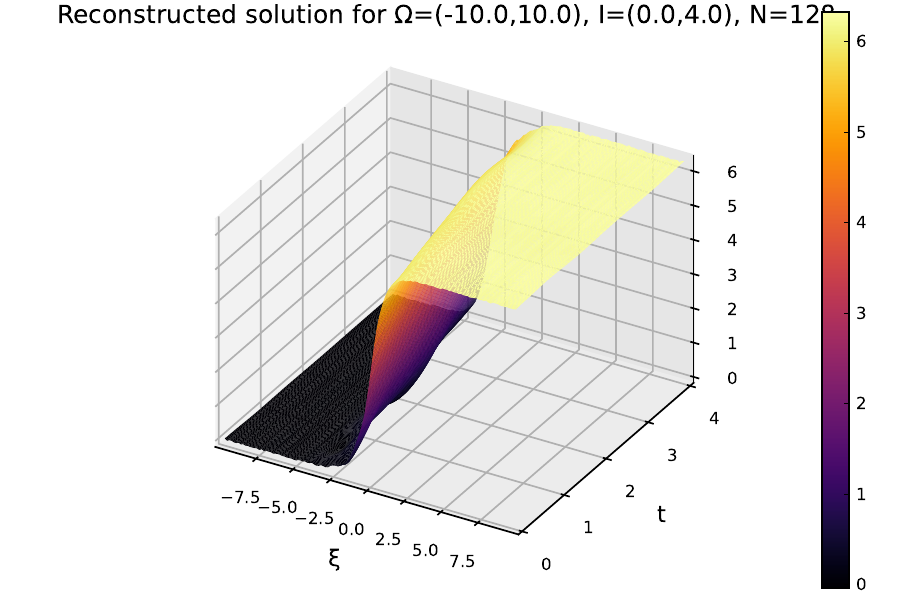}
    \caption{V$2$}
  \end{subfigure}
  \hfill
  \begin{subfigure}{0.45\linewidth}
    \centering
    \includegraphics[width=\linewidth, scale=.35]{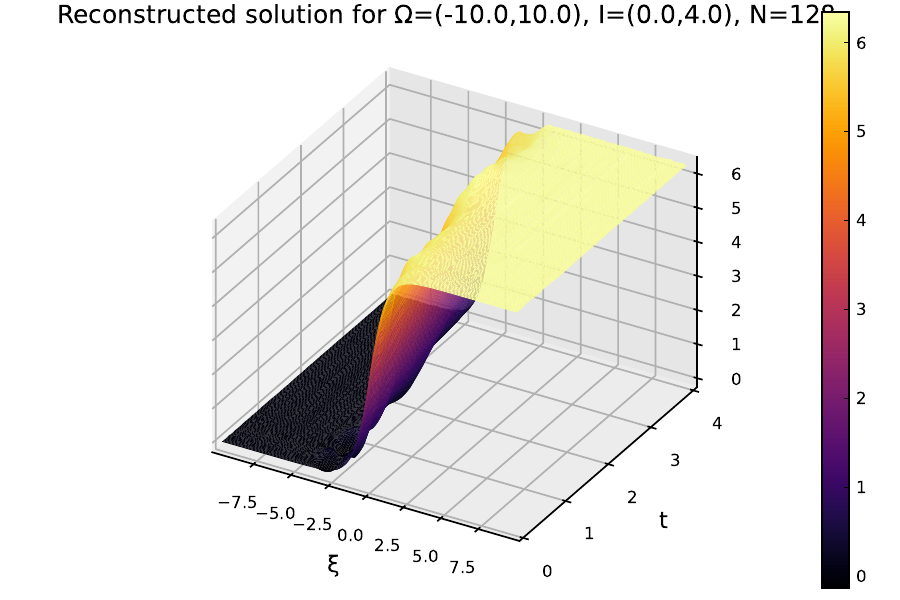}
    \caption{V$3$}
  \end{subfigure}

  \begin{subfigure}{0.45\linewidth}
    \centering
    \includegraphics[width=\linewidth, scale=.35]{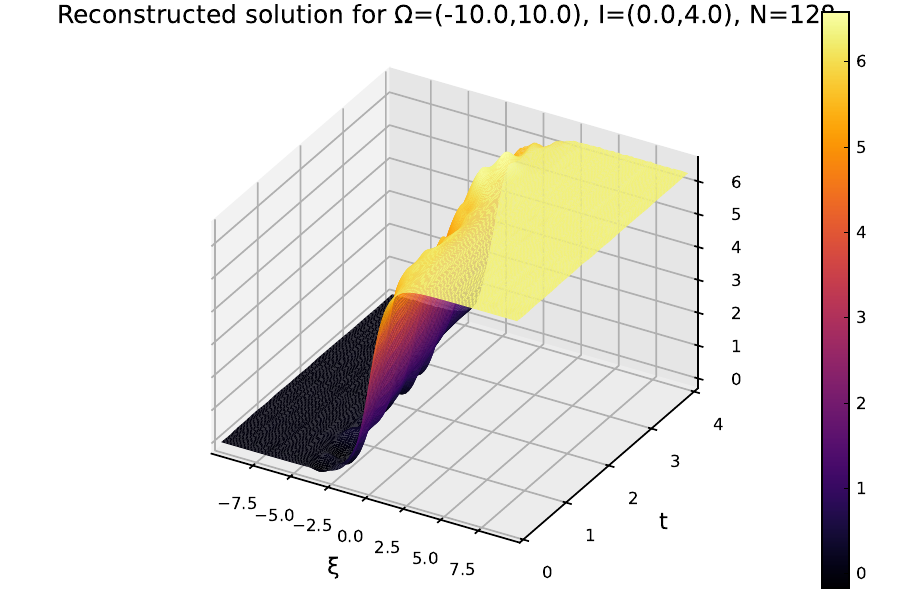}
    \caption{V$8$}
  \end{subfigure}
  \hfill
  \begin{subfigure}{0.45\linewidth}
    \centering
    \includegraphics[width=\linewidth, scale=.35]{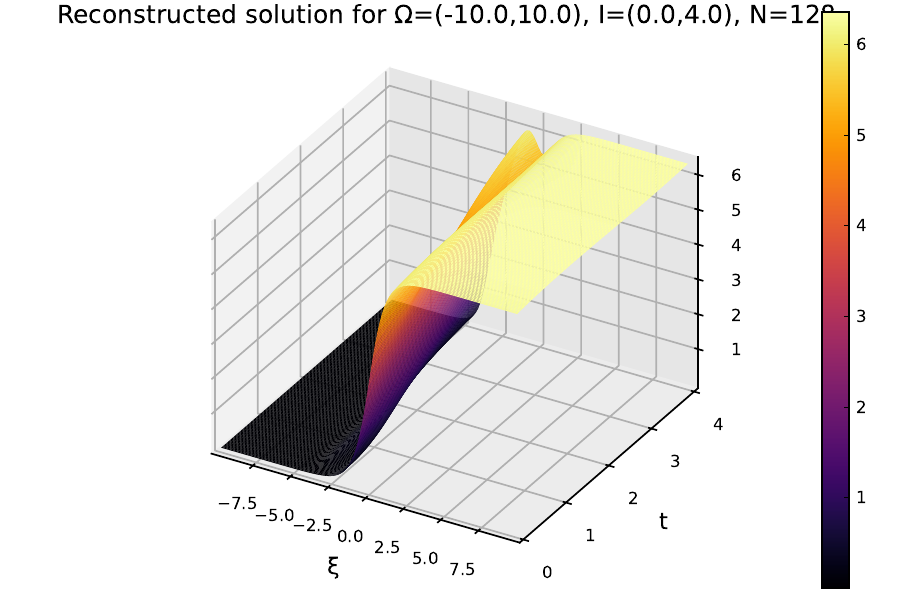}
    \caption{PSD}
  \end{subfigure}
  \caption{Comparison of the reconstruction of the $1$D single soliton solution from the integrated ROM solution in dimension $n=4$ to the full dimension $N=128$ for $\nu=-0.73$ using the autoencoder variants V$2$, V$3$, V$8$ ($50$ epochs) and the PSD method.}
  \label{reconstruction_comparison_single_soliton}
\end{figure}

\begin{figure}[ht!] 
\resizebox{\textwidth}{!}{%
  \centering
  \begin{subfigure}{0.43\textwidth}
    \centering
    \rotatebox{90}{\includegraphics[width=1.2\textwidth, height=.27\textheight]{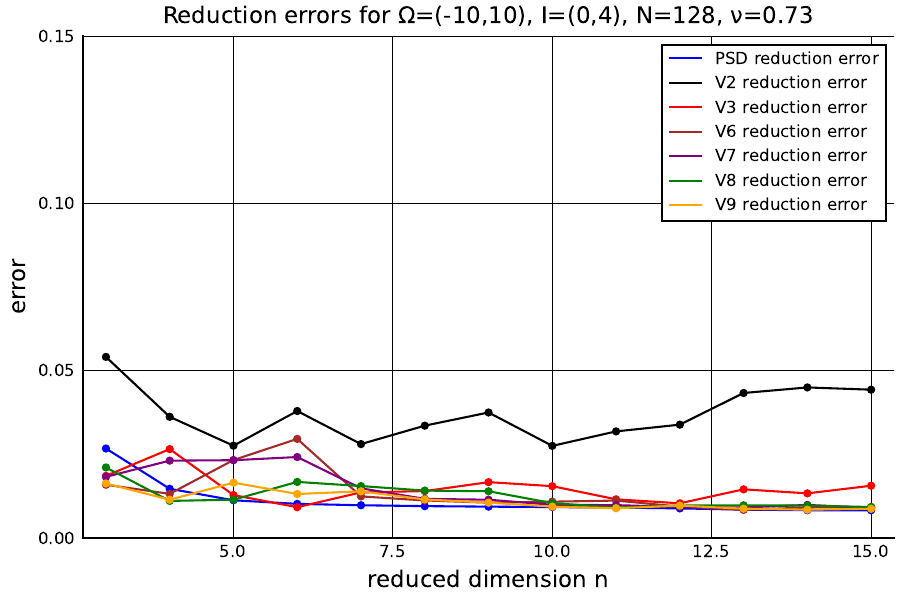}}
    \caption{$\nu=0.73$}
  \end{subfigure} 
  \begin{subfigure}{0.43\textwidth}
    \centering
    \rotatebox{90}{\includegraphics[width=1.2\textwidth, height=.27\textheight]{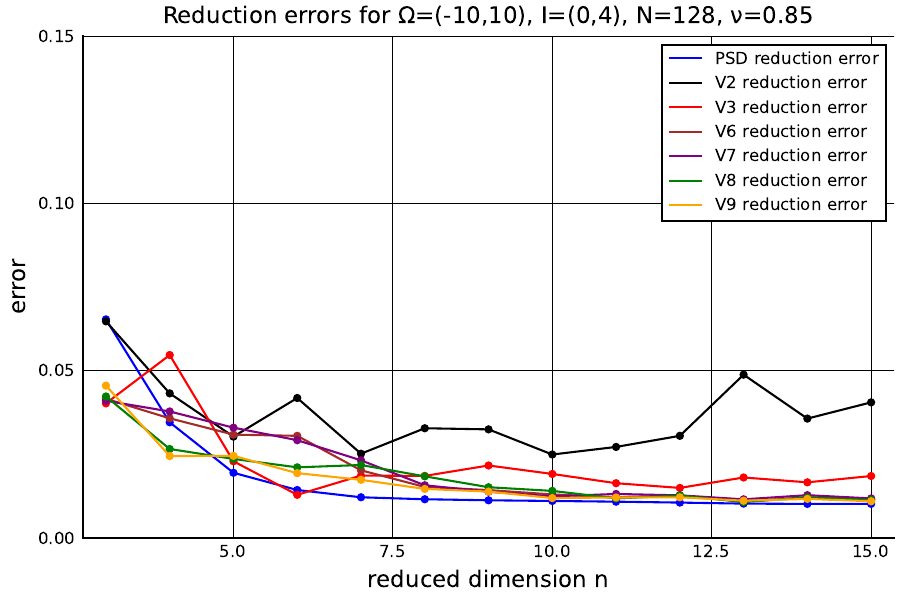}}
    \caption{$\nu=0.85$}
  \end{subfigure}
}  
\resizebox{\textwidth}{!}{%
  \begin{subfigure}{0.43\textwidth}
    \centering
    \rotatebox{90}{\includegraphics[width=1.2\textwidth, height=.27\textheight]{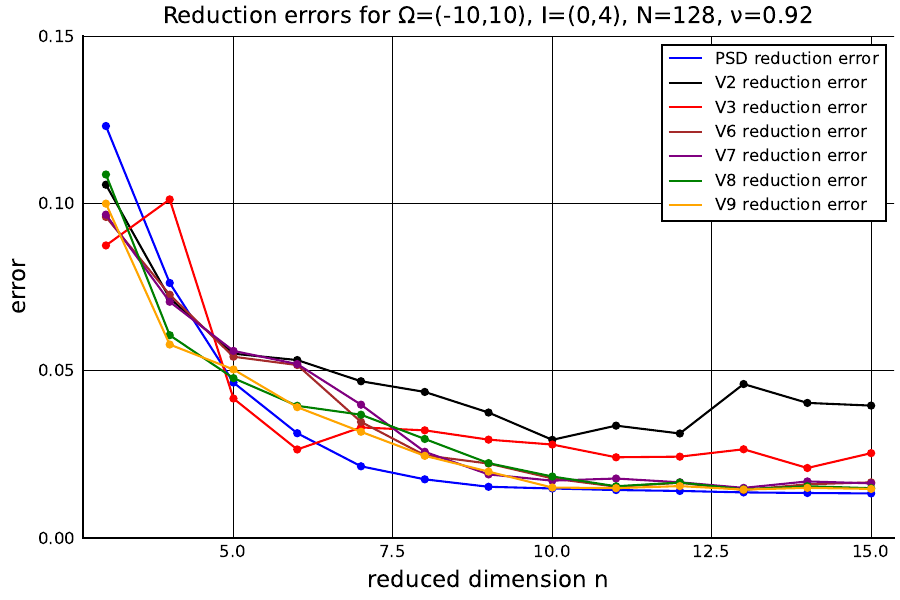}}
    \caption{$\nu=0.92$}
  \end{subfigure}
  \begin{subfigure}{0.43\textwidth}
    \centering
    \rotatebox{90}{\includegraphics[width=1.2\textwidth, height=.27\textheight]{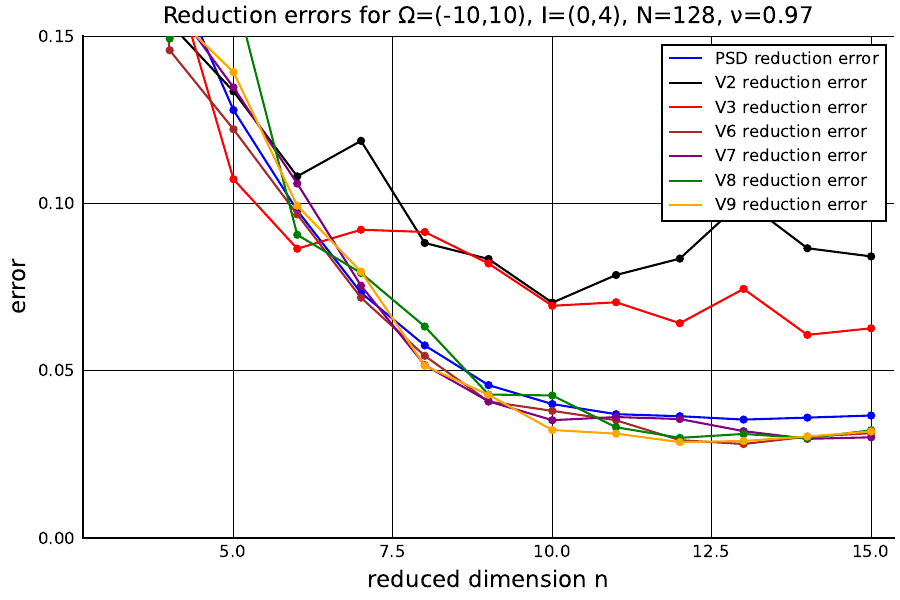}}
    \caption{$\nu=0.97$}
  \end{subfigure}
}
\caption{Reduction errors for the $1$D sine-Gordon equation with soliton-soliton doublets boundary conditions in dimensions $N=128$ and $n=3,4,...,15$ and $50$ epochs test runs.}
  \label{fig:1DsG_red_errors_N=128_50epochs_solitonsolitondoublets}
\end{figure}

\begin{figure}[ht!] 
\resizebox{\textwidth}{!}{%
  \centering
  \begin{subfigure}{0.43\textwidth}
    \centering
    \rotatebox{90}{\includegraphics[width=1.2\textwidth, height=.27\textheight]{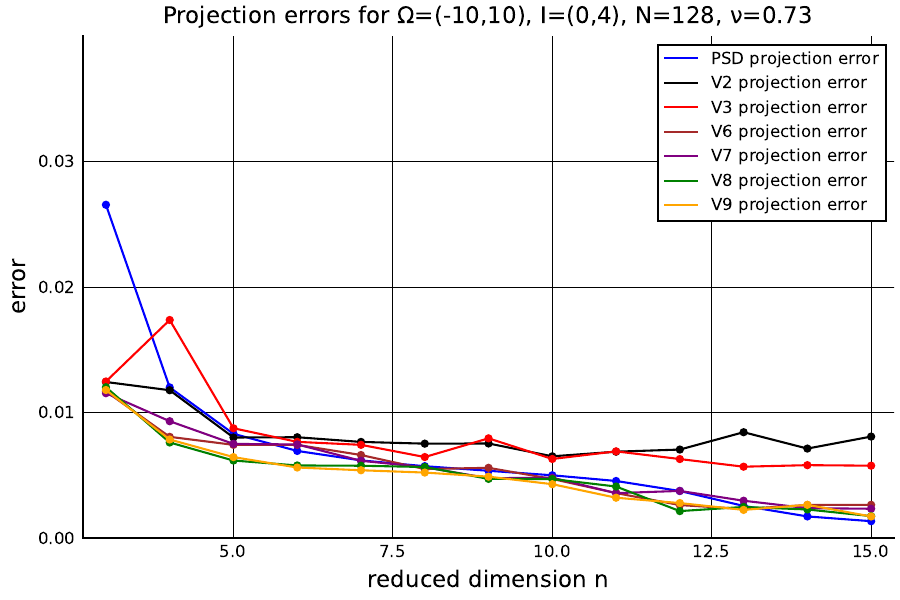}}
    \caption{$\nu=0.73$}
  \end{subfigure} 
  \begin{subfigure}{0.43\textwidth}
    \centering
    \rotatebox{90}{\includegraphics[width=1.2\textwidth, height=.27\textheight]{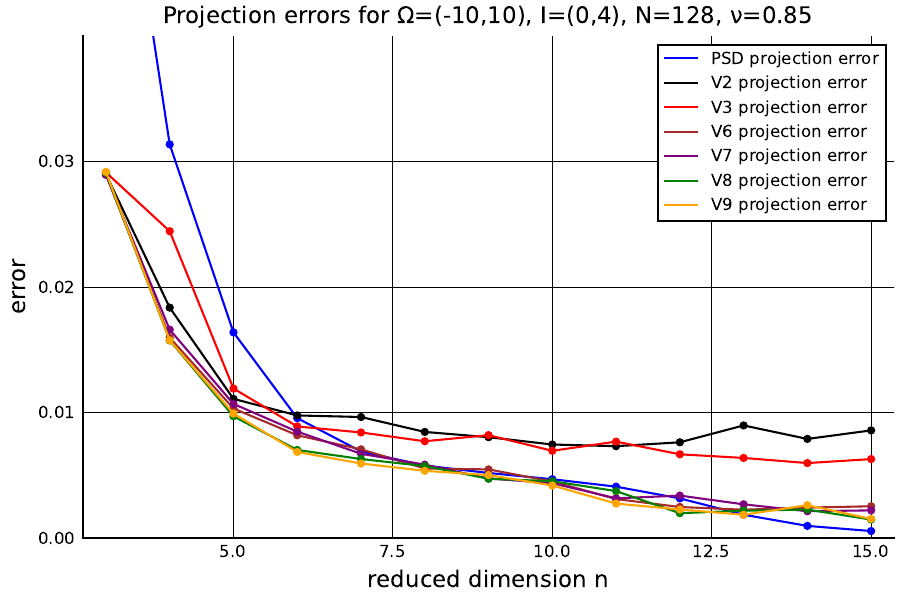}}
    \caption{$\nu=0.85$}
  \end{subfigure}
}  
\resizebox{\textwidth}{!}{%
  \begin{subfigure}{0.43\textwidth}
    \centering
    \rotatebox{90}{\includegraphics[width=1.2\textwidth, height=.27\textheight]{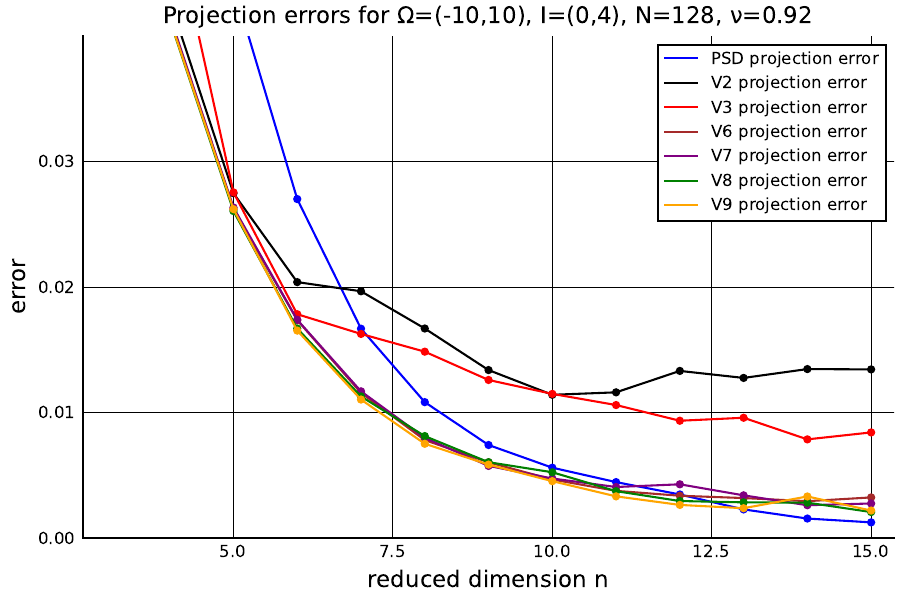}}
    \caption{$\nu=0.92$}
  \end{subfigure}
  \begin{subfigure}{0.43\textwidth}
    \centering
    \rotatebox{90}{\includegraphics[width=1.2\textwidth, height=.27\textheight]{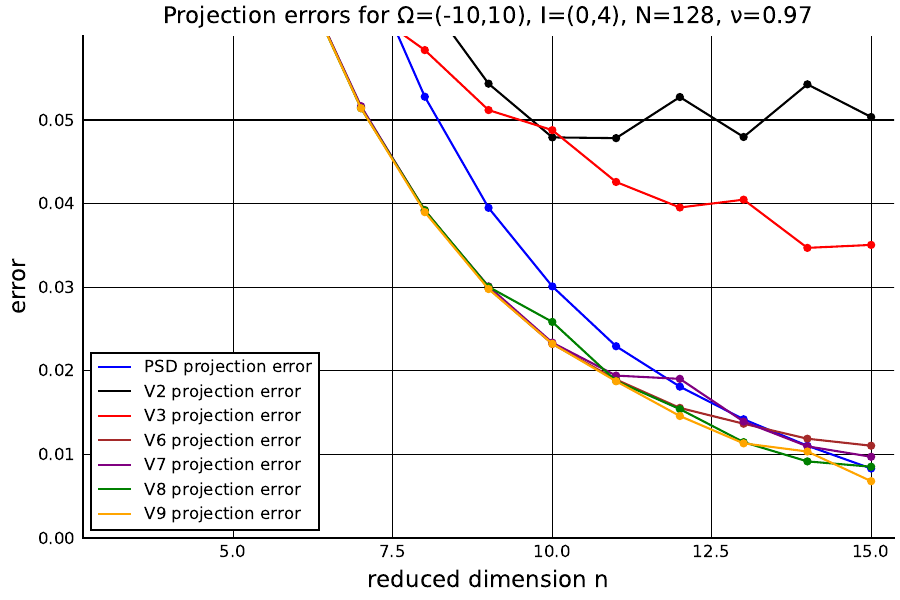}}
    \caption{$\nu=0.97$}
  \end{subfigure}
}
\caption{Projection errors for the $1$D sine-Gordon equation with soliton-soliton doublets boundary conditions in dimensions $N=128$ and $n=3,4,...,15$ and $50$ epochs test runs.}
  \label{fig:1DsG_proj_errors_N=128_50epochs_solitonsolitondoublets}
\end{figure}

\begin{figure}[ht!] 
  \begin{subfigure}{\textwidth}
    \resizebox{\textwidth}{!}{%
      \centering
        \includegraphics[width=0.05\textwidth]{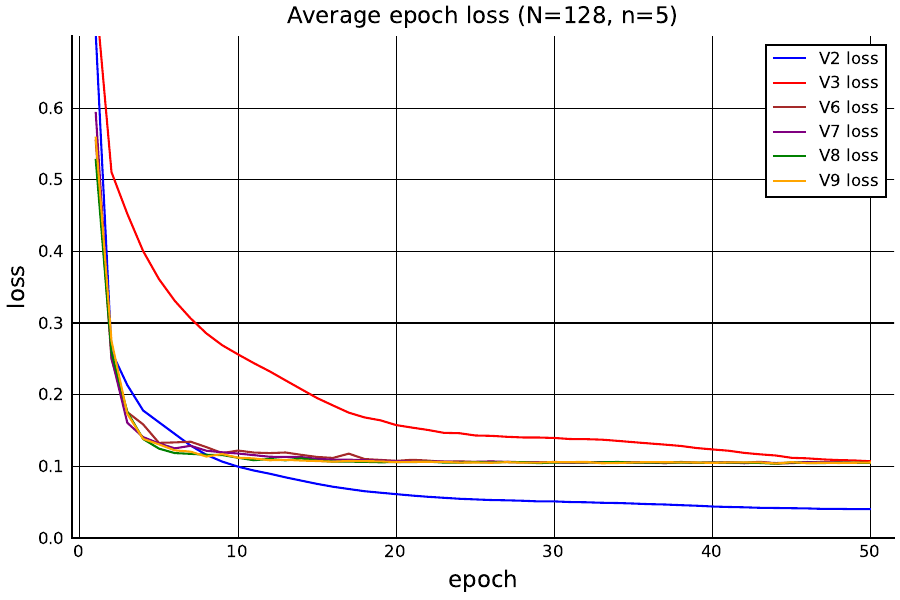}
    }
    \caption{$n=5$}
  \end{subfigure}
  \begin{subfigure}{\textwidth}
    \resizebox{\textwidth}{!}{%
      \centering
        \includegraphics[width=0.05\textwidth]{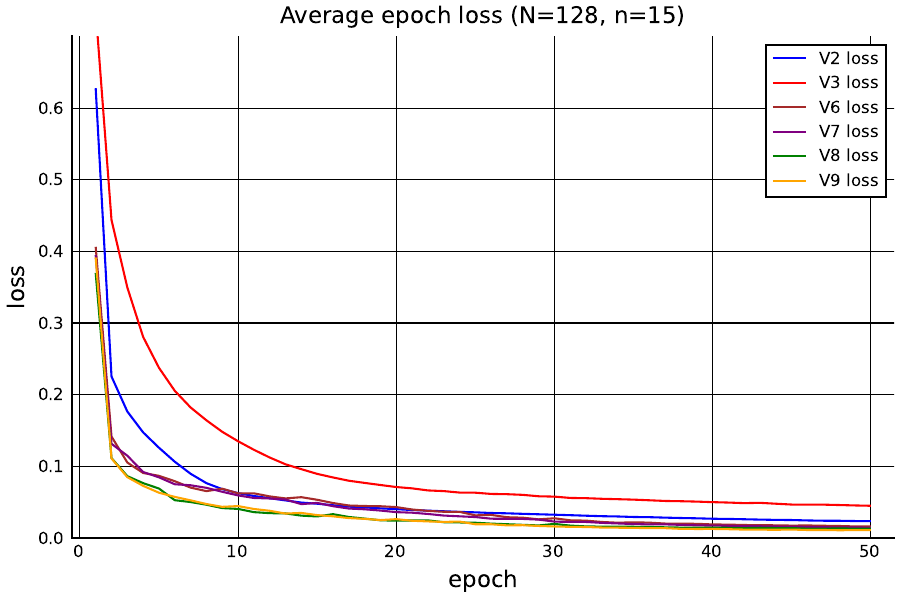}
    }
    \caption{$n=15$}
  \end{subfigure}
\caption{Comparison of $1$D sine-Gordon equation (soliton-soliton doublets) average epoch target function losses for variants V$2$, V$3$ and V$6$-V$9$ in dimension $N=128$ for $50$ epochs test runs.}
  \label{fig:1DsineGordon_epoch_losses_N=128_50epochs_solitonsolitiondoublets}
\end{figure}

\begin{figure}[ht!] 
  \centering
\begin{subfigure}{0.45\linewidth}
    \centering
    \includegraphics[width=\linewidth, scale=.35]{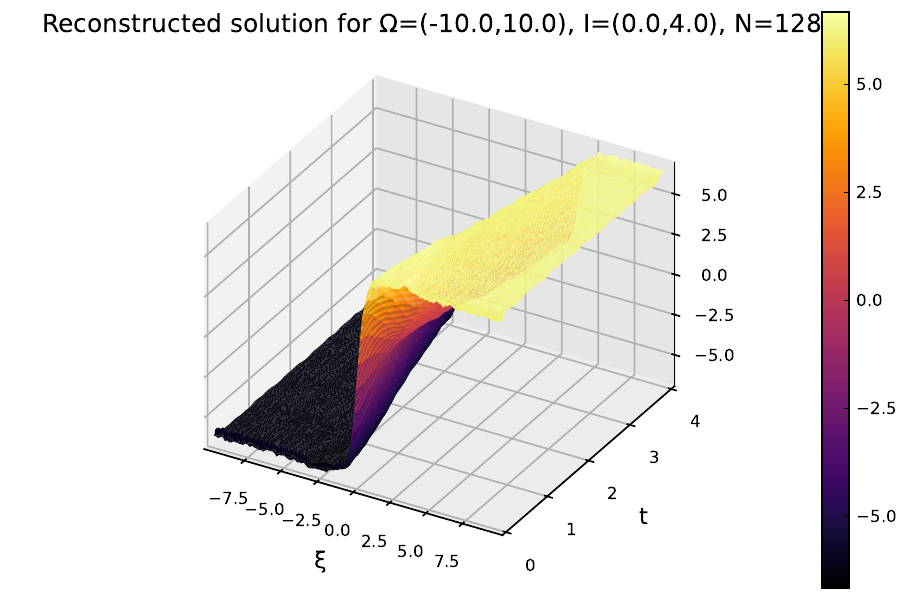}
    \caption{V$2$}
  \end{subfigure}
  \hfill
  \begin{subfigure}{0.45\linewidth}
    \centering
    \includegraphics[width=\linewidth, scale=.35]{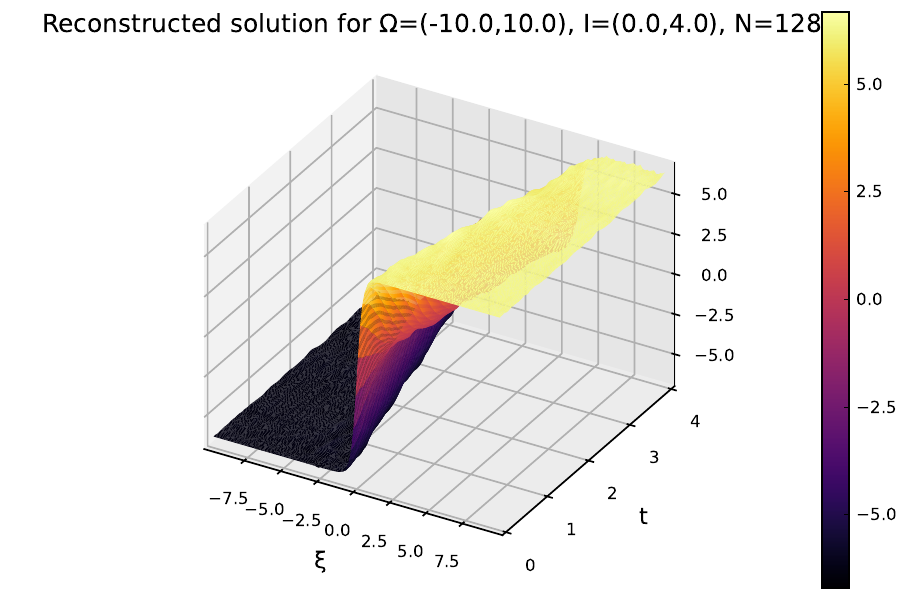}
    \caption{V$3$}
  \end{subfigure}

  \begin{subfigure}{0.45\linewidth}
    \centering
    \includegraphics[width=\linewidth, scale=.35]{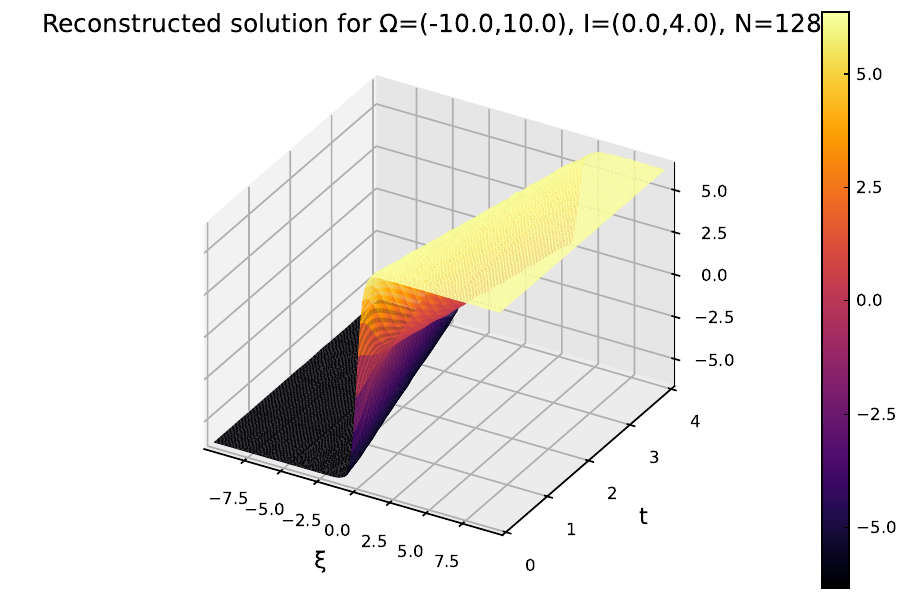}
    \caption{V$8$}
  \end{subfigure}
  \hfill
  \begin{subfigure}{0.45\linewidth}
    \centering
    \includegraphics[width=\linewidth, scale=.35]{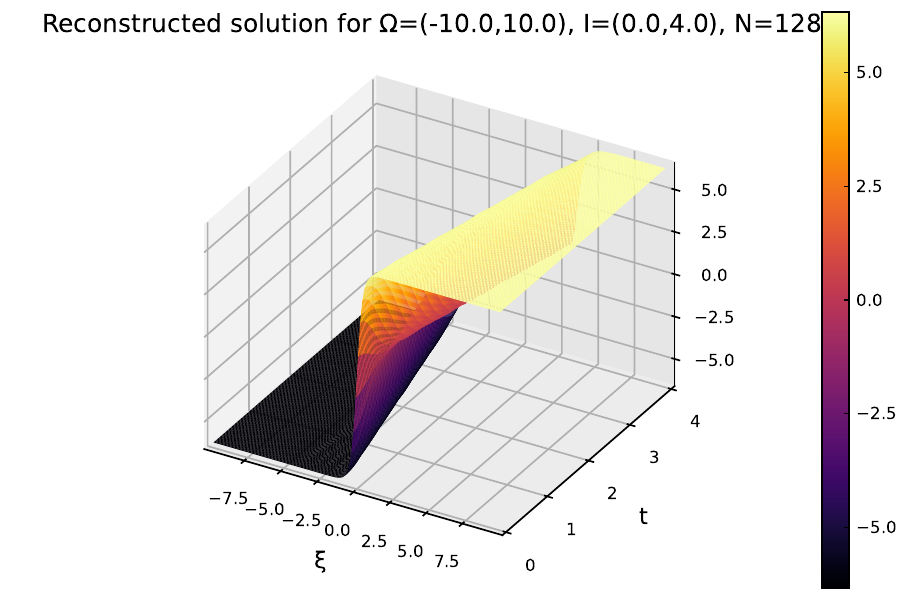}
    \caption{PSD}
  \end{subfigure}
  \caption{Comparison of the reconstruction of the $1$D soliton-soliton doublets solution from the integrated ROM solution in dimension $n=13$ to the full dimension $N=128$ for $\nu=0.97$ using the autoencoder variants V$2$, V$3$, V$8$ ($50$ epochs) and the PSD method.}
  \label{reconstruction_comparison_solitonsolitiondoublets}
\end{figure}

\clearpage

\section*{Declaration of Authorship}

I hereby declare that the thesis submitted is my own unaided work. All direct or indirect
sources used are acknowledged as references.\\
\\
I am aware that the thesis in digital form can be examined for the use of unauthorized aid
and in order to determine whether the thesis as a~whole or parts incorporated in it may be
deemed as plagiarism. For the comparison of my work with existing sources I agree that it
shall be entered in a~database where it shall also remain after examination, to enable
comparison with future theses submitted. Further rights of reproduction and usage, however,
are not granted here.\\
\\
This thesis was not previously presented to another examination board and has not been
published.

\vspace{2cm}

\noindent
Augsburg, \underline{\hspace{4cm}} \hfill \underline{\hspace{4cm}}	\\
\hfill

\end{document}

%% file: preamble/preamble_en_fn_custom.tex
\input{preamble_base_p1_fn_custom.tex}
\input{lang_en_fn_custom.tex}
\input{preamble_base_p2_fn_custom.tex}

\input{fn_custom.tex}

%% file: preamble/preamble_base_p1_fn_custom.tex
\usepackage[utf8]{inputenc}                                                      


\usepackage{csquotes}

\usepackage{xcolor}
\usepackage{colortbl}                                                            

\definecolor{cGray}{rgb}{0.45,0.45,0.45}
\definecolor{cRed}{rgb}{0.545,0.137,0}
\definecolor{cBlue1}{rgb}{0.13,0.21,1}
\definecolor{cBlue2}{rgb}{0,0.408,0.545}
\definecolor{cCyan}{rgb}{0,0.545,0.545}
\definecolor{cOrange1}{rgb}{1,0.666667,0.13726} 
\definecolor{cOrange2}{rgb}{1.0,0.502,0}
\definecolor{cGreen}{rgb}{0.1,0.65,0.1}
\definecolor{cBlack}{rgb}{0.0,0.0,0.0}

\usepackage{float, tabularx, booktabs}                                  
\usepackage{caption, subcaption}                                                 

\usepackage{wrapfig}                                                             

\usepackage{rotating}                                                            

\usepackage{mathtools, amsmath, amssymb, bbm}
\usepackage{stmaryrd}                                                            
\usepackage{amsthm}

\input{math_macros_symb_fn_custom.tex}
\input{math_macros_functions_fn_custom.tex}

\makeatletter
\renewcommand*\env@matrix[1][*\c@MaxMatrixCols c]{
  \hskip -\arraycolsep
  \let\@ifnextchar\new@ifnextchar
  \array{#1}}
\makeatother

\theoremstyle{definition}                                                        


\usepackage{logicproof}                                                         

\usepackage{bussproofs}                                                          
  {\leavevmode\hbox\bgroup}
  {\DisplayProof\egroup}

\usepackage{tikz}                                                                

\usetikzlibrary{calc}                                                            
\usetikzlibrary{positioning}                                                     
\usetikzlibrary{shapes}                                                          
\usetikzlibrary{arrows}                                                          
\usetikzlibrary{decorations.pathmorphing}                                        
\usetikzlibrary{patterns}                                                        
\usetikzlibrary{backgrounds}                                                     

\usetikzlibrary{automata}

\usepackage{graphicx}                                                            
\usepackage{pgfplots}                                                            
\usepgfplotslibrary{fillbetween}                                                 
\pgfplotsset{compat=1.15}

\usepackage[noload]{qtree}
\usepackage{pict2e}

\usepackage{listings}

\ExplSyntaxOn
\tl_new:N \l_listings_boxed_options_tl
\keys_define:nn { listings/boxed }
 {
  caption .tl_set:N = \l_listings_boxed_caption_tl,
  shortcaption .tl_set:N = \l_listings_boxed_shortcaption_tl,
  label .tl_set:N = \l_listings_boxed_label_tl,
  unknown .code:n =
          \tl_put_right:NV \l_listings_boxed_options_tl \l_keys_key_tl
          \tl_put_right:Nn \l_listings_boxed_options_tl { = #1 , },
 }
\box_new:N \l_listings_boxed_box

\lstnewenvironment{blstlisting}[1][]
 {
  \keys_set:nn { listings/boxed } { #1 }
  \exp_args:NV \lstset \l_listings_boxed_options_tl
  \hbox_set:Nw \l_listings_boxed_box
 }
 {
  \hbox_set_end:
  \cs_set_eq:cc {c@figure} {c@lstlisting}
  \tl_set_eq:NN \figurename \lstlistingname
  \tl_if_empty:NF \l_listings_boxed_caption_tl
   {pp
    \tl_if_empty:NTF \l_listings_boxed_shortcaption_tl
     {
      \captionof{figure}{\l_listings_boxed_caption_tl}
     }
     {
      \captionof{figure}[\l_listings_boxed_shortcaption_tl]{\l_listings_boxed_caption_tl}
     }
    \tl_if_empty:NF \l_listings_boxed_label_tl { \label{\l_listings_boxed_label_tl} }
   }
  \leavevmode\box_use:N \l_listings_boxed_box
 }
\ExplSyntaxOff

\newfloat{lstfloat}{htbp}{lop}

\lstset{
  mathescape=true,                                                               
  escapeinside={*@}{@*},                                                         
  literate={æ}{{\ae}}1{ø}{{\oe}}1{å}{{\aa}}1                                     
           {Æ}{{\AE}}1{Ø}{{\O}}1{Å}{{\AA}}1,                                     
}

\lstset{
  stepnumber=1,                                                                  
  numbers=left,                                                                  
  numbersep=5pt,                                                                 
  xleftmargin=\parindent,                                                        
  xrightmargin=\parindent,                                                       
  columns=[c]fixed,                                                              
  captionpos=b                                                                   
}
\DeclareCaptionFormat{listing}{%
  \makebox[2.1cm][l]{\qquad#1#2}%
  \parbox[t]{\dimexpr \textwidth-2.1cm}{#3}%
}
\captionsetup[lstlisting]{format=listing, singlelinecheck=off, labelsep=colon}

\lstset{
  showspaces=false,                                                              
  showstringspaces=false,                                                        
  showtabs=false,                                                                
  breaklines=false,                                                              
  breakatwhitespace=false,                                                       
  tabsize=4                                                                      
}

\lstset{
  numberstyle=\color{cGray},
  basicstyle=\small\upshape\ttfamily,
  commentstyle=\color{cGray},
  keywordstyle=[1]\color{cBlue1},
  stringstyle=\color{cGreen},
}



\usepackage[obeyFinal]{todonotes}

\DeclareDocumentCommand \todocite { g }{
  [
    {\bf \color{cOrange2} \IfValueTF{#1}{#1}{??}}
  ]
}


\usepackage{url}                                                                 
\usepackage{hyperref}

\newcommand*{\http}[1]{\href{http://#1}{#1}}                                     
\newcommand*{\mailto}[1]{\href{mailto:#1}{\nolinkurl{#1}}}                       

\usepackage{pgfkeys}

\newcommand{\setvalue}[1]{\pgfkeys{/variables/#1}}
\newcommand{\getvalue}[1]{\pgfkeysvalueof{/variables/#1}}
\newcommand{\declare}[1]{%
  \pgfkeys{
    /variables/#1.is family,
    /variables/#1.unknown/.style = {\pgfkeyscurrentpath/\pgfkeyscurrentname/.initial = ##1}
  }%
}

\declare{}

%% file: preamble/math_macros_symb_fn_custom.tex
\newcommand{\R}{\ensuremath{\mathbb{R}}}                                         
\newcommand{\N}{\ensuremath{\mathbb{N}}}                                         
\renewcommand{\S}{\ensuremath{\mathbb{S}}}                                       

%% file: preamble/math_macros_functions_fn_custom.tex
\newcommand{\norm}[1]{\left\lVert #1 \right\rVert}                               

 \makeatother

\usepackage{xparse}                                                              
\usepackage{xifthen}                                                             
\usepackage{xstring}                                                             
\usepackage{calc}                                                                

\newcounter{i}

\DeclareDocumentCommand \set { m g g }{                                          
     \left\lbrace                                                                
         #1 \IfValueT {#2} { \ \IfValueTF{#3}{#3}{|}\  #2 }
     \right\rbrace
}

\DeclareDocumentCommand \seq { g g g g } {                                       
    \setcounter{i}{0}                                                            
    \IfValueT {#2} { \addtocounter{i}{#2} }
    \IfValueTF {#1}
        {#1}
        {x}
    _{ \arabic{i} },
    \IfValueTF {#4} 
        {\addtocounter{i}{#4}}
        {\addtocounter{i}{1}}
    \IfValueTF {#1} 
        {#1}
        {x} 
    _{ \arabic{i} },
    \dots
    \IfValueTF {#3}
        { , #1_{#3} }
        {}
}

\DeclareDocumentCommand \eqclass { g g }{                                        
    \left[                                                                       
        \IfValueTF{#1}
            {#1}
            {\dot}
    \right]
    \IfValueTF{#2}
            {_{#2}}
            {}
}

\DeclareDocumentCommand \ero { g g } {                                           
    \begin{array}{c}                                                             
        \IfValueTF{#1}                                                           
            {_{#1}}                                                              
            {\phantom{\sim}}
    \\
        \sim
    \\
        \IfValueTF{#2}
            {^{#2}}
            {\phantom{\sim}}
    \end{array}
}

\DeclareDocumentCommand \matrep { g g g } {                                      
    {_{                                                                          
        \IfValueTF {#1}                                                          
            {#1}                                                                 
            {\epsilon}
    }}
    \left[
        \IfValueTF {#2}
            {#2}
            {\square}
    \right] {_{
        \IfValueTF {#3}
            {#3}
            {\epsilon}
    }}
}

\newcommand{\IndexedFunc}[3]{{#1}_{#2} \left( #3 \right)}

\DeclareDocumentCommand \Geo { g g }{                                            
    \IndexedFunc
        {\text{Geo}}
        {\IfValueTF{#1}
                {#1}
                {L}}
        {\IfValueTF{#2}
            {#2}
            {\lambda}}
}

\DeclareDocumentCommand \Alg { g g }{                                            
    \IndexedFunc
        {\text{Alg}}
        {\IfValueTF{#1}
                {#1}
                {L}}
        {\IfValueTF{#2}
            {#2}
            {\lambda}}
}

\DeclareDocumentCommand \series { g g g g }{                                     
    \set{\IfValueTF{#1}                                                          
      {#1}                                                                       
      {a}
     _{\IfValueTF{#2}
      {#2}
      {n}}      
      }
      _{\IfValueTF{#2}
        {#2}
        {n}
       =
       \IfValueTF{#3}
        {#3}
        {1}
      }
      ^{\IfValueTF{#4}
        {#4}
        {\infty}
      }
}

\DeclareDocumentCommand \pseries { g g g g }{                                    
    \set{\IfValueTF{#1}                                                          
      {#1}                                                                       
      {x}
     ^{\IfValueTF{#2}
      {#2}
      {k}}      
      }
      _{\IfValueTF{#2}
        {#2}
        {k}
       =
       \IfValueTF{#3}
        {#3}
        {1}
      }
      ^{\IfValueTF{#4}
        {#4}
        {\infty}
      }
}

\DeclareDocumentCommand \infseq { g g g g }{                                     
  \sum                                                                           
    _{                                                                           
        \IfValueTF{#2}
          {#2}
          {n}
         =
         \IfValueTF{#3}
          {#3}
          {1}
    }
      ^{\IfValueTF{#4}
        {#4}
        {\infty}
      }
    \IfValueTF{#1}
      {#1}
      {a}
   _{\IfValueTF{#2}
      {#2}
      {n}
  }     
}

\renewcommand{\div}[1]{\text{div}\left( #1 \right)}                              

%% file: preamble/lang_en_fn_custom.tex
\setvalue{lst_name = Code}
\setvalue{lsts_name = List of code}

\setvalue{thm_name     = Theorem}
\setvalue{lem_name     = Lemma}
\setvalue{prop_name    = Proposition}
\setvalue{cor_name     = Corollary}
\setvalue{def_name     = Definition}
\setvalue{conj_name    = Conjecture}
\setvalue{rem_name     = Remark}

\setvalue{obs_name     = Observation}
\setvalue{hyp_name     = Hypothesis}
\setvalue{example_name = Example}

\setvalue{proof_name   = Proof}

\usepackage[english]{babel}

%% file: preamble/preamble_base_p2_fn_custom.tex
\renewcommand{\lstlistingname}{\getvalue{lst_name}}
\floatname{lstfloat}{\getvalue{lst_name}}

\usepackage{amsthm}

 \newtheorem{theorem}{\getvalue{thm_name}}[section]
 \newtheorem{lemma}[theorem]{\getvalue{lem_name}}

 \newtheorem{definition}[theorem]{\getvalue{def_name}}


\usepackage{enumitem}

%% file: preamble/fn_custom.tex

\usepackage[theme=default]{jlcode_fn_custom}
\usepackage{algorithm}
\usepackage{algpseudocode}
\usepackage[a4paper, total={6in, 8in}, top=3.5cm, bottom=3.5cm, left=2.5cm, right=2.5cm]{geometry}
\usepackage{forest}
\usepackage{varwidth}

\setlength{\arrayrulewidth}{0.4mm}
\setlength{\tabcolsep}{7pt}


\def\G{{\mathcal G}}
\def\M{{\mathcal M}}
\def\Nmann{{\mathcal N}}
\def\P{{\mathcal P}}
\def\X{{\mathcal X}}
\def\Y{{\mathcal Y}}
\def\L{{\mathcal L}}
\def\A{{\mathcal A}}
\def\H{{\mathcal H}}
\def\O{{\mathcal O}}
\def\T{{\mathcal T}}
\def\S{{\mathcal S}}

\def\I{{\mathbb I}}

\def\g{{\mathfrak g}}

\def\St{{\mathrm{St}}}
\def\Sp{{\mathrm{Sp}}}
\def\GL{{\mathrm{GL}}}
\def\Orth{{\mathrm{O}}}
\def\cay{{\mathrm{cay}}}
\def\diff{{\mathrm{diff}}}
\def\sub{{\mathrm{sub}}}
\def\skew{{\mathrm{skew}}}
\def\grad{{\mathrm{grad}}}
\def\sym{{\mathrm{sym}}}
\def\tr{{\mathrm{tr}}}
\def\rank{{\mathrm{rank}}}
\def\Exp{{\mathrm{Exp}}}
\def\st{{\mathrm{s.t.}}}
\def\red{{\mathrm{red}}}
\def\reff{{\mathrm{ref}}}
\def\im{{\mathrm{im}}}
\def\psd{{\mathrm{PSD}}}
\def\ca{{\mathrm{cache}}}
\def\lift{{\mathrm{lift}}}
\def\retract{{\mathrm{retract}}}
\def\adam{{\mathrm{Adam}}}
\def\stiefeladam{{\mathrm{StiefelAdam}}}
\def\loss{{\mathrm{loss}}}
\def\red{{\mathrm{red}}}
\def\proj{{\mathrm{proj}}}
\def\train{{\mathrm{train}}}
\def\id{{\mathrm{Id}}}
\def\rand{{\mathrm{rand}}}

\DeclareMathOperator{\diag}{diag}
\DeclareMathOperator{\sech}{sech}

\DeclareMathOperator{\relu}{ReLU}
\DeclareMathOperator{\explu}{ExpLU}

\newcommand{\subsubsubsection}[1]{%
  \paragraph{#1}\mbox{}\\[1ex]}

\newcommand{\stexttt}[1]{{\small\texttt{#1}}}
\newcommand{\scripttexttt}[1]{{\scriptsize\texttt{#1}}}